\documentclass[12pt,english]{amsart}
\usepackage{lmodern}
\usepackage{helvet}

\usepackage[T1]{fontenc}
\usepackage[latin9]{inputenc}
\usepackage{mathrsfs}
\usepackage{amsthm}
\usepackage{amstext}
\usepackage{amssymb}
\usepackage{graphicx}
\usepackage{setspace}
\makeatletter

\newcommand{\noun}[1]{\textsc{#1}}

\numberwithin{equation}{section}
\numberwithin{figure}{section}
  \theoremstyle{plain}
  \newtheorem*{prop*}{\protect\propositionname}
  \theoremstyle{plain}
  \newtheorem*{cor*}{\protect\corollaryname}
 \theoremstyle{definition}
 \newtheorem*{defn*}{\protect\definitionname}
  \theoremstyle{plain}
  \newtheorem*{thm*}{\protect\theoremname}
\theoremstyle{plain}
\newtheorem{thm}{\protect\theoremname}[section]
  \theoremstyle{remark}
  \newtheorem{rem}[thm]{\protect\remarkname}
  \theoremstyle{plain}
  \newtheorem{lem}[thm]{\protect\lemmaname}
  \theoremstyle{plain}
  \newtheorem{cor}[thm]{\protect\corollaryname}
  \theoremstyle{plain}
  \newtheorem{prop}[thm]{\protect\propositionname}
  \theoremstyle{definition}
  \newtheorem{defn}[thm]{\protect\definitionname}
\newenvironment{lyxlist}[1]
{\begin{list}{}
{\settowidth{\labelwidth}{#1}
 \setlength{\leftmargin}{\labelwidth}
 \addtolength{\leftmargin}{\labelsep}
 }}
{\end{list}}

\usepackage{amsfonts}
\usepackage{amssymb}
\usepackage[all,cmtip]{xy}
\usepackage{array}
\usepackage{lmodern}
\usepackage[T1]{fontenc}

\def\QQ{\mathbb{Q}}
\def\RR{\mathbb{R}}
\def\CC{\mathbb{C}}

\def\ZZ{\mathbb{Z}}
\def\PP{\mathbb{P}}

\def\vf{\varphi}
\def\vft{\tilde{\vf}}
\def\lm{\mathfrak{m}}
\def\lt{\mathfrak{t}}
\def\lh{\mathfrak{h}}
\def\l{\mathfrak{l}}
\def\n{\mathfrak{n}}
\def\la{\mathfrak{a}}
\def\ad{\text{ad}}
\def\Ad{\text{Ad}}

\def\lq{\mathfrak{q}}

\def\g{\mathfrak{g}}

\def\k12{\mathcal{K}_{\lambda_1,\lambda_2}}
\def\tk12{\tilde{\mathcal{K}}_{\lambda_1,\lambda_2}}
\def\ck12{\check{\mathcal{K}}_{\lambda_1,\lambda_2}}

\def\b{\bullet}
\def\fb{F^{\b}}
\def\tfb{\tilde{F}^{\b}}
\def\hfb{\hat{F}^{\b}}
\def\Q{\mathcal{Q}}
\def\M{\mathcal{M}}
\def\N{\mathcal{N}}
\def\A{\mathcal{A}}
\def\L{\mathcal{L}}

\theoremstyle{definition}

\theoremstyle{definition}

\theoremstyle{theorem}

\theoremstyle{theorem}

\theoremstyle{theorem}

\theoremstyle{definition}

\newtheorem*{disc}{Discussion}

\makeatother

\usepackage{babel}
  \providecommand{\corollaryname}{Corollary}
  \providecommand{\definitionname}{Definition}
  \providecommand{\lemmaname}{Lemma}
  \providecommand{\propositionname}{Proposition}
  \providecommand{\remarkname}{Remark}
  \providecommand{\theoremname}{Theorem}
\providecommand{\theoremname}{Theorem}

\begin{document}

\title{Naive boundary strata and nilpotent orbits}

\author{Matt Kerr and Gregory Pearlstein }

\begin{abstract}
We study certain real Lie-group orbits in the compact duals of Mumford-Tate
domains, verifying a prediction of \cite{GGK} and determining which
orbits contain a limit point of some period map. A variety of examples
are worked out for the groups $SU(2,1)$, $Sp_{4}$, and $G_{2}$.
\end{abstract}
\maketitle

\section{Introduction}

In a previous work \cite{KP}, we introduced and studied boundary
components for Mumford-Tate domains, which are homogeneous classifying
spaces $D=G(\RR)^{\circ}/\mathcal{H}$ for Hodge structures with additional
symmetries (in a Tannakian sense) \cite{GGK}. Here $G$ is a reductive,
connected $\QQ$-algebraic group, $G(\RR)^{\circ}$ the identity component
of the (Lie) group of real points, and $\mathcal{H}$ a compact subgroup.
The boundary components $B(N)$ essentially parametrize all possible
LMHS (limiting mixed Hodge structures) for period maps $\Phi:\mathcal{S}\to\Gamma\backslash D$
into such a domain with given monodromy logarithm $N$, and also admit
a homogeneous description. A feature of that work was the interesting
representation theory that arises from considering symmetries and
asymptotics of Hodge structures in tandem.

The purpose of the present study is to better understand the interaction
between asymptotic Hodge theory and the $G(\RR)^{\circ}$-orbit structure
of the compact dual $\check{D}=G(\CC)/Q$ of $D$. Here $Q\leq G(\CC)$
is a parabolic subgroup, and $\check{D}$ a projective variety containing
$D$ as an analytic open subset. The ``naive boundary strata'' of
the title are the orbits in the topological boundary of $D$ in $\check{D}$,
and a natural question is which ones are Hodge-theoretically ``accessible''
in the sense of containing a limit point of (a lift $\tilde{\Phi}$
of) some period map $\Phi$. In addition to obtaining a nice answer
to this question ($\S5.2$), we shall clarify the relationship of
these ``boundary orbits'' to the boundary components, and obtain
a mixed-Hodge-theoretic parametrization of all the orbits and description
of their incidence structure.

Now the traditional way to record the asymptotics of a period map
$\Phi$ is via the limiting mixed Hodge structure, and not the ``naive''
limit point of $\tilde{\Phi}$ in $\partial D$. There is a good reason
for this: because of logarithmic growth of periods, the latter loses
information recorded by the LMHS. One has the classic example (cf.
\cite{Ca}) of a degenerating family of genus-2 curves in two parameters,
with two cycles vanishing at the origin. The degenerate fiber is a
rational curve with two pairs of points identified, and the cross
ratio of these 4 points is encoded in the extension data of the LMHS,
while the ``naive'' limit and even the cohomology of the singular
fiber record nothing. From an algebro-geometric point of view, then,
it is unclear why one would want to study the interaction between
variational Hodge theory and the orbit structure of $\check{D}$.

Our motivation for this work arose instead from a perspective heavily
influenced by problems in complex geometry and representation theory,
where the (finitely many) $G(\RR)^{\circ}$-orbits are objects of
some importance \cite{FHW,Wo}. Recent work of Robles \cite{Ro} has
finally settled the question of maximal dimensions of integral manifolds
of the infinitesimal period relation on $\check{D}$ (and hence of
images of period maps in $\Gamma\backslash D$). It seems that one
way of producing interesting maximal-dimensional VHS is by threading
integral manifolds through ``accessible'' orbits in $\partial D$,
and that this approach holds some promise for the much-studied question
of smoothing Schubert varieties in $\check{D}$ in their cohomology
classes . We also mention that in the forthcoming work \cite{GGK2},
the proof of psuedo-convexity of $D$ will be recast in terms of our
Hodge-theoretic analysis of $\partial D$.

On the representation-theoretic side, the $G(\RR)^{\circ}$-orbits
are related to the construction of infinite-dimensional unitary representations
of $G(\RR)$ via parabolic induction. Moreover, one reason for writing
\cite{KP} was to see for which M-T domains one might extend H. Carayol's
approach \cite{Car} to putting an arithmetic structure on automorphic
cohomology. Our analysis of codimension-1 boundary strata suggests
(cf. $\S5.3$) how to generalize his definition of Fourier coefficients
to at least some cases where $G$ is an exceptional group. We shall
pursue these connections in future works.

\subsection*{Summary:}

In the remainder of the Introduction, we briefly describe the main
results. Given a polarized Hodge structure $(V,{}'B,{}'\varphi_{0})$
with Mumford-Tate group ${}'G$, let $\text{Ad}:{}'G\twoheadrightarrow{}'G^{ad}=:G$
be the adjoint map, and set $Z:=\ker(\text{Ad})(\CC)=Z({}'G(\CC))$,
$\vf_{0}:=\Ad\circ{}'\vf_{0}$. Let $\Theta$ be the Cartan involution
of $G$ induced by conjugation by $\vf_{0}(\sqrt{-1})$.

For any $g\in{}'G(\RR)^{\circ}$, if $'\vf_{1}=g^{-1}\circ{}'\vf_{0}\circ g$
satisfies $\Ad\circ{}'\vf_{1}=\vf_{0}$, then ${}'\vf_{1}/{}'\vf_{0}$
is a cocharacter of $Z^{\circ}$, whereupon it must clearly be trivial.
So we have a diagram \[\xymatrix{
{}'\check{D}:= {}'G(\CC)/Q_{F^{\b}_{{}'\vf_0}} \ar @{->>} [d] & {}'G(\RR)^{\circ}/\mathcal{H}_{{}'\phi_0} =: {}'D \ar [d]^{\cong} \ar @{_(->} [l] 
\\
\check{D}:= G(\CC)/Q_{F^{\b}_{\vf_0}} & G(\RR)^{\circ}/\mathcal{H}_{\vf_0} =: D \ar @{_(->} [l]
}\]in which $Z$ belongs to $Q_{F_{{}'\vf_{0}}^{\b}}$ and the left-hand
side is finite-to-one. But the parabolic subgroup $Q_{F_{\vf_{0}}^{\b}}$
is necessarily connected, and so in fact ${}'\check{D}$ and $\check{D}$
are the same. For this reason, we may work without loss of generality
in the adjoint setting. Note that we view the points of the compact
dual $\check{D}$ as flags on $\mathfrak{g}:=Lie(G(\CC))$. We shall
make the simplifying assumptions that the polarization on $(\mathfrak{g},\vf_{0})$
induced by $'B$ is a multiple of the Killing form $B$, and that
the horizontal distribution on $\check{D}$ (equiv. ${}'\check{D}$)
is bracket-generating.

In \cite[(IV.B.10)]{GGK}, it was conjectured that one should be able
to parametrize $\check{D}$ via pairs $(H,\chi)$, where $H\leq G_{\RR}$
is a maximal algebraic torus and $\chi\in\mathrm{X_{*}(H(\CC))}$
a complex co-character of $H$, in $G(\RR)^{\circ}$-equivariant fashion.
This is proved in $\S3$ (Theorem \ref{thm surj}), the take-away
from which summarized in the following
\begin{prop*}
(a) Given any point $F^{\b}\in\check{D}$, there exists a pair $(H,\chi)$
as above such that $F^{\b}=\mathscr{F}(H,\chi)^{\b}$ \emph{(}cf.
\eqref{def F}-\eqref{def F 2}\emph{)}.

(b) The pair $(H,\chi)$ determines a bigrading $\mathfrak{g}_{\CC}=\oplus_{p,q}\g_{\chi}^{p,q}$
with $F^{\b}=\oplus_{p\geq\b}\g_{\chi}^{p,q}$ and $\overline{\g_{\chi}^{p,q}}=\g_{\chi}^{q,p}$.

(c) The $G(\RR)^{\circ}$-orbit containing $F^{\b}$ has real codim.
$\sum_{p,q>0}\dim_{\CC}(\g_{\chi}^{p,q})$ in $\check{D}$.
\end{prop*}
In $\S4$, this is used to index the $G(\RR)^{\circ}$-orbits in $\check{D}$
and describe their incidence relations, beginning with the
\begin{cor*}
Write $F_{\vf_{0}}^{\b}=\mathscr{F}(H_{0},\chi_{0})^{\b}$. Let $\{H_{0},H_{1},\ldots,H_{n}\}$
be representatives of the $G(\RR)^{\circ}$-conjugacy classes of Cartan
subgroups of $G(\RR)^{\circ}$, obtained by Cayley transforms from
$H_{0}$, with complex \emph{{[}}resp. real\emph{{]}} Weyl groups
$W_{\CC}(H_{j})$ \emph{{[}}resp. $W_{\RR}^{\circ}(H_{j})$\emph{{]}}.
Let $\chi_{j}\in\mathrm{X}_{*}(H_{j}(\CC))$ be obtained from $\chi_{0}$
in the same manner, with stabilizers $W_{j}\leq W_{\CC}(H_{j})$.
Then $\mathscr{F}$ induces an ``orbit map'' from the finite set
\[
\Xi:=\bigcup_{j}W_{\RR}^{\circ}(H_{j})\backslash\{(H_{j},w\chi_{j})\,|\, j\in\{0,\ldots,n\},\, w\in W_{\CC}(H_{j})/W_{j}\}
\]
onto the set of $G(\RR)^{\circ}$-orbits. This map is a bijection
if $\check{D}$ is a complete flag variety \emph{(}cf. Lemma \ref{lem complete flag}\emph{)}.
\end{cor*}
Denoting analytic closure by $cl$, the partial order on orbits given
by $\mathcal{O}_{1}\geq\mathcal{O}_{2}\iff cl(\mathcal{O}_{1})\supseteq\mathcal{O}_{2}$
is known as \emph{Bruhat order}, and is generated (at least in the
complete flag setting) by Cayley transforms and cross actions in a
sense made precise in \cite{Ye}. In $\S4.3$ it is briefly explained
how to understand these processes in terms of ``naive'' limits of
Hodge flags and the framework of the Corollary.

In the remainder of the paper, we are interested in those orbits incident
to $D$ itself:
\begin{defn*}
An orbit $\mathcal{O}\subset cl(D)$ is called a \emph{naive boundary
stratum}.
\end{defn*}
Given a polarized variation of Hodge structure $\Phi$ over a punctured
disk with monodromy logarithm $N$, one can take the limit in two
ways. The limiting flag $\tilde{\mathscr{F}}(\Phi)\in\tilde{B}(N)\subseteq\check{D}$
associated to the LMHS is obtained by first twisting by $e^{-\frac{\log(q)}{2\pi i}N}$
and then taking the $q\to0$ limit; the \emph{naive limit} $\hat{\mathscr{F}}(\Phi)\in cl(D)$
is the limit with no twist. They are related by the naive limit map
\[
\mathscr{F}_{lim}^{N}:\tilde{B}(N)\to cl(D)
\]
defined and studied in $\S5.1$. For instance, if a MHS $(\tilde{F}^{\b},W(N)_{\b})\in\tilde{B}(N)$
is $\RR$-split, then $\mathscr{F}_{lim}^{N}(\tilde{F}^{\bullet})$
is the flag obtained by flipping the associated bigrading about the
antidiagonal; moreover, $\mathscr{F}_{lim}^{N}$ factors the projection
$\tilde{B}(N)\to D(N)$ induced by passing to the $\QQ$-splitting
of the LMHS.

As we described above, the instinctive question is how to determine
whether a given naive boundary stratum $\mathcal{O}=G(\RR)^{\circ}.\mathscr{F}(H,\chi)^{\b}$
contains a naive flag, or equivalently some $\hat{B}(N)$. To this
end, we introduce (cf $\S5.2$) the following terminology:
\begin{itemize}
\item $\mathcal{O}$ is \emph{rational} if the filtration $\tilde{W}_{\b}:=\oplus_{-p-q\leq\b}\g_{\chi}^{p,q}$
is $G(\RR)^{\circ}$-conjugate to a $\QQ$-rational one; and
\item $\mathcal{O}$ is \emph{polarizable} if there exists a nonzero element
$\hat{N}\in\g_{\RR}^{-1,-1}$ such that $\hat{N}^{j}:\g_{\chi}^{p,j-p}\overset{\cong}{\to}\g_{\chi}^{p-j,-p}$
($\forall p\in\ZZ,j\in\mathbb{N}$) and a positivity condition holds.
\end{itemize}
Two of our main general results (cf. Theorem \ref{thm equiv}\emph{ff})
may then be stated as follows:
\begin{thm*}
A stratum $\mathcal{O}$ contains a $\hat{B}(N)$ if and only if $\mathcal{O}$
is rational and polarizable. All strata of codimension one are polarizable.
\end{thm*}
In $\S6$, we work out for a variety of examples (including the three
$G_{2}$-domains of \cite{GGK}) the complete incidence diagram and
the associated bigradings for $G(\RR)^{\circ}$-orbits, and determine
which of the boundary strata are polarizable.

\subsubsection*{Acknowledgments:}

This paper has some overlap with recent work of M. Green and P. Griffiths,
and we wish to thank them as well as R. Kulkarni and C. Robles for
helpful conversations and correspondence. The authors acknowledge
partial support from NSF Grant DMS-1068974 (Kerr) and NSF Grant DMS-1002625
(Pearlstein).

\section{Preliminaries}

Let $G$ be a connected $\QQ$-algebraic adjoint group, $H_{\mathbb{C}}\leq G_{\mathbb{C}}$
a maximal algebraic torus subgroup. The groups of complex points $G(\CC)$,
$H(\CC)$ have natural Lie group structures, and we let $\mathfrak{g}$,
$\mathfrak{h}$ denote the complex Lie algebras. From $G$, $\mathfrak{g}$
inherits an underlying $\QQ$-Lie algebra $\mathfrak{g}_{\QQ}$, and
we let $B:\mathfrak{g}_{\QQ}\times\mathfrak{g}_{\QQ}\to\QQ$ denote
the (symmetric, nondegenerate) Killing form $B(X,Y)=Tr(\ad X\circ\ad Y)$.

Consider the lattice
\[
\Lambda^{*}:=\ker\left\{ \exp(2\pi i(\cdot)):\,\mathfrak{h}\to H(\CC)\right\} .
\]
Sending $\phi\in\Lambda^{*}$ to the co-character 
\[
\begin{array}{cccc}
\chi_{\phi}: & \CC^{*} & \to & H(\CC)\\
 & z & \mapsto & e^{\log(z)\phi}
\end{array}
\]
yields an isomorphism
\[
\Lambda^{*}\overset{\cong}{\longrightarrow}\mathrm{X}_{*}(H(\CC)),
\]
with inverse $\chi\mapsto\chi'(1)$. Writing $\Lambda:=Hom(\Lambda^{*},2\ZZ)\subset\mathfrak{h}^{*}$,
we have 
\[
\mathfrak{g}=\mathfrak{h}\oplus\bigoplus_{\alpha\in\Delta}\mathfrak{g}_{\alpha}=\mathfrak{h}\oplus\bigoplus_{\alpha\in\Delta}\CC\langle X_{\alpha}\rangle,
\]
where the roots $\Delta\subset\Lambda$ generate $\Lambda$, and
\[
\ad(h)X_{\alpha}=\alpha(h)X_{\alpha}\;\;\;(\forall h\in\mathfrak{h}).
\]
In particular, for $\chi=\chi_{\phi}\in\mathrm{X}_{*}(H(\CC))$, we
define
\[
\begin{array}{cccc}
\pi_{\chi}: & \Lambda & \to & \ZZ\\
 & \lambda & \mapsto & \frac{1}{2}\lambda(\phi)
\end{array}
\]
so that $\ad(\phi)X_{\alpha}=\pi_{\chi}(\alpha)X_{\alpha}$, and 
\[
\begin{array}{ccc}
\Ad(\chi(z))X_{\alpha} & = & e^{\log(z)\ad\phi}X_{\alpha}\\
 & = & e^{2\log(z)\pi_{\chi}(\alpha)}X_{\alpha}\\
 & = & z^{2\pi_{\chi}(\alpha)}X_{\alpha}.
\end{array}
\]
We shall write for $i>0$ \begin{equation}\label{def F}\mathscr{F}(H,\chi)^{i}\mathfrak{g}_{\CC}:=\bigoplus_{\tiny\left\{ \begin{array}{c}\alpha\in\Delta\\\pi_{\chi}(\alpha)\geq i\end{array}\right.}\mathfrak{g}_{\alpha}
\end{equation}and for $i\leq0$ \begin{equation}\label{def F 2}\mathscr{F}(H,\chi)^{i}\mathfrak{g}_{\CC}:=\mathfrak{h}\oplus\bigoplus_{\tiny\left\{ \begin{array}{c}\alpha\in\Delta\\\pi_{\chi}(\alpha)\geq i\end{array}\right.}\mathfrak{g}_{\alpha};
\end{equation}note that the (partial) flag $\mathscr{F}(H,\chi)^{\bullet}$ depends
only on $H$ and $\pi_{\chi}$.
\begin{rem}
$G(\RR)$ and $G(\CC)$ operate on flags in $\g_{\CC}$ via $\Ad$.
This will often be tacit; that is, $\Ad(g)F^{\bullet}$ will be written
$g.\fb$. This is especially necessary in $\S5$ where the notation
would otherwise become unwieldy.
\end{rem}
Next, we specialize to the case where $H$ is defined over $\RR$.
More precisely, let $\Theta=\Psi_{C}\in Aut(G_{\RR})$ be a Cartan
involution, so that $\theta:=\Ad C\in Aut(\mathfrak{g}_{\RR})$ satisfies
$\theta^{2}=\text{id}$ and $-B(\,\cdot\,,\theta(\,\cdot\,))>0$.
Take $H\leq G_{\RR}$ to be a $\Theta$-stable Cartan subgroup, and
let $\mathfrak{g}_{\RR},\mathfrak{h}_{\RR}$ be the Lie algebras of
$G(\RR),H(\RR)$. We have decompositions into $\pm1$-eigenspaces
\[
\mathfrak{g}_{\RR}=\mathfrak{k}\oplus\mathfrak{p\,,\;\;\;\mathfrak{h}_{\RR}=\mathfrak{t}\oplus\mathfrak{a}=(\mathfrak{h}_{\RR}\cap\mathfrak{k})\oplus(\mathfrak{h}_{\RR}\cap\mathfrak{p})}
\]
of $\theta$, and clearly $-B>0$ {[}resp. $<0${]} on $\mathfrak{k}$
{[}resp. $\mathfrak{p}${]}. A root $\alpha\in\Delta$ is \emph{real}
if $\alpha(\mathfrak{t}_{\CC})=0$, \emph{imaginary} if $\alpha(\mathfrak{a}_{\CC})=0$,
and otherwise \emph{complex}. Indeed, we have $\Lambda^{*}\subset i\mathfrak{t}\oplus\mathfrak{a}$,
and so the action of $\theta$ resp. $\rho:=\,$complex conjugation
on the root vectors $X_{\alpha}\in\mathfrak{g}$ sends $\alpha\mapsto-\bar{\alpha}$
resp. $\bar{\alpha}$. In particular, for $\alpha$ imaginary, we
have
\[
\begin{array}{ccccc}
\theta(\alpha)=\alpha & \implies & \theta X_{\alpha}\in\RR\langle X_{\alpha}\rangle & \implies & \theta X_{\alpha}=X_{\alpha}\;\text{[resp. }-X_{\alpha}\text{]}\\
 &  &  & \implies & X_{\alpha}\in\mathfrak{k}_{\CC}\;\text{[resp. }\mathfrak{p}_{\CC}\text{]}
\end{array}
\]
in which case we say $\alpha$ is compact {[}resp. noncompact{]} imaginary.
So we have a decomposition
\[
\Delta=\Delta_{\RR}\cup\Delta_{\CC}\cup\Delta_{c}\cup\Delta_{n},
\]
with complex roots occurring in quadruplets and other types in $\pm$
pairs. Note that every $G(\RR)^{\circ}$-conjugacy class of Cartans
contains a $\Theta$-stable member. We will write
\[
W_{\RR}(H):=\frac{N(G(\RR),H(\RR))}{H(\RR)}\leq\frac{N(G(\CC),H(\CC))}{H(\CC)}=:W_{\CC}(H)
\]
for the real and complex Weyl groups. The latter is of course generated
by the reflections in the roots $\Delta$. An algorithm for computing
the real Weyl group (as implemented by the ATLAS computer software)
is described in \cite[sec. 6]{Ad} and \cite{Tr}. More germane for
our purposes is the connected Weyl group
\[
W_{\RR}^{\circ}(H):=\frac{N\left(G(\RR)^{\circ},H(\RR)^{\circ}\right)}{H(\RR)^{\circ}}\leq W_{\RR}(H),
\]
which contains the subgroup generated by reflections in $\Delta_{\RR}\cup\Delta_{c}$
but may be larger than this unless $H(\RR)$ is compact or split.

Now assume that there exists a maximal torus $T\leq G_{\RR}$ with
$T(\RR)$ compact. Taking $H:=T$, all roots are imaginary (in the
sense that $\rho:\alpha\mapsto-\alpha$), and we \emph{define} the
compact {[}resp. noncompact{]} ones to be those with $-B(X_{\alpha},\overline{X_{\alpha}})>0$
{[}resp. $<0${]}. (We may normalize so that $\overline{X_{\alpha}}=X_{-\alpha}$
resp. $-X_{-\alpha}$ for $\alpha\in\Delta_{n}$ resp. $\Delta_{c}$.)
Setting 
\[
\mathfrak{k}:=\mathfrak{t}\oplus\left(\bigoplus_{\alpha\in\Delta_{c}}\CC\langle X_{\alpha}\rangle\right)\cap\mathfrak{g}_{\RR}\,,\;\;\mathfrak{p}:=\left(\bigoplus_{\alpha\in\Delta_{n}}\CC\langle X_{\alpha}\rangle\right)\cap\mathfrak{g}_{\RR},
\]
$K:=\exp(\mathfrak{k})$ is maximal compact, and the involution $\theta$
defined by linearly extending $\theta|_{\mathfrak{k}}:=\text{id}$,
$\theta|_{\mathfrak{p}}:=-\text{id}$ is Cartan. In particular, $T$
is $\Theta$-stable, and $\Delta=\Delta_{c}\cup\Delta_{n}$.

Assume further that there exists a co-character $\chi_{0}\in\mathrm{X}_{*}(T(\CC))$
such that $\pi_{\chi_{0}}(\Delta_{c})\subset2\ZZ$ and $\pi_{\chi_{0}}(\Delta_{n})\subset1+2\ZZ$.
Let $\varphi_{0}$ denote the restriction of 
\[
\CC^{*}\overset{\chi_{0}}{\longrightarrow}T(\CC)\hookrightarrow G(\CC)
\]
to $S^{1}\to G(\RR)$, and $\Ad:G(\RR)\to Aut(\mathfrak{g}_{\RR},B)$
the adjoint homomorphism. Then $(\mathfrak{g}_{\QQ},\Ad\circ\varphi_{0},-B)$
is a polarized Hodge structure of weight $0$, with decomposition
$\mathfrak{g}_{\CC}=\oplus_{j\in\ZZ}\mathfrak{g}^{j,-j}$, where
\[
\mathfrak{g}^{j,-j}:=\left\{ \gamma\in\mathfrak{g}_{\CC}\,|\,\Ad(\chi_{0}(z))\gamma=z^{2j}\gamma\right\} =\bigoplus_{\tiny\left\{ \begin{array}{c}
\alpha\in\Delta,\\
\pi_{\chi_{0}}(\alpha)=j
\end{array}\right.}\CC\langle X_{\alpha}\rangle
\]
if $j\neq0$ and $\mathfrak{g}^{0,0}:=\mathfrak{t}_{\CC}\oplus\bigoplus_{\alpha\in\left(\ker\pi_{\chi_{0}}\right)\cap\Delta}\CC\langle X_{\alpha}\rangle.$
(Note that $C=\varphi(i)$.) The conjugacy class of $\varphi_{0}$
(or ``connected Hodge domain'')
\[
D:=G(\RR)^{\circ}.\varphi_{0}\cong G(\RR)^{\circ}/\mathcal{H}_{\varphi_{0}}
\]
parametrizes a set of $(-B)$-polarized Hodge structures on $\mathfrak{g}_{\QQ}$
with the same Hodge numbers; a very general point in $D$ has Mumford-Tate
group $G$. Writing 
\[
F_{0}^{k}\mathfrak{g}_{\CC}:=\bigoplus_{j\geq k}\mathfrak{g}^{j,-j}\,\left(=\mathscr{F}(T,\chi_{0})^{k}\mathfrak{g}_{\CC}\right),
\]
$D$ is an analytic open subset in its compact dual
\[
\check{D}:=G(\CC).F_{0}^{\bullet}\cong G(\CC)/Q_{F_{0}^{\bullet}}.
\]
Flags%
\footnote{In this paper, ``flag'' will mean what is sometimes called a ``partial
flag'', i.e. not necessarily ``maximal'' or ``complete''.%
} $F^{\bullet}\in\check{D}$ are called \emph{semi-Hodge}; a \emph{Hodge}
flag is one which satisfies $\mathfrak{g}_{\CC}=\oplus_{p\in\ZZ}F^{p}\cap\overline{F^{-p}}$.
Equivalently, $F^{\bullet}$ is of the form $\mathscr{F}(T,\chi)^{\bullet}$
for some compact $T\leq G_{\RR}$. As above $\chi$ has an associated
weight $0$ (but not necessarily $(-B)$-polarized) Hodge structure
$\varphi:\, S^{1}\to G(\RR)$, and we write $F^{\bullet}=F_{\varphi}^{\bullet}$.
We denote the non-Hodge locus in $\check{D}$ by $\mathfrak{Z}$,
and shall (by abuse of notation) write $\varphi\in\check{D}\backslash\mathfrak{Z}$.

Finally, taking $\Delta^{+}$ to be a system of positive roots with
$\pi_{\chi}(\Delta^{+})\subset\mathbb{Z}_{\geq0}$, assume that $\pi_{\chi}$
takes values $0$ and $1$ on the simple roots. Then it is known (cf.
\cite{Ro}) that the $G(\CC)$-invariant horizontal distribution $\mathcal{W}\subset T\check{D}$
given (at $\varphi$) by
\[
F^{-1}\mathfrak{g}_{\CC}/F^{0}\mathfrak{g_{\CC}}\subset\mathfrak{g}_{\CC}/F^{0}\mathfrak{g_{\CC}}
\]
is bracket-generating; that is, 
\[
\mathcal{W}+[\mathcal{W},\mathcal{W}]+[\mathcal{W},[\mathcal{W},\mathcal{W}]]+\cdots=T\check{D}
\]
(or equivalently, $F^{-1}+[F^{-1},F^{-1}]+\cdots=\mathfrak{g}_{\CC}$
on the Lie algebra level).

The three assumptions delineated in the last three paragraphs will
remain in effect for the rest of this paper. With $T$, $\theta$
as above, we can obtain $\theta$-stable representatives of all $G(\RR)^{\circ}$-conjugacy
classes of (real) Cartans by taking successive \emph{Cayley transforms}
in imaginary noncompact roots. Given $H\leq G_{\RR}$ and $\alpha\in\Delta_{n}$,
the Cayley transform in $\alpha$ is defined in terms of conjugation
by $\mathbf{c}_{\alpha}:=\exp\left(\frac{\pi}{4}\left(X_{-\alpha}-X_{\alpha}\right)\right)$,
where the root vectors $X_{\pm\alpha}$ are assumed to be normalized
so that $X_{-\alpha}=\overline{X_{\alpha}}$ and $[[X_{\alpha},X_{-\alpha}],X_{\alpha}]=2X_{\alpha}$
(in particular, $[X_{\alpha},X_{-\alpha}]\in\Lambda^{*}$). More precisely,
it replaces $H$ by the real algebraic torus underlying $\Psi_{\mathbf{c}_{\alpha}}(H_{\CC})$,
which by abuse of notation shall be denoted $\Psi_{\mathbf{c}_{\alpha}}(H)$.
This has the effect of increasing the real rank $\dim_{\RR}\mathfrak{a}$
of $H$ by $1$, replacing $\mathfrak{h}_{\RR}$ by $\left(\ker\alpha|_{\mathfrak{h_{\RR}}}\right)\oplus\RR\langle X_{\alpha}+X_{-\alpha}\rangle.$
(Conjugation by the square $\mathbf{c}_{\alpha}^{2}$, which stabilizes
$H$, yields the Weyl reflection in $\alpha$.) Up to scaling, the
new root vectors are the images of the old ones by $\Ad(\mathbf{c}_{\alpha})$,
and $\Ad(\mathbf{c}_{\alpha})X_{\pm\alpha}$ in particular are real
root vectors. The process may be reversed by applying Cayley transforms
$\Ad(\mathbf{d}_{\beta})$ in $\beta\in\Delta_{\RR}$, cf. \cite[sec. VI.7]{Kn}.
\begin{rem}
One may pictorially represent the situation in a graph with $G(\RR)^{\circ}$-conjugacy
classes of real Cartans as nodes and Cayley transforms as edges; the
ATLAS software can compute these so-called Hasse diagrams (cf. \cite{Ad},
\cite{AdC}).
\end{rem}

\section{From semi-Hodge flags to Cartan data}

Take $G$, $H_{0}=T_{0}\leq G_{\mathbb{R}}$, $\chi_{0}\in\mathrm{X}_{*}(T(\CC))$,
and $\vf_{0}\in D\subset\check{D}$ to be as in $\S2$, with associated
flag $F_{0}^{\bullet}$. The points of $D$ are the $(-B)$-polarized
Hodge structures $G(\mathbb{R})^{\circ}$-conjugate to $\vf_{0}$.
In this section we will give a similar characterization of the points
of $\check{D}$.

\subsection{The bigrading}

Let $F^{\bullet}\in\check{D}$, with $Q_{F^{\bullet}}\subset G(\mathbb{C})$
the parabolic subgroup preserving $F^{\bullet}$. Inside $\check{D}\cong G(\CC)/Q_{F^{\bullet}}$
we have the connected $G(\RR)^{\circ}$-orbit 
\[
\mathcal{O}_{F^{\bullet}}:=G(\RR)^{\circ}.F^{\bullet}\cong G(\RR)^{\circ}/\{Q_{F^{\bullet}}\cap\overline{Q_{F^{\bullet}}}\cap G(\RR)^{\circ}\}.
\]

\begin{lem}
\label{bruhat lemma} There exists a Cartan subgroup $H\leq G_{\RR}$
with 
\[
H(\RR)\subseteq Q_{F^{\bullet}}\cap\overline{Q_{F^{\bullet}}}\cap G(\RR).
\]
\end{lem}
\begin{proof}
See \cite[Thm. 2.6(1)]{Wo} or \cite[Lemma 2.1.2]{FHW}.
\end{proof}
Fix $H_{0}$, and write $\mathfrak{h}_{\RR}:=Lie(H(\RR))$, $\mathfrak{h}:=Lie(H(\CC))$
for the corresponding Lie subalgebras. From $F^{\bullet}$ we obtain
a filtration \begin{equation}\label{eqn *sharp}\tilde{W}_{-k}\mathfrak{g}_{\CC}:=\bigoplus_{p\in\ZZ}\left( F^p\cap \overline{F^{k-p}}\right)\mathfrak{g}_{\CC}
\end{equation}which is stabilized by $Q_{F^{\bullet}}\cap\overline{Q_{F^{\bullet}}}$
hence by $H(\CC)$; this is of course defined over $\RR$. Its nontriviality
``measures'' the failure of $F^{\bullet}$ to be a (pure) Hodge
flag.
\begin{lem}
\label{bigrading lemma}There is a unique bigrading\begin{equation}\label{eqn !!}\mathfrak{g}_{\CC}=\bigoplus_{(p,q)\in\ZZ}\mathfrak{g}^{p,q}
\end{equation}satisfying:\vspace{2mm}\\
(i) each $\mathfrak{g}^{p,q}$ is a sum of root spaces of $(\mathfrak{g},\mathfrak{h})$;\\
(ii) $F^{a}=\oplus_{\tiny\left\{ \begin{array}{c}
(p,q)\\
p\geq a
\end{array}\right.}\mathfrak{g}^{p,q}$;\\
(iii) $\overline{\mathfrak{g}^{p,q}}=\mathfrak{g}^{q,p}$;\\
(iv) $\tilde{W}_{-k}=\oplus_{\tiny\left\{ \begin{array}{c}
(p,q)\\
p+q\geq k
\end{array}\right.}\mathfrak{g}^{p,q}$;\\
(v) $B|_{\mathfrak{g}^{p,q}\times\mathfrak{g}^{p',q'}}$ nondegenerate
for $(p',q')=(-p,-q)$, and otherwise $0$.\vspace{2mm}\\
The $\{\dim(\mathfrak{g}^{p,q})\}$ do not depend upon the choice
of $H$.\end{lem}
\begin{proof}
We know that $F_{0}^{\bullet}:=F_{\vf_{0}}^{\bullet}=\mathscr{F}(H_{0},\chi_{0})^{\bullet}$,
and that $G(\CC)$ acts transitively on $\check{D}$. Taking $g\in G(\CC)$
so that $g.F_{0}^{\bullet}=F^{\bullet}$, we have $\mathscr{F}\left(\Psi_{g}(H_{0,\CC}),\Psi_{g}(\chi_{0})\right)^{\bullet}=F^{\bullet},$
and $\Psi_{g}(H_{0,\CC})\leq\Psi_{g}Q_{F_{0}^{\bullet}}=Q_{F^{\bullet}}.$
Since $Q_{F^{\bullet}}$ is connected, any two Cartans of $Q_{F^{\bullet}}$
(i.e. Cartans of $G(\CC)$ contained in $Q_{F^{\bullet}}$) are conjugate
by an element of $Q_{F^{\bullet}}$ (cf. \cite[11.16 and 12.1(a)]{Bo}).
Hence, we may arrange to have $\Psi_{g}(H_{0,\CC})=H_{\CC}$; write
$\Psi_{g}(H_{0})=H$ and $\Psi_{g}(\chi_{0})=:\chi_{\xi}=:\chi.$

Since $H$ is real, we have $\chi,\bar{\chi}\in\mathrm{X}_{*}(H(\CC)),$
and so 
\[
\mathfrak{g}^{p,q}:=\left\{ \gamma\in\mathfrak{g}_{\CC}\left|\begin{array}{c}
Ad(\chi(z))\gamma=z^{2p}\gamma\\
Ad(\bar{\chi}(z))=z^{2q}\gamma
\end{array}\right.\right\} 
\]
\[
=\oplus_{\alpha\in\Delta,\Tiny\left\{ \begin{array}{c}
\pi_{\chi}(\alpha)=p\\
\pi_{\bar{\chi}}(\alpha)=q
\end{array}\right.}\CC\langle X_{\alpha}\rangle
\]
gives a bigrading. Since \begin{equation}\label{eqn *!}\left\{ \begin{array}{c}F^{a}=\mathscr{F}(H,\chi)=\bigoplus_{\Tiny\begin{array}{c}(p,q)\\p\geq a\end{array}}\mathfrak{g}^{p,q}\\\overline{F^{b}}=\mathscr{F}(H,\bar{\chi})=\bigoplus_{\Tiny\begin{array}{c}(p,q)\\q\geq b\end{array}}\mathfrak{g}^{p,q},\end{array}\right.
\end{equation} this gives \emph{(i)}-\emph{(iv)}, and \emph{(v)} follows from the
fact that $X_{\alpha}\in\mathfrak{g}^{p,q}\iff X_{-\alpha}\in\mathfrak{g}^{-p,-q}$.
Uniqueness is clear, as is the last statement i since $\dim\mathfrak{g}^{p,q}=\dim\left(\{F^{p}\cap\overline{F^{q}}\}/\{F^{p}\cap\overline{F^{q+1}}+F^{p+1}\cap\overline{F^{q}}\}\right).$
\end{proof}
There are several easy remarks at this point. The first is that $[\mathfrak{g}^{p,q},\mathfrak{g}^{p',q'}]\subseteq\mathfrak{g}^{p+p',q+q'}$
and so the isotropy Lie algebra $\mathfrak{q}_{F^{\bullet}}:=Lie(Q_{F^{\bullet}})$
identifies with $F^{0}\mathfrak{g}_{\CC}$. Next, setting 
\[
W_{k}\mathfrak{g}_{\CC}:=\bigoplus_{p+q\leq k}\mathfrak{g}^{p,q}
\]
(which is clearly defined over $\RR$),%
\footnote{To a Hodge theorist, this filtration is much more familiar than $\tilde{W}_{\bullet}$,
but (unlike $\tilde{W}_{\bullet}$) depends on the choice of $H$.%
} we see that $(\mathfrak{g}_{\RR},F^{\bullet},W_{\bullet})$ give
the data of a \emph{split $\RR$-mixed Hodge structure}, which is
split by property \emph{(iii)} in the Lemma. From the end of the proof,
it satisfies the symmetry $\dim Gr_{j}^{W}\mathfrak{g}=\dim Gr_{-j}^{W}\mathfrak{g}.$
Defining new cocharacters $\chi_{Y},\chi_{\phi}$ by \begin{equation}\label{eqn !sharp}Y:=\frac{\xi+\bar{\xi}}{2},\;\;\;\phi:=\frac{\xi-\bar{\xi}}{2},
\end{equation} the mixed Hodge representation associated to $(\mathfrak{g}_{\RR},F^{\bullet},W_{\bullet})$
is
\[
\begin{array}{cccccc}
\tilde{\vf}_{F^{\bullet}}: & \RR^{*}\times S^{1} & \to & H(\RR) & \hookrightarrow & G(\RR)\\
 & (w,z) & \mapsto & \chi_{Y}(w)\chi_{\phi}(z) & =: & \tilde{\vf}_{F^{\bullet}}(w,z).
\end{array}
\]
The composition of its complexification with $Ad$, mapping $\CC^{*}\times\CC^{*}\to Aut(\mathfrak{g},B)$,
restricts on $\mathfrak{g}^{p,q}$ to multiplication by $w^{p+q}z^{p-q}$.
Since we can act with $G(\RR)^{\circ}$ (via $Ad$) compatibly on
$F^{\bullet}$,$W_{\bullet}$,$H$,$\phi$,$Y$, and $\{\mathfrak{g}^{p,q}\}$,
we have
\begin{cor}
The
\[
h_{\mathcal{O}}^{p,q}:=\dim_{\CC}(\mathfrak{g}^{p,q})
\]
are well-defined invariants of of the $G(\RR)^{\circ}$-orbit $\mathcal{O}$.
\end{cor}
Finally, there is the
\begin{prop}
$F^{\bullet}$ is Hodge $\iff$ $\tilde{W}_{\bullet}$ is trivial
($\iff$ $W_{\bullet}$ is trivial).\end{prop}
\begin{proof}
$Y$ grades $\tilde{W}_{\bullet}$, so if $Gr_{j}^{\tilde{W}}=\{0\}$
for $j\neq0$ then $Y=0$ and $\vf:=\tilde{\vf}_{F^{\bullet}}|_{S^{1}}$
is a Hodge structure with $F^{\bullet}=F_{\vf}^{\bullet}$. Conversely,
if $F^{\bullet}$ is Hodge then the $p$-opposed condition $F^{p}\cap\overline{F^{-p+1}}=\{0\}$
holds, and so $\tilde{W}_{-1}=\{0\}$; by symmetry, $\mathfrak{g}/\tilde{W}_{0}=\{0\}$.
\end{proof}

\subsection{From Cartan data to semi-Hodge flags}

We are now prepared to parametrize the flags in $\check{D}$ by Cartan
data. Let $\xi_{\CC}$ denote the set of maximal tori ${}'H\leq G_{\CC}$,
and $\Xi_{\CC}$ the set of pairs $({}'H,\chi)$ ($\chi\in\mathrm{X}_{*}({}'H(\CC))$)
$G(\CC)$-conjugate to $(H_{0,\CC},\chi_{0})$. Define subsets $\tilde{\xi}_{\RR}\subset\xi_{\CC}$
resp. $\tilde{\Xi}_{\RR}\subset\Xi_{\CC}$ by imposing the requirement
that ${}'H=H_{\CC}$ for $H$ defined over $\RR$, and smaller subsets
$\xi_{\RR}$ resp. $\Xi_{\RR}$ by insisting that $H(\RR)$ be compact.
Finally let $\Xi_{\RR}^{\circ}\subset\Xi_{\RR}$ be the $G(\RR)^{\circ}$-orbit
of $(H_{0},\chi_{0})$. Applying $\mathscr{F}$ produces a $G(\RR)^{\circ}$-equivariant%
\footnote{$G(\RR)$-equivariant if the left-most column is removed; right-most
column $G(\CC)$-equivariant.%
} commutative diagram \[\xymatrix{
\xi_{\RR} \ar@{=} [r] & \xi_{\RR} \ar@{^(->} [r] & \tilde{\xi}_{\RR} \ar@{^(->} [r] & \xi_{\CC} 
\\
\Xi^{\circ}_{\RR} \ar@{->>} [u]_{\pi_{\RR}^{\circ}} \ar@{->>} [d]^{\mathscr{F}^{\circ}_{\RR}} \ar@{^(->} [r] & \Xi_{\RR} \ar@{->>} [u]_{\pi_{\RR}} \ar [d]^{\mathscr{F}_{\RR}} \ar@{^(->} [r] & \tilde{\Xi}_{\RR} \ar@{->>} [u]_{\tilde{\pi}_{\RR}} \ar [d]^{\tilde{\mathscr{F}}_{\RR}} \ar@{^(->} [r] & \Xi_{\CC} \ar@{->>} [u]_{\pi_{\CC}} \ar@{->>} [d]^{\mathscr{F}_{\CC}}
\\
D \ar@{^(->} [r] & \check{D} \setminus \mathfrak{Z} \ar@{^(->} [r] & \check{D} \ar@{=} [r] & \check{D}
}
\]in which surjectivity of the leftmost and rightmost upward arrows
follows from the $G(\RR)^{\circ}$-conjugacy {[}resp. $G(\RR)$-conjugacy{]}
of all compact maximal real {[}resp. maximal complex{]} tori. The
desired parametrization of Hodge and semi-Hodge flags is then given
by the following
\begin{thm}
\label{thm surj}$(i)$ $\tilde{\mathscr{F}}_{\RR}$ and (ii) $\mathscr{F}_{\RR}$
are surjective.\end{thm}
\begin{proof}
\emph{(i)} Given $F^{\bullet}\in\check{D}$, by \eqref{eqn *!} we
have $F^{\bullet}=\mathscr{F}(H_{\CC},\chi)^{\bullet}$ for some $H\leq G_{\RR}$.

Let $g\in G(\CC)$ be such that $g.F_{0}^{\bullet}=F^{\bullet}$ in
$\check{D}$. Then 
\[
F^{\bullet}=g.\mathscr{F}\left(T_{0,\CC},\chi_{0}\right)^{\bullet}=\mathscr{F}\left(\Psi_{g}(T_{0,\CC}),\Psi_{g}(\chi_{0})\right)^{\bullet}
\]
and the $\{F^{i}\}$ are sums of root spaces of $\Psi_{g}(T_{0,\CC})$,
so that $\Psi_{g}(T_{0,\CC})\subset Q_{F^{\bullet}}$. We also have
$H_{\CC}\subset Q_{F^{\bullet}}$, and so (as in the proof of Lemma
\ref{bigrading lemma}) $H_{\CC}$ and $\Psi_{g}(T_{0,\CC})$ are
conjugate by $\rho\in Q_{F^{\bullet}}$. That is, $\Psi_{\rho g}(T_{0,\CC})=H_{\CC}$
and $\rho g.F_{0}^{\bullet}=F^{\bullet}$; and we conclude that $F^{\bullet}=\mathscr{F}\left(H_{\CC},\Psi_{\rho g}(\chi_{0})\right)^{\bullet}\in\mathscr{F}\left(\tilde{\Xi}_{\RR}\right)$.

\emph{(ii)} Given $\vf\in\check{D}\backslash\mathfrak{Z}$, there
is a compact maximal torus $T\leq G_{\RR}$ with $T(\RR)\supset\vf(S^{1})$,
and $F_{\vf}^{\bullet}=\mathscr{F}(T_{\CC},\chi_{\vf}).$ The rest
of the argument is as in \emph{(i)}.\end{proof}
\begin{rem}
\label{rem p. 13}(a) \emph{(ii)} is essentially part of Theorem (VI.B.9)
in \cite{GGK}, while \emph{(i)} establishes the conjecture made in
Remark (VI.B.10) of {[}op. cit.{]}.%
\footnote{Note that the uniqueness of semi-Hodge decompositions asserted there
is not correct; it depends upon the choice of Cartan.%
}

(b) In the proof of \emph{(i)}, for any $\mu\in\mathrm{X}_{*}(H(\CC))$
we have $\mu=\chi\iff\mathscr{F}(H_{\CC},\mu)=F^{\bullet}$. Hence
$\psi_{\rho g}(\chi_{0})=\chi$ and the original $(H_{\CC},\chi)$
belongs to $\tilde{\Xi}_{\RR}$. This shows that \emph{a real-Cartan/co-character
pair not in $\tilde{\Xi}_{\RR}$} (i.e. not $G(\CC)$-conjugate to
$(T_{0},\chi_{0})$) \emph{does not yield a flag in $\check{D}$. }

(c) Suppose $\mathscr{F}(H,\chi)=F_{\vf}^{\bullet}$ is the flag of
a $(-B)$-polarized Hodge structure $\vf$. Then $B(\cdot,\overline{\cdot}$)
is definite on each $\mathfrak{g}_{\vf}^{p,-p}$. Since $G\subset Aut(\mathfrak{g},-B)$,
the isotropy group $\mathcal{H}_{\vf}\subset G(\RR)$ is compact.
Moreover, $H(\RR)$ commutes with $\vf(S^{1})$, whereupon we have
$\vf(S^{1})\subset H(\RR)\subset\mathcal{H}_{\vf}$, forcing $H(\RR)$
compact. We conclude that \emph{$\mathscr{F}\left(\tilde{\Xi}_{\RR}\backslash\Xi_{\RR}\right)$
avoids the $(-B)$-polarized locus} (which resides between $D$ and
$\check{D}\backslash\mathfrak{Z}$).
\end{rem}
For later reference we emphasize the obvious
\begin{prop}
\label{prop obvious}If $F^{\bullet}\in\check{D}$ is $\mathscr{F}_{\CC}$
of $({}'H,\chi)\in\Xi_{\CC}$, then ${}'H(\CC)\subset Q_{F^{\bullet}}$.\end{prop}
\begin{proof}
The $F^{i}$ are sums of eigenspaces of $\text{ad}(\chi(z))$ (which
are sums of root spaces of ${}'H$), and ${}'\mathfrak{h}$ belongs
to the ``trivial'' eigenspace hence to $F^{0}\mathfrak{g}$($=\mathfrak{q}_{F^{\bullet}}$). 
\end{proof}

\subsection{Discretizing the Cartan data}

Now any given $H\in\tilde{\xi}_{\RR}$ is $G(\RR)^{\circ}$-conjugate
to some Cartan in the list of all successive Cayley transforms of
$H_{0}$ in noncompact imaginary roots (\cite[p. 394]{Kn}). Removing
all but one Cartan in each $G(\RR)^{\circ}$-conjugacy class, we shall
fix henceforth:
\begin{itemize}
\item the resulting sublist $\{H_{0},H_{1},\ldots,H_{n}\}=:\tilde{\xi}_{\RR}^{\theta}$;
\item $\underline{c}(j)\,:=$ product of Cayley transforms with $\Psi_{\underline{c}(j)}(H_{0})=H_{j}$;
\item $\chi_{j}(z):=e^{\log(z)\ad(\xi_{j})}:=\Psi_{\underline{c}(j)}(\chi_{0}(z))\in\mathrm{X}_{*}(H_{j}(\CC))$;
\item the Cartan involution $\Theta:=\Psi_{\vf_{0}(i)}$($=$identity on
$H_{0}$);
\item its $\pm1$-eigenspaces $\mathfrak{k},\mathfrak{p}\subset\mathfrak{g}_{\RR}$.
\end{itemize}
For any $\Theta$-stable $H$, $\alpha\in\Delta_{n}$ $\implies$
$\theta(X_{\pm\alpha})=-X_{\pm\alpha}$ $\implies$ $\Theta(\mathbf{c}_{\alpha})=\mathbf{c}_{\alpha}^{-1}$
$\implies$ $\Theta(\Psi_{\mathbf{c}_{\alpha}}(H))=\Psi_{\mathbf{c}_{\alpha}^{-1}}(H)=\Psi_{\mathbf{c}_{\alpha}}\Psi_{\mathbf{c}_{\alpha}^{-2}}(H)=\Psi_{\mathbf{c}_{\alpha}}(H)$
since $\mathbf{c}_{\alpha}^{\pm2}=w_{\alpha}$ is the Weyl element.
Hence every $H\in\tilde{\xi}_{\RR}^{\theta}$ is $\Theta$-stable.
If $G_{\RR}$ is split, then we shall take $H_{n}$ to be the (unique)
split Cartan in $\tilde{\xi}_{\RR}^{\theta}$.

Noting that $\xi_{0}=\phi_{0}\in i(\mathfrak{h}_{0})_{\RR}$, we have
$\xi_{j}=\text{Ad}(\underline{c}(j))\phi_{0}$. Given any $w\in W_{\CC}(H_{j})$,
$\tilde{w}:=\Psi_{\underline{c}(j)^{-1}}(w)$ belongs to $W_{\CC}(H_{0})$
and 
\[
\xi_{j}^{[w]}:=\text{Ad}(w)\xi_{j}=\Ad(\underline{c}(j))\Ad(\tilde{w})\phi_{0}=:\Ad(\underline{c}(j))\phi_{\tilde{w}},
\]
hence 
\[
\theta(\xi_{j}^{[w]})=\Ad\left(\Theta(\underline{c}(j))\right)\theta(\phi_{0}^{[w]})=\Ad(\overline{\underline{c}(j)})\phi_{0}^{[w]}=\overline{\Ad(\underline{c}(j))\overline{\phi_{0}^{[w]}}}
\]
\[
=-\overline{\Ad(\underline{c}(j))\phi_{0}^{[w]}}=-\xi_{j}^{[w]}.
\]
Writing $Y_{j}^{[w]}:=\frac{\xi_{j}+\overline{\xi_{j}^{[w]}}}{2}$,
$\phi_{j}:=\frac{\xi_{j}^{[w]}-\overline{\xi_{j}^{[w]}}}{2}$,
\[
\mathfrak{a}_{j}:=(\mathfrak{k}_{j})_{\RR}\cap\mathfrak{p}\,,\;\;\;\mathfrak{k}_{j}:=(\mathfrak{h}_{j})_{\RR}\cap\mathfrak{k},
\]
this yields $\phi_{j}^{[w]}\in i\mathfrak{t}_{j}$, $Y_{j}^{[w]}\in\mathfrak{a}_{j}$.
This allows us to associate (nonuniquely) Hodge-compatible Cartan
data to any flag in $\check{D}$:
\begin{prop}
\label{prop '}Given any $F^{\bullet}\in\check{D}$:

(i) there exist $H_{j}\in\tilde{\xi}_{\RR}^{\theta}$, $w\in W_{\CC}(H_{j})$,
and $g\in G(\RR)^{\circ}$ such that
\[
F^{\bullet}=\mathscr{F}\left(\Psi_{g}(H_{j}),\Psi_{g}(\chi_{j}^{[w]})\right)^{\b}\;\left(=:\mathscr{F}(H,\chi)^{\b}\right);
\]
and

(ii) referring to \eqref{eqn !sharp}, there is a Cartan involution
$\Theta_{F^{\bullet}}$ and corresponding $\mathfrak{k}_{F^{\bullet}},\mathfrak{p}_{F^{\bullet}}\subset\mathfrak{g}_{\RR}$,
$\mathfrak{t}_{F^{\bullet}},\mathfrak{a}_{F^{\bullet}}\subset\mathfrak{h}_{\RR}$
such that 
\[
Y\in\mathfrak{a}_{F^{\bullet}}\;\;\textit{and}\;\;\phi\in i\mathfrak{t}_{F^{\bullet}}.
\]
\end{prop}
\begin{proof}
\emph{(i)} By Theorem \ref{thm surj}, $F^{\bullet}$ is $\mathscr{F}$
of some $(H,\chi)\in\tilde{\Xi}_{\RR}$; clearly $H$ is some $\Psi_{g}(H_{j})$.
So $\left(H_{j},\Psi_{g^{-1}}(\chi)\right)$ is $G(\CC)$-, hence
$W_{\CC}(H_{j})$-, conjugate to $(H_{j},\chi_{j})$.

\emph{(ii)} Put $\Theta_{F^{\bullet}}:=\Psi_{g\vf_{0}(i)g^{-1}}$;
then $H$ is $\Theta_{\fb}$-stable and we have $\mathfrak{a}_{F^{\bullet}}=\Ad(g)\mathfrak{a}_{j}$
and $\mathfrak{t}_{F^{\bullet}}=\Ad(g)\mathfrak{t}_{j}$, so the result
follows from the above computations.\end{proof}
\begin{cor}
\label{cor double prime}Given $F^{\bullet}=\mathscr{F}(H,\chi)\in\check{D}$,
with the bigrading $\mathfrak{g}_{\CC}=\oplus\mathfrak{g}^{p,q}$
associated to $H$, $\Theta_{F^{\bullet}}(\mathfrak{g}^{p,q})=\mathfrak{g}^{-q,-p}$.\end{cor}
\begin{proof}
By Prop. \ref{prop '}\emph{(ii)}, $\theta_{F^{\bullet}}(Y)=-Y$ and
$\theta_{F^{\bullet}}(\phi)=\phi$, whereupon the formula for $\tilde{\vf}_{F^{\bullet}}$
gives 
\[
\begin{array}{ccc}
\Ad(\tilde{\vf}_{F^{\bullet}}(w,z))\circ\theta_{F^{\bullet}} & = & \theta_{F^{\bullet}}\circ\Ad(\Theta_{F^{\bullet}}(\vft_{F^{\bullet}}(w,z))\\
 & = & \theta_{F^{\bullet}}\circ\Ad(\vft_{F^{\bullet}}(w^{-1},z)).
\end{array}
\]
Restrict this to $\mathfrak{g}^{p,q}$.\end{proof}
\begin{cor}
\label{cor prime}If $F^{\bullet}\in\mathfrak{Z}$, then for some
$p>0$, $\dim_{\CC}(\mathfrak{g}^{p,p})\neq0$.\end{cor}
\begin{proof}
The idea is to use the basic fact that $G(\RR)$ has a compact Cartan.
Any noncompact Cartan is then an iterated Cayley transform of such
and so must have a real root.

Clearly $H$ is noncompact (i.e. $H\in\tilde{\xi}_{\RR}\backslash\xi_{\RR}$)
and $Y\in\mathfrak{a}_{F^{\bullet}}\backslash\{0\}$ (where $\mathfrak{h}_{\RR}=\mathfrak{a}_{F^{\bullet}}\oplus\mathfrak{t}_{F^{\bullet}}$).
Since the $\dim(\mathfrak{g}^{p.q})$ depend only on $F^{\bullet}$,
we may take $H$ to be of minimal (positive) real rank. Suppose $\beta(Y)=0$
$\forall\mbox{\ensuremath{\beta\in\Delta}}_{\RR}(\neq\emptyset)$,
and let $\hat{H}$ be the (inverse) Cayley transform of $H$ in some
$\beta_{0}\in\Delta_{\RR}$. Then $\hat{\mathfrak{a}}_{F^{\bullet}}=\ker(\beta_{0}|_{\mathfrak{a}_{F^{\bullet}}})\ni Y$
and $\hat{\mathfrak{t}}_{F^{\bullet}}\supset\mathfrak{t}_{F^{\bullet}}\ni i\phi$
$\implies$ $\vft_{F^{\bullet}}$ still factors through $\hat{H}$
$\implies$$F^{\bullet}=\mathscr{F}(\hat{H},\hat{\chi})$, in contradiction
to the presumed minimality of the real rank of $H$.

So $\beta(Y)\neq0$ for some $\beta\in\Delta_{\RR}$, while $\phi\in\mathfrak{t}_{F^{\b},\CC}$
$\implies$ $\mbox{\ensuremath{\beta}}(\phi)=0$, whereupon
\[
\Ad(\vft_{\fb}(w,z))X_{\beta}=\Ad(\chi_{Y}(w))X_{\beta}=w^{\beta(Y)}X_{\beta}
\]
$\implies X_{\beta}\in\mathfrak{g}^{\frac{1}{2}\beta(Y),\frac{1}{2}\beta(Y)}$.
\end{proof}
Because even the real rank ($=\dim_{\RR}(\mathfrak{a}_{\fb})$) of
$\mathfrak{h}_{\RR}$ may not be unique in Proposition \ref{prop '}
(cf. $\S6$), the question arises as to whether some choices are better
than others. At least in one fairly general setting, we shall now
see that this is so. Consider the (well-defined) real parabolic
\[
\mathcal{Q}:=Q_{\tilde{W}}\leq G_{\RR}
\]
defined by \eqref{eqn *sharp}, with Lie algebra $\mathfrak{q}=\tilde{W}_{0}\mathfrak{g}_{\RR}$.
This has unipotent radical $\mathfrak{n}:=\tilde{W}_{-1}\mathfrak{g}_{\RR}$
and, with the choice of $\mathfrak{h}$ and $\Theta_{F^{\bullet}}$,
the natural subalgebras $\tilde{\mathfrak{m}}:=\left(\oplus_{p+q=0}\mathfrak{g}^{p,q}\right)\cap\mathfrak{g}_{\RR}$
and (noting $\mathfrak{a}_{\fb}\subset\mathfrak{g}_{\RR}^{0,0}\subset\tilde{\mathfrak{m}}$)
$\mathfrak{a}:=\ker\{\ad:\mathfrak{a}_{\fb}\to End(\tilde{\mathfrak{m}})\}.$
Let $\mathfrak{m}$ be a direct-sum complement to $\mathfrak{a}$
in $\tilde{\mathfrak{m}}$, and $\M,\A,\N\leq G_{\RR}$ the subgroups
corresponding to $\lm,\la,\n$. This gives a \emph{Hodge-theoretically
defined} Langlands decomposition \begin{equation}\label{eqn **}\Q = \M \A \N
\end{equation}of the parabolic. Write $\L:=\Q/\N$, $\l:=Lie(\L)=Gr_{0}^{\tilde{W}}\g$.
\begin{defn}
$\Q$ (resp. $\lq$, $\fb$) is \emph{cuspidal} $\iff$ $\L/Z(\L)$
has a compact Cartan subgroup.
\end{defn}
Assume that $\Q$ is cuspidal; then it admits a Langlands decomposition
in which the reductive group has a compact Cartan. We claim that (when
true) this may be demonstrated Hodge-theoretically, i.e. via \eqref{eqn **}.
\begin{prop}
\label{prop HT cuspidality}In this case, we can choose $\Theta_{\fb}$
and $H$ ($\Theta_{\fb}$-stable) so that $\la=\la_{\fb}$, making
$\lt_{F^{\b}}\subset\lm$ a Cartan subalgebra. The choices of $H$
which accomplish this are of minimal real rank.\end{prop}
\begin{proof}
Suppose we have $\Theta_{\fb}$, $H$, $\chi$, etc. with $\la\subsetneq\la_{\fb}$,
and note that $\lh_{\RR}\subset\g_{\RR}^{0,0}$. If $\L/Z(\L)$ has
a compact Cartan, then the image $\hat{H}$ of $H$ in it may be (inverse)
Cayley transformed into one. This requires the presence of a real
root vector $X_{\beta}\in\l/\mathfrak{z}(\l)$, which can only lie
in $\left(\l/\mathfrak{z}(\l)\right)^{0,0}=\g_{\RR}^{0,0}/\mathfrak{z}(\l)$.
The preimage $H'$ of $\hat{H}':=\Psi_{\mathbf{d}_{\beta}}(\hat{H})$
has $\lh_{\RR}'\subset\g_{\RR}^{0,0}$ and real rank one less than
$H$. Evidently $\mathbf{d}_{\beta}$ commutes with $\vft_{\fb}$,
and so we still have $\chi\in\mathrm{X}_{*}(H_{\CC}')$. Continue
until the image in $\L/Z(\L)$ is compact.
\end{proof}

\section{Connected real orbits and naive boundary strata}

Continuing to fix $\fb_{0}$, $(H_{0},\chi_{0})$, $\Theta$, and
$\tilde{\xi}_{\RR}$, we now turn to the enumeration and analysis
of the $G(\RR)^{\circ}$-orbits in $\check{D}$.

\subsection{Basic results on orbits}
\begin{prop}
\label{prop codim}Given $F^{\bullet}=\mathscr{F}(H,\chi)\in\check{D}$
with associated bigrading \eqref{eqn !!}, the real codimension of
$\mathcal{O}_{\fb}:=G(\RR)^{\circ}.\fb\subset\check{D}$ is
\[
c_{\fb}:=\sum_{(p,q)\in(\ZZ_{0})^{\times2}}\dim_{\CC}(\g^{p,q}).
\]
\end{prop}
\begin{proof}
Recalling $\lq_{\fb}=\oplus_{(p,q)\in\ZZ_{\geq0}\times\ZZ}\g^{p,q}$
and $T_{\fb}\check{D}\cong\g_{\CC}/\lq_{\fb}$, we have
\[
\begin{array}{ccc}
\dim_{\RR}T_{\fb}\mathcal{O} & = & \dim_{\RR}\left(\g_{\RR}/\{\lq_{\fb}\cap\g_{\RR}\}\right)\\
 & = & \dim_{\CC}\g_{\CC}-\dim_{\CC}\left(\lq_{\fb}\cap\lq_{\overline{\fb}}\right)
\end{array}
\]
while
\[
\dim_{\RR}T_{\fb}\check{D}=2\dim_{\CC}\left(\g_{\CC}/\lq_{\fb}\right)
\]
\[
=2\sum_{(p,q)\in\ZZ_{<0}\times\ZZ}\dim_{\CC}\g^{p,q}
\]
\[
=\sum_{(p,q)\in(\ZZ\backslash\{0\})\times\ZZ}\dim_{\CC}\g^{p,q}
\]
\[
=\dim_{\CC}\g_{\CC}-\dim_{\CC}\left(\lq_{\fb}\cap\lq_{\overline{\fb}}\right)+\dim_{\CC}\left(\bigoplus_{p,q>0}\g^{p,q}\right).
\]
\end{proof}
\begin{cor}
$\mathcal{O}_{\fb}\subset\check{D}$ is open $\iff$ $\fb\in\check{D}\backslash\mathfrak{Z}$.\end{cor}
\begin{proof}
If $\fb\in\mathfrak{Z}$, then by Corollary \ref{cor prime} $c_{\fb}\neq0$;
while $\fb\in\mathfrak{Z}$ $\implies$ only the $\{\g^{p,-p}\}$
are nontrivial $\implies$ $c_{\fb}=0$.
\end{proof}
In Wolf's study \cite{Wo} of complex flag manifolds, it is shown
that $\check{D}$ contains a unique closed orbit $\mathcal{O}_{\mathrm{c}}$
(of real codimension $c_{\mathrm{c}}$); this is in the closure of
all the other $G(\RR)^{\circ}$-orbits, and is acted upon transitively
by $K$. One has $c_{\fb}<c_{\mathrm{c}}$ for any $\fb\notin\mathcal{O}_{\mathrm{c}}$;
and \begin{equation}\label{eqn p. 19 *}c_{\mathrm{c}}\leq \dim_{\CC}\check{D},
\end{equation}with equality if and only if $\mathcal{O}_{\mathrm{c}}$ contains
a \emph{real} flag. We shall say that a flag is of \emph{Hodge-Tate
type} if its $\dim(\g^{p,q})$ are zero for $p\neq q$.
\begin{cor}
(i) Equality holds in \eqref{eqn p. 19 *} if and only if $\check{D}$
contains a Hodge-Tate flag, in which case $\mathcal{O}_{\mathrm{c}}$
is the set of such flags in $\check{D}$.

(ii) In particular, this happens whenever $G$ is $\RR$-split.\end{cor}
\begin{proof}
\emph{(i)} Real flags ($\fb=\overline{\fb}$) are obviously Hodge-Tate,
and the $\{\dim(\g^{p,q})\}$ are constant on orbits. Conversely,
if $\fb$ is Hodge-Tate, then $\dim Gr_{\fb}^{p}=\dim Gr_{\fb_{0}}^{p}=\dim Gr_{\fb_{0}}^{-p}$
$\implies$ 
\[
\dim\check{D}=\sum_{p>0}\dim(\g_{\fb_{0}}^{p,-p})=\sum_{p>0}\dim(\g_{\fb}^{p,p})=c_{\fb}.
\]

\emph{(ii)} $G_{\RR}$ split $\implies$ $H_{n}$ split $\implies$
$\mathrm{X}_{*}(H_{n}(\CC))=\mathrm{X}_{*}(H_{n}(\RR))$ $\implies$
$\mathscr{F}(H_{n},\chi_{n})$ is real.
\end{proof}
Before proceeding to the heart of the section, we can say something
about the codimension-1 orbits as well:
\begin{cor}
\label{cor codim 1 p.19b}If $c_{\fb}=1$, then $\g^{1,1}$ is spanned
by a single real root vector, and the other $\{\g^{p,q}\}$ with $p,q>0$
are zero.\end{cor}
\begin{proof}
By Proposition \ref{prop codim}, only one $\g^{p_{0},q_{0}}$ with
$p_{0},q_{0}>0$ can be nonzero, and since $\dim\g^{q,p}=\dim\g^{p,q}$
we must have $q_{0}=p_{0}$. Now our standing bracket-generating assumption
says that 
\[
\mathcal{W}_{\fb}+[\mathcal{W}_{\fb},\mathcal{W}_{\fb}]+[\mathcal{W}_{\fb},[\mathcal{W}_{\fb},\mathcal{W}_{\fb}]]+\cdots=T_{\fb}\check{D},
\]
 whilst by Frobenius
\[
T_{\fb}\mathcal{O}_{\fb}+[T_{\fb}\mathcal{O}_{\fb},T_{\fb}\mathcal{O}_{\fb}]+\cdots=T_{\fb}\mathcal{O}_{\fb},
\]
and so $\mathcal{W}_{\fb}/\left\{ \mathcal{W}_{\fb}\cap T_{\fb}\mathcal{O}_{\fb}\right\} \neq\{0\}$.
Taking real dimensions,
\[
\begin{array}{ccc}
0 & < & \dim_{\CC}(\g^{-1,-1})+2\sum_{q<-1}\dim_{\CC}(\g^{-1,q})\\
 & = & \dim_{\CC}(\g^{1,1})+2\sum_{q>1}\dim_{\CC}(\g^{1,q})\\
 & = & \dim_{\CC}(\mathfrak{g}^{1,1})
\end{array}
\]
$\implies$ $p_{0}=1$.
\end{proof}

\subsection{Orbit inventory}

Now let $Q_{j}\leq G(\CC)$ denote the parabolic stabilizing $\fb_{j}:=\mathscr{F}(H_{j},\chi_{j})^{\bullet}.$
The Weyl subgroup
\[
\mathrm{W}_{j}:=\frac{N(Q_{j},H_{j}(\CC))}{H_{j}(\CC)}=\text{Stab}\chi_{j}\leq W_{\CC}(H_{j})
\]
is generated by the reflections in the roots belonging to $\ker(\xi_{j})$.

Consider the finite set
\[
\tilde{\Xi}_{\RR}^{\theta}:=\tilde{\pi}_{\RR}^{-1}\left(\tilde{\xi}_{\RR}^{\theta}\right)=\left\{ (H_{j},\chi_{j}^{[w]})\,\left|\, j\in\{0,\ldots,n\},\, w\in W_{\CC}(H_{j})\right.\right\} 
\]
of distinguished $\theta$-stable-Cartan/co-character pairs (and its
subset $\Xi_{\RR}^{\theta}:=\tilde{\pi}_{\RR}^{-1}(\{H_{0}\})$),
where $\chi_{j}^{[w]}(z)=\Psi_{w}(\chi_{j}(z))$. Let $\overline{\tilde{\Xi}_{\RR}^{\theta}}$
(resp. $\overline{\Xi_{\RR}^{\theta}}$) be the set of equivalence
classes modulo the relation
\[
(H_{j},\chi_{j}^{[w]})\sim(H_{k},\chi_{k}^{[w']})\;\iff\; j=k\text{ and }w'\in W_{\RR}^{\circ}(H_{j}).w.\mathrm{W}_{j},
\]
and $\mathcal{O}_{\RR}^{G}(\check{D})$ denote the set of $G(\RR)^{\circ}$-orbits.
Writing $W_{\RR}^{\circ}(H_{j}).w.\mathrm{W}_{j}=:\{w\}$ for the
double cosets, we introduce the orbit map
\[
\begin{array}{ccccc}
\mathbf{o}: & \overline{\tilde{\Xi}_{\RR}^{\theta}} & \to & \mathcal{O}_{\RR}^{G}(\check{D})\\
 & (H_{j},\chi_{j}^{\{w\}}) & \mapsto & \mathcal{O}_{\mathscr{F}(H_{j},\chi_{j}^{[w]})^{\b}} & =:\mathbf{o}_{j}^{\{w\}}.
\end{array}
\]

\begin{thm}
\label{thm p. 20}(i) $\mathbf{o}$ is well-defined and surjective,
with%
\footnote{$\mathbf{e}$ denotes the identity element in a Weyl group; so $\chi_{0}^{\{\mathbf{e}\}}=\chi_{0}$.%
} $\mathbf{o}^{-1}(D)=\{(H_{0},\chi_{0}^{\{\mathbf{e}\}})\}$.

(ii) It restricts to a bijection between $\overline{\Xi_{\RR}^{\theta}}$
and the set $\mathcal{O}_{\RR}^{G}(\check{D}\backslash\mathfrak{Z})$
of open orbits.

(iii) The codimension-one orbits are of the form $\mathbf{o}_{j}^{\{w\}}$
for $H_{j}$ of real rank $1$.\end{thm}
\begin{proof}
\emph{(i)-(ii)}: For well-definedness, $\{w\}=\{w'\}$ $\implies$
$\chi_{j}^{[w']}=\Psi_{g}(\chi_{j}^{[w]})$ for some $g\in N(G(\RR)^{\circ},H_{j})$
$\implies$ 
\[
\mathscr{F}(H_{j},\chi_{j}^{[w']})^{\b}=\mathscr{F}\left(\Psi_{g}(H_{j}),\Psi_{g}(\chi_{j}^{[w]})\right)^{\b}=\Psi_{g}\left(\mathscr{F}(H_{j},\chi_{j}^{[w]})\right)^{\b}.
\]
Surjectivity is Proposition \ref{prop '}\emph{(i)}.

Suppose $\fb:=\mathscr{F}(H_{0},\chi_{0}^{[w]})^{\bullet}$ 
\[
=\Ad(g)\mathscr{F}(H_{0},\chi_{0}^{[w']})^{\b}=\mathscr{F}(\Psi_{g}(H_{0}),\Psi_{g}(\chi_{0}^{[w']}))^{\b},
\]
 where $g\in G(\RR)^{\circ}$. Using Proposition \ref{prop obvious},
$H_{0}(\RR)$ and $\Psi_{g}(H_{0}(\RR))$ are compact maximal tori
of the real reductive Lie group%
\footnote{While $\fb$ is Hodge, the corresponding Hodge structure need not
be $(-B)$-polarized, and so $\mathcal{H}_{\fb}$ need not be compact.%
} 
\[
\mathcal{H}_{\fb}:=G(\RR)^{\circ}\cap Q_{\fb}(\cap Q_{\overline{\fb}}).
\]
So there exists $g'\in\mathcal{H}_{\fb}$ such that $H_{0}=\Psi_{g'g}(H_{0})$;
let $w_{r}\in W_{\RR}^{\circ}(H_{0})$ be the element represented
by $g'g$. Then 
\[
\fb=\Ad(g')\fb=\mathscr{F}\left(\Psi_{g'g}(H_{0}),\Psi_{g'g}(\chi_{0}^{[w']})\right)^{\b}=\mathscr{F}\left(H_{0},\chi_{0}^{[w_{r}w']}\right)^{\b}
\]
\begin{flalign*}\;\;\; & \implies \;\;\; \chi_0^{[w]}=\chi_0^{[w_r w']} &\\
\;\;\;\ & \implies \;\;\; w.\mathrm{W}_0 = w_r w'.\mathrm{W}_0 & \\
\;\;\; & \implies \;\;\; \{w\}=\{w_r w'\}=\{w'\}.
\end{flalign*}This establishes \emph{(ii)}, and the last statement of \emph{(i)}
follows from this and Remark \ref{rem p. 13}(c).

\emph{(iii)}: Suppose $\fb=\mathscr{F}\left(H_{j},\chi_{j}^{[w]}\right)^{\b}$
has $c_{\fb}=1$, $\g^{1.1}=\CC\langle X_{\beta}\rangle$, $\text{rk}_{\RR}H_{j}\,(=\dim_{\RR}\la_{j})\geq2.$
Then $\Ad(\mathbf{d}_{\beta})X_{\beta}\in\CC\langle X_{\beta'}\rangle$,
$\beta'$ a noncompact imaginary root for $H_{j'}=\Psi_{\mathbf{d}_{\beta}}(H_{j})$.
Since $\text{rk}_{\RR}H_{j'}\geq1$, $H_{j'}$ is not maximally compact
and so has a real root $\beta_{0}'$, necessarily orthogonal to $\beta'$.
Under (conjugation by) $\mathbf{c}_{\beta'}$, $\beta_{0}'$ goes
to a real root $\beta_{0}$ (of $H_{j}$) orthogonal to $\beta$,
and the vanishing of the $\{\g^{p,p}\}_{p\neq0,\pm1}$ forces $X_{\pm\beta}\in\g^{0,0}$.
Therefore, conjugation by $\mathbf{d}_{\beta_{0}}$( $=e^{i\frac{\pi}{4}(X_{-\beta_{0}}-X_{\beta_{0}})}$)
replaces $H_{j}$ by $H_{j'}$ of smaller real rank, whilst leaving
$Y$ and $\phi$ --- hence $\fb$ --- alone.\end{proof}
\begin{cor}
(i) $|\mathcal{O}_{\RR}^{G}(\check{D})|\leq\sum_{j=0}^{n}\left|W_{\RR}^{\circ}(H_{j})\backslash W_{\CC}(H_{j})/\mathrm{W}_{j}\right|\;(<\infty).$

(ii) $|\mathcal{O}_{\RR}^{G}(\check{D}\backslash\mathfrak{Z})|=\left|W_{\RR}^{\circ}(H_{0})\backslash W_{\CC}(H_{0})/\mathrm{W}_{0}\right|.$\end{cor}
\begin{rem}
$W_{\CC}(H_{j})\geq\mathrm{W}_{j}$ are the same for each $j$, whilst
$W_{\RR}^{\circ}(H_{j})$ varies considerably. If $G_{\RR}$ is split
($\implies H_{n}$ is), then $W_{\CC}(H_{n})=W_{\RR}(H_{n})=W_{\RR}^{\circ}(H_{n})$
by \cite[14.6]{BT}.
\end{rem}
When the isotropy group $\mathcal{H}_{\vf_{0}}$ is abelian, it is
a compact maximal torus in $G(\RR)^{\circ}$, and we say that $\check{D}$
is a \emph{complete flag variety} (owing to how this situation most
often arises).

Recall our standing assumption (cf. $\S2$) that $\pi_{\chi_{0}}$
takes only the values $0$ and $1$ on the simple roots.
\begin{lem}
\label{lem complete flag}The following are equivalent:

(i) $\check{D}$ is a complete flag variety;

(ii) $Q_{\fb_{0}}$ is a Borel subgroup of $G(\CC)$;

(iii) $Q_{\fb}$ is Borel for any $\fb\in\check{D}$;

(iv) $\pi_{\chi_{0}}$ takes the value $1$ on every simple root.
\end{lem}
Under the equivalent conditions of the Lemma, $H_{0}(\RR)^{\circ}=\mathcal{H}_{\vf_{0}}$
and $\dim_{\CC}(H_{0}(\CC))=\dim_{\CC}(Gr_{\fb}^{0}\g)$, so that
for any $\fb=\mathscr{F}(H_{j},\chi_{j}^{[w]})^{\bullet}$, from $\g^{0,0}\supseteq\lh_{j}$
we have \begin{equation}\label{eqn dagger*}\g^{0,0} = \lh_j.
\end{equation}Moreover, since every $Q_{\fb_{j}}$ is Borel, the $\{\mathrm{W}_{j}\}$
are all trivial.
\begin{thm}
\label{thm 4.9}In the complete flag setting, $\mathbf{o}$ is a bijection
and 
\[
|\mathcal{O}_{\RR}^{G}(\check{D})|=\sum_{j=0}^{n}|W_{\RR}^{\circ}(H_{j})\backslash W_{\CC}(H_{j})|.
\]
\end{thm}
\begin{proof}
Suppose $\fb:=\mathscr{F}(H_{j},\chi_{j}^{[w]})^{\b}$
\[
=\Ad(g)\mathscr{F}(H_{k},\chi_{k}^{[w']})^{\bullet}=\mathscr{F}(\Psi_{g}(H_{k}),\Psi_{g}(\chi_{k}^{[w']}))^{\bullet},
\]
where $g\in G(\RR)^{\circ}$, and let $\g=\oplus\g^{p,q}$ be the
bigrading induced by $H_{j}$. Using Proposition \ref{prop obvious},
$H_{j}(\RR)^{\circ}$ and $\Psi_{g}(H_{k}(\RR)^{\circ})$ are maximal
tori in the identity component of 
\[
\mathcal{H}_{\fb}:=G(\RR)^{\circ}\cap Q_{\fb}(\cap Q_{\overline{\fb}}),
\]
and \eqref{eqn dagger*} says that $\mathcal{H}_{\fb}^{\circ}/U(\mathcal{H}_{\fb}^{\circ})$
is a torus of the same dimension; that is, $\mathcal{H}_{\fb}^{\circ}$
is (connected) solvable. By \cite[Prop. 19.2]{Bo}, there exists $\tilde{g}\in\mathcal{H}_{\fb}^{\circ}$
such that $\Psi_{\tilde{g}}\left(\Psi_{g}(H_{k}(\RR)^{\circ})\right)=H_{j}(\RR)^{\circ}$;
i.e. $\Psi_{\tilde{g}g}(H_{k})=H_{j}$ with $\tilde{g}g\in G(\RR)^{\circ}$.
But then $k=j$ (cf. $\S3.3$), and $\tilde{g}g$ represents an element
$w_{r}\in W_{\RR}^{\circ}(H_{j})$; we have
\[
\fb=\Ad(\tilde{g})\mathscr{F}\left(\Psi_{g}(H_{j}),\Psi_{g}(\chi_{j}^{[w']})\right)^{\bullet}=\mathscr{F}\left(H_{j},\chi_{j}^{[w_{r}w']}\right)^{\bullet}
\]
\begin{flalign*}\;\;\; & \implies \;\;\; \chi_j^{[w]}=\chi_j^{[w_r w']} &\\
\;\;\;\ & \implies \;\;\; w.\mathrm{W}_j = w_r w'.\mathrm{W}_j & \\
\;\;\; & \implies \;\;\; \{w\}=\{w_r w'\}=\{w'\},
\end{flalign*}done.
\end{proof}
Continuing to assume $\check{D}$ a complete flag variety:
\begin{cor}
The ``real rank map'' $\check{D}\to\ZZ_{>0}$ given by 
\[
\fb=\mathscr{F}(H,\chi)^{\bullet}\mapsto\dim_{\RR}\la_{\fb}
\]
is well-defined.
\end{cor}
We also recover the well-known
\begin{cor}
The $\left|W_{\RR}^{\circ}(H_{0})\backslash W_{\CC}(H_{0})\right|$
open orbits are in 1-to-1 correspondence with Weyl chambers up to
reflections in the compact roots. (Explicitly, the correspondence
is given by sending $\mathbf{o}_{0}^{\{w\}}\mapsto w.C_{0}$, where
$C_{0}$ is the chamber associated to $Q_{\fb_{0}}$.)
\end{cor}

\subsection{Closure order}

It remains to address how the various orbits fit together. Consider
the following two operations on the finite set of points $\{\mathscr{F}(H_{j},\chi_{j}^{[w]})^{\bullet}\}$
in $\check{D}$:
\begin{enumerate}
\item \emph{Cayley transforms} $\mathbf{c}_{\alpha}$ in noncompact imaginary
roots:
\[
\fb=\mathscr{F}\left(H_{j},\chi_{j}^{[w]}\right)^{\bullet}\longmapsto\mathscr{F}\left(\Psi_{\mathbf{c}_{\alpha}}(H_{j}),\Psi_{\mathbf{c}_{\alpha}}(\chi_{j}^{[w]})\right)^{\bullet}=:\mathbf{c}_{\alpha}\fb;
\]

\item \emph{Cross actions}, i.e. $c_{\fb}$-increasing Weyl reflections
$w_{\gamma}$ in complex roots:
\[
\fb=\mathscr{F}\left(H_{j},\chi_{j}^{[w]}\right)^{\bullet}\longmapsto\mathscr{F}\left(H_{j},\chi_{j}^{[w_{\gamma}w]}\right)^{\bullet}=:w_{\gamma}\fb.
\]

\end{enumerate}
There is a well-developed theory of \emph{Bruhat order} (i.e., closure
order%
\footnote{By this we mean, in general, the partial order on orbits given by
$\mathcal{O}_{1}\geq\mathcal{O}_{2}\iff cl(\mathcal{O}_{1})\supseteq\mathcal{O}_{2}$.%
}) for the $K_{\CC}$-orbits on complete flag varieties, where $K_{\CC}\leq G(\CC)$
is the complexification of a maximal compact subgroup of $G(\RR)^{\circ}$.%
\footnote{More precisely, one takes $K_{\CC}$ to be the identity connected
component of $G(\CC)^{\Theta}$%
} (The foundational article is \cite{RS}; also see the helpful recent
exposition \cite{Ye}.) We can import these results into our setting
by way of Matsuki duality, which produces a 1-to-1 correspondence
between $K_{\CC}$- and $G(\RR)^{\circ}$-orbits in a complete flag
variety $\check{D}$, while reversing closure order. See \cite{Ma1}
and the Introduction to \cite{Ma2}.

The upshot of this for us is twofold: 
\begin{lyxlist}{00.00.0000}
\item [{(a)}] In the general case, where $\check{D}$ is not necessarily
a complete flag variety, $\mathcal{O}_{\mathbf{c}_{\alpha}\fb}$ and
$\mathcal{O}_{w_{\alpha}\fb}$ always lie in the analytic closure
$cl(\mathcal{O}_{\fb})=:\mathcal{O}_{\fb}\amalg\partial\mathcal{O}_{\fb}$.
This may also be deduced directly from the discussion below. 
\item [{(b)}] In the complete flag case, the codimension-one inclusions
obtained as in (a) generate all closure relations in the sense of
\cite[Theorem 3.15]{Ye} (the ``subexpression property'' which generates
more relations than mere iteration of (a)).\end{lyxlist}
\begin{rem}
\label{rem Bruhat}(i) We should accompany these statements with the
warning that our $\mathbf{c}_{\alpha}$ and $w_{\gamma}$ (which operate
differently from the Cayley transform and cross-action in \cite{Ye})
are not well-defined operations on the level of orbits: if $w_{r}\in W_{\RR}^{\circ}$,
it can happen that (while $\mathcal{O}_{w_{r}\fb}=\mathcal{O}_{\fb}$)
$\mathcal{O}_{\mathbf{c}_{\alpha}w_{r}\fb}\neq\mathcal{O}_{\mathbf{c}_{\alpha}\fb}$
and so forth.

(ii) On the other hand, with $\alpha\in\Delta_{n}(H_{j})$ and $\fb$
as above, we need not worry about both $\mathbf{c}_{\alpha}\fb$ and
$\mathbf{c}_{-\alpha}\fb$: if $\beta\in\Delta(\Psi_{\mathbf{c}_{\alpha}}(H_{j}))$
is $\mathbf{c}_{\alpha}$ of $\alpha$ (more precisely, $X_{\beta}=-i(\Ad\mathbf{c}_{\alpha})X_{\alpha}$),
then $\mathbf{d}_{\beta}=\mathbf{c}_{\alpha}^{-1}$ and $\mathbf{d}_{\beta}^{2}=w_{\beta}$
$\implies$ $w_{\beta}\mathbf{c}_{\alpha}\fb=\mathbf{c}_{\alpha}^{-1}\fb=\mathbf{c}_{-\alpha}\fb$
$\implies$ $\mathcal{O}_{\mathbf{c}_{\alpha}\fb}=\mathcal{O}_{\mathbf{c}_{-\alpha}\fb}.$

(iii) We can further simplify computations by noticing that if $w_{r}\in W_{\RR}^{\circ}(H_{j})$
and $\alpha\in\Delta_{n}(H_{j})$, then we have $w_{r}(\alpha)\in\Delta_{n}(H_{j})$
and so $\mathbf{c}_{\alpha}w_{r}=w_{r}w_{r}^{-1}\mathbf{c}_{\alpha}w_{r}=w_{r}\mathbf{c}_{w_{r}(\alpha)}$
$\implies$ $\mathcal{O}_{\mathbf{c}_{w_{r}(\alpha)}\fb}=\mathcal{O}_{\mathbf{c}_{\alpha}w_{r}\fb}.$
\end{rem}
\begin{disc}To see what is going on in a simple case, consider an
$\mathfrak{sl}_{2}$-triple $X_{\alpha},h_{\alpha},X_{-\alpha}$ with
$\alpha\in\Delta_{n}$, $X_{-\alpha}=\overline{X_{\alpha}}$, $h_{\alpha}=[X_{\alpha},X_{-\alpha}]$
and $[h_{\alpha},X_{\pm\alpha}]=\pm2X_{\pm\alpha}$. Put 
\[
\mathsf{F}^{1}:=\CC\langle X_{\alpha}\rangle\subset\mathsf{F}^{0}:=\CC\langle X_{\alpha}\rangle_{B}^{\perp}=\CC\langle X_{\alpha},h_{\alpha}\rangle\subset\mathfrak{sl}_{2,\CC}.
\]
Writing \begin{flalign*}\;\;\; & \gamma_t:=\exp\left\{\frac{t}{2}(X_{-\alpha}-X_{\alpha})\right\}\in SL_2(\CC), &\\
\text{we have} \\
\;\;\;\ & h(t):=\Ad (\gamma_t)h_{\alpha} = (\cos t)h_{\alpha} + (\sin t)(X_{\alpha}+X_{-\alpha}),& \\ 
\;\;\; & X_{\pm}(t):=\Ad (\gamma_t)X_{\pm \alpha}  &\\
\;\;\; & \;\;\;\;\;\;\;\;\;\; = \frac{1}{2}(\sin t)h_{\alpha} + \frac{1}{2}(\cos t +1)X_{\pm \alpha} + \frac{1}{2}(\cos t -1)X_{\mp \alpha}.
\end{flalign*}In particular, this gives $h(\frac{\pi}{2})=\Ad(\mathbf{c}_{\alpha})h_{\alpha}=X_{\alpha}+X_{-\alpha}$,
$X_{\pm}(\frac{\pi}{2})=\Ad(\mathbf{c}_{\alpha})X_{\pm\alpha}=\frac{1}{2}h_{\alpha}$
and $h(\pi)=\Ad(w_{\alpha})h_{\alpha}=-h_{\alpha}$, $X_{\pm}(\pi)=\Ad(w_{\alpha})X_{\pm\alpha}=X_{\mp\alpha}.$

The \emph{flag} $\mathsf{F}^{1}(t):=\CC\langle X_{+}(t)\rangle\subset\mathsf{F}^{0}(t):=\CC\langle X_{+}(t)\rangle_{B}^{\perp}$
is in fact in the \emph{real} (i.e., $SL_{2}(\RR)$-)orbit of $\mathsf{F}^{\bullet}$
for $t\in[0,\frac{\pi}{2})$; explicitly, we have $\mathsf{F}^{\bullet}(t)=\Ad(\mu_{t})\mathsf{F}^{\bullet}$,
where $\mu_{t}=\text{diag}\left\{ f(t)^{\frac{1}{2}},f(t)^{-\frac{1}{2}}\right\} $
and $f(t)=\frac{1-\sin t}{\cos t}$. The problem at $t=\frac{\pi}{2}$
(where $\mathsf{F}^{\bullet}(\frac{\pi}{2})=\mathbf{c}_{\alpha}\mathsf{F}^{\bullet}$)
is that $X_{+}(t)$ becomes pure imaginary,%
\footnote{If one puts $X_{\pm\alpha}=\left(\begin{array}{cc}
1 & \mp i\\
\mp i & -1
\end{array}\right)$, $h_{\alpha}=\left(\begin{array}{cc}
0 & i\\
-i & 0
\end{array}\right)$, then $(\Ad\mathbf{c}_{\alpha})X_{\alpha}=iX_{\beta}=i\left(\begin{array}{cc}
0 & 0\\
1 & 0
\end{array}\right).$ For $X_{\alpha}\in\mathfrak{g}_{\fb}^{1,-1}$, $X_{\beta}\in\g_{\mathbf{c}_{\alpha}\fb}^{1,1}.$
The Hodge-theoretically minded reader will no doubt think that $\left(\begin{array}{cc}
0 & 0\\
1 & 0
\end{array}\right)$ should be in $\g^{-1,-1}$; this is resolved by the effect of the
naive limit map in $\S5$ below, which roughly ``flips'' indices
$(p,q)\mapsto(-q,-p)$.%
} with \emph{real span}, and the real group $SL_{2}(\RR)$ cannot take
a non-real line to a real one. The comparable result in the general
case follows from this one since the flag $F^{\bullet}$ on $\mathfrak{g}$
along an $SL_{2}$-orbit is determined by the restriction of $F^{\bullet}$
to the $\mathfrak{sl}_{2}$. For cross-actions, the analysis is similar
except the corresponding ``non-real to real'' problem occurs at
$t=\pi$.\end{disc}The most interesting general statement (\emph{not}
assuming $\check{D}$ is a complete flag variety) we can make beyond
Remark \ref{rem Bruhat} and (a)-(b) above it, is that Cayley transforms
give all the codimension-one orbits in the closure of an open orbit:
\begin{prop}
\label{prop p. 25}Let $\mathbf{o}_{0}^{\{w\}}$ be any open orbit
in $\check{D}$ and $\mathbf{o}_{j}^{\{w'\}}$ an orbit of codimension
$1$, where we may take $H_{j}$ of real rank $1$ (cf. Theorem \ref{thm p. 20}(iii)).
Then $\mathbf{o}_{j}^{\{w'\}}\subset\partial\mathbf{o}_{0}^{\{w\}}$
$\iff$ $(H_{j},\chi_{j}^{[w']})=\left(\Psi_{\mathbf{c}_{\alpha}}(H_{0}),\Psi_{\mathbf{c}_{\alpha}}(\chi_{0}^{[w_{0}]})\right)$
for some $\alpha\in\Delta_{n}$, $w_{0}\in\{w\}$.\end{prop}
\begin{proof}
Since ``$\Longleftarrow$'' is clear from the preceding discussion,
we prove the converse, using the elementary observation that any codimension-1
$G(\RR)^{\circ}$-orbit can bound on at most 2 open $G(\RR)^{\circ}$-orbits.

By assumption we have $\Delta_{\RR}(H_{j})=\{\beta,-\beta\}$, so
that $\mathbf{d}_{\beta}$ sends $H_{j}$ to $H_{0}$ (the only real-rank
$0$ Cartan in $\tilde{\xi}_{\RR}^{\theta}$), and $\beta$ to $\alpha\in\Delta_{n}(H_{0})$,
$\fb=\mathscr{F}(H_{j},\chi_{j}^{[w']})^{\bullet}$ to $\mathscr{F}(H_{0},\chi_{0}^{[w_{1}]})^{\bullet}$
(where $\chi_{0}^{[w_{1}]}=\mathbf{d}_{\beta}(\chi_{j}^{[w']})$).
Since $\mathbf{c}_{\alpha}$ reverses the operation, it gives $\partial\mathbf{o}_{0}^{\{w_{1}\}}\supset\mathbf{o}_{j}^{\{w'\}}$.
Writing $w_{\alpha}\in W_{\CC}(H_{0})$ for the element induced by
$\mathbf{c}_{\alpha}^{2}$, we have $\mathbf{d}_{-\beta}=\mathbf{c}_{\alpha}^{2}\mathbf{d}_{\beta}$
hence 
\[
\mathbf{d}_{-\beta}\mathscr{F}(H_{j},\chi_{j}^{[w']})^{\bullet}=\mathscr{F}(H_{0},\chi_{0}^{[w_{\alpha}w_{1}]}),
\]
so that $\partial\mathbf{o}_{0}^{\{w_{\alpha}w_{1}\}}\supset\mathbf{o}_{j}^{\{w'\}}$
is given by $\mathbf{c}_{-\alpha}$. Writing $\delta_{\beta}:=i\frac{\pi}{4}(X_{-\beta}-X_{\beta})$,
it is clear from the discussion that $e^{\epsilon\delta_{\beta}}\fb\in\mathbf{o}_{0}^{\{w_{1}\}}$
whilst $e^{-\epsilon\delta_{\beta}}\fb\in\mathbf{o}_{0}^{\{w_{\alpha}w_{1}\}}$.
Since the projection of $\delta_{\beta}\in\g^{1,1}\oplus\g^{-1,-1}$
to $T_{\fb}\check{D}$ is transverse to $T_{\fb}\mathbf{o}_{j}^{\{w'\}}$,
this clearly establishes that $\mathbf{o}_{0}^{\{w_{1}\}}$ and $\mathbf{o}_{0}^{\{w_{\alpha}w_{1}\}}$
are the only open orbits bounding on $\mathbf{o}_{j}^{\{w'\}}$. Now
apply Theorem \ref{thm p. 20}(ii).\end{proof}
\begin{defn}
The \emph{(naive) boundary strata} of the Mumford-Tate domain $D$
are the $G(\RR)^{\circ}$-orbits in $\partial D\subset\check{D}$.\end{defn}
\begin{cor}
\label{cor p. 27}(i) To obtain (representatives of) all codimension-1
boundary strata, it suffices to consider (modulo equivalence) those
$\fb=\mathbf{c}_{\alpha}\mathscr{F}(H_{0},\chi_{0}^{[w_{r}]})^{\bullet}$,
$\alpha\in\Delta_{n}(H_{0})$ and $w_{r}\in W_{\RR}^{\circ}$, with
$c_{\fb}=1$.

(ii) In the complete flag case, we have $c_{\fb}=1$ in (i) $\iff$
$\alpha$ is orthogonal to a wall of $w_{r}.C_{0}$. The resulting
codimension-1 stratum separates $D$ from the open orbit corresponding
to $w_{\alpha}w_{r}.C_{0}$, the Weyl chamber ``across the wall''
from $w_{r}.C_{0}$.\end{cor}
\begin{proof}
\emph{(i)} is clear from Proposition \ref{prop p. 25}. For \emph{(ii)},
it suffices to consider the case $w_{r}=1$. By Remark \ref{rem Bruhat}(ii),
we may assume $\alpha\in\Delta_{n}^{+}$; $\mathbf{c}_{\alpha}$ transforms
$H_{0}\mapsto H_{j}$, $\alpha$ to $\beta\in\Delta_{\RR}(H_{j})=\{\beta,-\beta\}$,
and $\fb_{0}$ to $\fb:=\mathscr{F}(H_{j},\chi:=\Psi_{\mathbf{c}_{\alpha}}(\chi_{0}))^{\bullet}$
with associated $Y$ and $\phi$. Now,
\[
\alpha\text{ is }\perp\text{ to a wall of }C_{0}\;\iff
\]
\[
\alpha\text{ is simple for }\Delta^{+}\;\iff
\]
\begin{equation}\label{eqn *p.27}\left\{ \begin{array}{c}\beta\text{ is simple for the system } \Delta^{+}(\chi):=\mathbf{c}_{\alpha}(\Delta^{+})\\=\pi_{\chi}^{-1}(\ZZ_{>0})\cap\Delta(H_{j})\text{ of positive roots.}\end{array}\right.
\end{equation}Since $\pi_{\chi}$ of simple roots is $1$ by Lemma \ref{lem complete flag},
and therefore exceeds $1$ for any other positive root, \eqref{eqn *p.27}
is equivalent to $\pi_{\chi}(\beta)=1$.

If $c_{\fb}=1$, then by Corollary \ref{cor codim 1 p.19b} $\g^{1,1}=\CC\langle X_{\beta}\rangle$,
and so $\pi_{\chi}(\beta)=1$. Conversely, assume \eqref{eqn *p.27}
and consider $\gamma\in\Delta\backslash\{\beta\}$ with $X_{\gamma}\in F^{1}$
($\implies$ $\gamma\in\Delta^{+}(\chi)$). Then $\gamma$ is of the
form $\gamma=r\beta+i\delta$, with $r\in\RR$ and $\delta$ in some
open half-plane $\mathbb{H}_{j}\subset\mathfrak{t}_{j}$, and $\bar{\gamma}=r\beta+i(-\delta)$
$\implies$ $\overline{X_{\gamma}}\notin F^{1}$ $\implies$ $X_{\gamma}\notin\overline{F^{1}}$.
So no root vectors besides $X_{\beta}$ lie in $F^{1}\cap\overline{F^{1}}$,
and $c_{\fb}=1$.

The last statement follows from $\mathbf{c}_{\alpha}^{2}=w_{\alpha}$.\end{proof}
\begin{rem}
(i) In the situation of (the proof of) Corollary \ref{cor p. 27},
$\dim\la_{j}=1$, $Y\in\la_{j}$, and $(\ad Y)X_{\beta}=2X_{\beta}$
force $\{X_{\beta},Y,X_{-\beta}\}$ to be an $\mathfrak{sl}_{2,\RR}$-triple.

(ii) We can reduce the computation of $c_{\fb}$ to pictures by computing
the $\dim(\g^{p,q})$, but the following can be faster. Let $\Delta_{+}(\chi_{j}^{\{w\}})\subset\Delta(H_{j})$
be the roots \emph{positively} \emph{graded} by $\chi_{j}^{\{w\}}$.
(This is an actual system of positive roots if and only if $\check{D}$
is a complete flag variety.) Then $c_{\fb}=\left|\Delta_{+}(\chi_{j}^{\{w\}})\cap\overline{\Delta_{+}(\chi_{j}^{\{w\}})}\right|.$
\end{rem}

\section{Boundary components and the naive limit map}

The reader may have noticed the formal similarity between the $\RR$-mixed
Hodge structures associated to flags in $\check{D}$ (cf. $\S3$)
and limiting mixed Hodge structures. In this section, we shall elaborate
on that relationship by determining precisely when a naive boundary
stratum $\mathcal{O}\subset\partial D$ contains a flag $\fb$ in
the ``naive'' limit of a polarized variation of Hodge structure
into $\Gamma\backslash D$, $\Gamma\leq G(\QQ)$ a discrete group.

\subsection{Limiting filtrations}

For simplicity, let
\[
\Phi:\Delta^{*}\to\langle T\rangle\backslash D
\]
be the period map associated to a PVHS over the punctured unit disk,%
\footnote{or more generally over any $\Delta_{\epsilon}^{*}:=\{z\in\Delta^{*}|\,|z|<\epsilon\}$.%
} with unipotent monodromy $T\in G(\QQ)$, and $N:=\log(T)\in\g_{\QQ}$.
We can take the limit of $\Phi$ at the origin in two different ways:
\begin{enumerate}
\item Choosing a local parameter $q$ on $\Delta^{*}$ (and thus $\tau:=\ell(q):=\frac{\log(q)}{2\pi i}$
on $\lh$)
\[
\Psi:=e^{-\tau N}\Phi:\,\Delta^{*}\to\check{D}
\]
is well-defined and extends across the origin by the Nilpotent Orbit
Theorem of \cite{Sc}.%
\footnote{Technically, one chooses also a lift $\tilde{\Phi}$ to define $\Psi$,
but here this is absorbed by the choice of $q$. If we start with
a period map into $\Gamma\backslash D$, the lift becomes essential.%
} Define the \emph{limiting Hodge flag}
\[
\tilde{\mathscr{F}}_{lim}(\Phi):=\Psi(0)\in\check{D},
\]
where the tilde is a reminder of the dependence on $q$.
\item Choosing a lift $\tilde{\Phi}:\lh\to D$ of $\Phi$, we define the
\emph{naive limiting flag} by
\[
\hat{\mathscr{F}}_{lim}(\Phi):=\lim_{\Im(\tau)\to\infty}\tilde{\Phi}(\tau)\in cl(D),
\]
where the limit is taken whilst confining $\Re(\tau)$ to an arbitrary
compact interval. As we shall see, it depends only on $\Phi$.
\end{enumerate}
Remark that in (1), transversality forces $\tilde{\mathscr{F}}_{lim}(\Phi)^{-1}\ni N$
in the limit.

Write $\tilde{F}^{\bullet}:=\tilde{\mathscr{F}}_{\lim}(\Phi)^{\bullet}$,
and let $\tilde{W}_{\b}:=W(N)_{\b}$ denote the unique filtration
on $\g_{\QQ}$ satisfying \begin{flalign*}\;\;\; & \text{(i)}\;\; N(\tilde{W}_{\ell}\g_{\QQ})\subset \tilde{W}_{\ell-2}\g_{\QQ} \;(\forall \ell) &\\
\;\;\;\ & \text{(ii)} \;\; N^k :\, Gr^{\tilde{W}}_k \g_{\QQ} \to Gr^{\tilde{W}}_{-k}\g_{\QQ}\;\text{is an isomorphism}\;(\forall k\geq 0). & 
\end{flalign*}Then by the $SL_{2}$-orbit theorem of \cite{Sc}, $\psi_{q}\Phi:=(\tfb,\tilde{W}_{\b})$
is a $\QQ$-MHS on $\g$, called the limiting mixed Hodge structure
of $\Phi$ (with respect to the parameter $q$). Let 
\[
\g_{\CC}=\bigoplus_{(p,q)\in\ZZ^{2}}\tilde{\g}_{0}^{p,q}
\]
be the unique (Deligne) bigrading such that \begin{flalign*}\;\;\; & \text{(a)} \;\; \tilde{F}^a \g_{\CC} = \bigoplus_{p\geq a ;\, q\in \ZZ} \tilde{\g}^{p,q}_0 &\\
\;\;\;\ & \text{(b)} \;\; \tilde{W}_b \g_{\CC} = \bigoplus_{p+q\leq b} \tilde{\g}^{p,q}_0 & \\
\;\;\; & \text{(c)} \;\; \overline{\tilde{\g}^{b,a}_0} \equiv \tilde{\g}^{a,b}_0 \; \text{mod}\; \oplus_{p<a;\, q<b}\tilde{\g}^{p,q}_0 ,&
\end{flalign*}with \emph{equality} in (c) if and only if $\psi_{q}\Phi$ is $\RR$-split.

Now we clearly have $N\in\left(\tilde{F}^{-1}\cap\overline{\tilde{F}^{-1}}\cap\tilde{W}_{-2}\right)_{\RR}\subset\tilde{\g}_{0,\RR}^{-1,-1}$.
There also exists a unique element $\delta\in\left(\oplus_{(p,q)\in(\ZZ_{<0})^{\times2}}\tilde{\g}_{0}^{p,q}\right)_{\RR}$
(commuting with $N$) and a holomorphic map $\Gamma:\Delta\to\oplus_{p<0;\, q\in\ZZ}\tilde{\g}_{0}^{p,q}$
(with $\Gamma(0)=0$) such that, putting $\tilde{F}_{\RR}^{\b}:=e^{-i\delta}\tfb$,
$(\tfb_{\RR},\tilde{W}_{\b})$ is $\RR$-split and $\tilde{\Phi}(\tau)=e^{\tau N}e^{\Gamma(q)}\tfb$.
Writing $\g_{\CC}=\oplus_{(p,q)\in\ZZ^{2}}\tilde{\g}^{p,q}$ for the
bigrading associated to $\left(\tfb_{\RR},\tilde{W}_{\b}\right)$,
we remark that $\tfb_{\RR}$ does not depend on the choice of $q$,
while $\delta$ is still in $\oplus_{(p,q)\in(\ZZ_{<0})^{\times2}}\tilde{\g}^{p,q}$
and $N\in\tilde{\g}^{-1,-1}$. Moreover, the element $\tilde{Y}\in End(\g_{\RR})$
defined by $\ad(\tilde{Y})|_{\tilde{\g}^{p,q}}=(p+q)\text{id}_{\tilde{\g}^{p,q}}$
($\forall p,q$) belongs to $\tilde{\g}_{\RR}^{0,0}$ (see the proof
of Lemma 3.2 in \cite{KP}, or below) and there is a unique $N_{+}\in\tilde{\g}_{\RR}^{1,1}$
completing $(N,Y)$ to an $\mathfrak{sl}_{2}$-triple. One consequence
of this is that \begin{equation}\label{eqn ***}W(N_+)_{-k} = \bigoplus_{p+q\geq k}\tilde{\g}^{p,q}.
\end{equation}

Computing 
\[
\begin{array}{ccc}
\tilde{\Phi}(\tau) & = & e^{\tau N}e^{\Gamma(q)}e^{i\delta}\tfb_{\RR}\;\;\;\;\;\;\;\;\;\;\;\;\;\;\\
 & = & e^{\tau N}e^{\Gamma(q)}e^{-\tau N}e^{\tau N}e^{i\delta}\tfb_{\RR}\\
 & = & e^{\tau[N,\Gamma(q)]}e^{i\delta}e^{\tau N}\tfb_{\RR},\;\;\;\;\;\;\;
\end{array}
\]
we note that by \cite[p. 478]{CKS}\begin{equation}\label{eqn !! p. 32}e^{\tau N}\tfb_{\RR} = e^{\frac{1}{\tau}N_+\hat{F}^{\b}},
\end{equation}where\begin{equation}\label{eqn ** p. 32}\hat{F}^a := \bigoplus_{p\in \ZZ ;\, q \leq -a} \tilde{\g}^{p,q}.
\end{equation}So for the naive limit we have
\[
\begin{array}{ccc}
\hat{\mathscr{F}}_{lim}(\Phi) & = & \underset{\Im(\tau)\to0}{\lim}e^{q\ell(q)\mathcal{O}(1)}e^{i\delta}e^{\frac{1}{\tau}N_{+}}\hat{F}^{\bullet}\\
 & = & e^{i\delta}\hat{F}^{\b}\;\;\;\;\;\;\;\;\;\;\;\;\;\;\;\;\;\;\;\;\;\;\;\;\;\;\;\;\\
 & = & \hat{F}^{\b},\;\;\;\;\;\;\;\;\;\;\;\;\;\;\;\;\;\;\;\;\;\;\;\;\;\;\;\;\;\;\;
\end{array}
\]
since $\delta\in\hat{F}^{1}$. We conclude from this that the naive
limit can be determined from the limiting Hodge flag, but is independent
of $q$; in fact, it only depends on the \emph{$SL_{2}$-orbit }$e^{\tau N}\tfb_{\RR}$
canonically associated to $\Phi$, which shares its naive limit. Therefore
it will suffice to restrict our investigation of which boundary strata
contain a naive limit flag to limits of $SL_{2}$-orbits. The following
definitions will serve to formalize these observations.
\begin{defn}
\label{defn boundary component}Given a nilpotent element $N\in\g_{\QQ}$,
let $\tilde{B}(N)$ {[}resp. $\tilde{B}_{\RR}(N)${]} be the subset
of $\check{D}$ consisting of flags $\tfb$ such that $e^{\tau N}\tfb$
is a nilpotent {[}resp. $SL_{2}$-{]}orbit: that is,\vspace{2mm}\\
(a) $e^{\tau N}\tfb\in D$ for $\Im(\tau)\gg0$\\
(b) $N\tilde{F}^{j}\subset\tilde{F}^{j-1}$ ($\forall j$)\\
{[}(c) $\left(\tfb,W(N)_{\b}\right)$ is $\RR$-split{]}.\vspace{2mm}\\
The (Hodge-theoretic, rational) \emph{boundary component} associated
to $N$ is 
\[
B(N):=\Ad(e^{\CC\langle N\rangle})\backslash\tilde{B}(N),
\]
with $\RR$-split locus $B_{\RR}(N):=\Ad(e^{\RR\langle N\rangle})\backslash\tilde{B}_{\RR}(N)$.
\end{defn}
Let $N$ be such that $\tilde{B}(N)\neq\emptyset$. Given $\tfb\in\tilde{B}(N)$,
$e^{\tau N}\tfb$ may be regarded as a period map $\Phi_{(\tfb,N)}:\Delta_{\epsilon}^{*}\to\langle e^{N}\rangle\backslash D$
with LMHS $\psi_{q}\Phi_{(\tfb,N)}=\left(\tfb,W(N)_{\b}\right)$.
Clearly, we may regard $\tilde{B}(N)$ as the set of possible LMHS
for period maps with local monodromy $e^{N}$.
\begin{defn}
The \emph{naive limit map}
\[
\begin{array}{cccc}
\mathscr{F}_{lim}^{N}: & \tilde{B}(N) & \longrightarrow & cl(D)\\
 & \tfb & \longmapsto & \hat{\mathscr{F}}_{lim}\left(\Phi_{(\tfb,N)}\right)
\end{array}
\]
sends nilpotent orbits to their naive limit flags.
\end{defn}
There are several important remarks at this point. First, it is clear
that $\mathscr{F}_{\infty}^{N}$ factors through $\tilde{B}_{\RR}(N)$
(and $B(N)$, hence $B_{\RR}(N)$): we have a diagram\[\xymatrix{\tilde{B}(N) \ar@{->>} [rd]_{\sigma_{\RR}} \ar@{->>} [rrd]^{\mathscr{F}^N_{lim}} \\ \tilde{B}_{\RR}(N) \ar@{=} [r] \ar@{^(->} [u] & \tilde{B}_{\RR}(N) \ar@{->>} [r]_{\mathscr{F}^N_{lim}} & \hat{B}(N) \ar@{^(->} [r] & cl(D), }\]where
$\sigma_{\RR}$ is the canonical splitting described above, and
\[
\hat{B}(N):=\mathscr{F}_{lim}^{N}\left(\tilde{B}(N)\right)=\mathscr{F}_{lim}^{N}\left(\tilde{B}_{\RR}(N)\right).
\]

\begin{prop}
A naive boundary stratum $\mathcal{O}$ contains a naive limit of
a VHS if and only if $\mathcal{O}$ contains a $\hat{B}(N)$.
\end{prop}
Next, we may regard $\mathscr{F}_{lim}^{N}$ as sending a ($\QQ$-)LMHS
$\left(\fb,W(N)_{\b}\right)$ to the $\RR$-MHS $\left(\hat{F}^{\b},W(N_{+})_{\b}\right)$,
where $\hfb=\mathscr{F}_{lim}^{N}(\tfb)$ and $N_{+}$ is as in the
argument above. By \eqref{eqn ***} and \eqref{eqn ** p. 32}, this
has Deligne bigrading
\[
\hat{\g}^{p,q}:=\tilde{\g}^{-q,-p}.
\]
In other words, viewing MHS in terms of their bigradings, on the $\RR$-split
locus $\tilde{B}_{\RR}(N)$ the naive limit map \emph{is nothing but
the reflection about the antidiagonal}. As an easy consequence, the
upside-down ``weight'' filtration \eqref{eqn *sharp} attached to
$\hfb\in\check{D}$ is completely determined by $N$:
\[
\sum_{p\in\ZZ}\hat{F}^{p}\cap\overline{\hat{F}^{j-p}}=W(N)_{-j}\,(=\tilde{W}_{-j}).
\]

Furthermore, since Hodge tensors remain Hodge in the limit, the mixed
Hodge representation $\tilde{\vf}_{\tfb}(w,z)$ attached to the bigrading
factors through $G(\RR)$, forcing the associated $i\tilde{\phi}$
and $\tilde{Y}$ into $\tilde{\mathfrak{g}}_{\RR}^{0,0}$. For $\tfb\in\tilde{B}_{\RR}(N)$,
the ``flip'' merely sends these to $i\hat{\phi}:=i\tilde{\phi}$
and $\hat{Y}:=-\tilde{Y}$.%
\footnote{From this perspective, the ``loss of extension-class information''
we shall describe later seems rather surprising, but has the heuristic
explanation that ``more Hodge tensors reside at the bottom of $\mathfrak{sl}_{2}$-chains
than at the top'': flipping them to the top annihilates some extensions.%
} Taking a Cartan subalgebra $\lh\ni\tilde{Y},\tilde{\phi}$, $\lh$
lives in $\tilde{\g}^{0,0}$, whereupon the entirety of Lemma \ref{bigrading lemma}
holds with $\g^{p,q}:=\hat{\g}^{p,q}$, $\fb:=\hfb$, and $\tilde{W}_{\b}:=W(N)_{\b}$.
So the LMHS provides Cartan/co-character data for the naive limit
flag; in particular:
\begin{prop}
In the complete flag case, and more generally whenever $\dim_{\CC}\g^{0,0}=\text{rank}(G_{\CC})$,
$\mathscr{F}_{lim}^{N}|_{\tilde{B}_{\RR}(N)}$ factors unambiguously
through $\tilde{\Xi}_{\RR}$.
\end{prop}
Finally, we observe that it is possible to make $\mathscr{F}_{lim}$
even more ``symmetric'', by extending Definition \ref{defn boundary component}
to the setting of $\RR$-nilpotent orbits, i.e. where $N\in\g_{\RR}$.
The resulting \emph{real boundary components} $\tilde{B}(N)\supset\tilde{B}_{\RR}(N)$
can now only be regarded as parametrizing $\RR$-LMHS $\left(\tfb,W(N)_{\b}\right)$
(with the attendant \emph{much} coarser equivalence relation%
\footnote{Obviously, we are not going modulo this relation, or $\tilde{B}_{\RR}(N)$
would reduce to a point.%
}), an apparent weakness. On the other hand, a short computation with
the Deligne-Schmid formula \eqref{eqn !! p. 32} shows that if we
let $G(\RR)^{\circ}$ act on everything in sight ($(N,Y,N_{+}),\tilde{F}_{\RR}^{\b},\hat{F}^{\b},\tilde{\g}^{p,q}$,
etc.) then the naive limit becomes a $G(\RR)^{\circ}$-equivariant
map\begin{equation}\label{eqn dagger p. 36}\mathscr{F}^{\mathcal{N}}_{lim}:\,\bigcup_{N\in\mathcal{N}}\tilde{B}_{\RR}(N)\longrightarrow cl(D),
\end{equation}where $\mathcal{N}$ is any \emph{nilpotent orbit} (that is, the $G(\RR)^{\circ}$-orbit
of a nilpotent element%
\footnote{This is the usual meaning of the term in Lie theory (as opposed to
Hodge theory).%
}) in $\g_{\RR}$. The image of \eqref{eqn dagger p. 36} is obviously
a boundary stratum, which we shall denote by $\hat{\mathcal{B}}(\mathcal{N})$.
In this sense, if one flag in a stratum is a naive limit, they all
are. On the other hand, there can exist strata of the form $\hat{\mathcal{B}}(\mathcal{N})$
that do not contain a $\hat{B}(N)$ for $N\in\g_{\QQ}$ (cf. $\S6.2.3$),
essentially because there can exist nilpotent orbits with no rational
points.

\subsection{Main results}

Returning to the question motivating this section, we wish to determine
when a naive boundary stratum contains a $\hat{B}(N)$ (or more generally,
is a $\hat{\mathcal{B}}(\mathcal{N})$). The key is given by the following
two definitions, which concern the situation 
\[
\fb=\mathscr{F}(H,\chi)^{\b}\in\mathcal{O}\subset cl(D)
\]
together with its associated
\begin{itemize}
\item bigrading (cf. Lemma \ref{bigrading lemma})\begin{equation}\label{eqn !! p. 37}\g_{\CC} = \bigoplus_{p,q} \g^{p,q},\;\; \overline{\g^{p,q}}=\g^{q,p},
\end{equation}
\item filtration 
\[
\tilde{W}_{-j}\g_{\CC}:=\sum_{p\in\ZZ}F^{p}\cap\overline{F^{j-p}}=\bigoplus_{p+q\geq j}\g^{p,q}
\]
defined over $\RR$, and
\item $\RR$-parabolic subalgebra $\lq:=\tilde{W}_{0}\g$.\end{itemize}
\begin{defn}
\label{defn rational}$\mathcal{O}$ is \emph{rational} $\iff$ $\tilde{W}_{\b}$
is conjugate to a filtration defined over $\QQ$.\end{defn}
\begin{rem}
\label{rem * p. 37} In Definition \ref{defn rational}, it suffices
to assume $\lq$ is $G(\RR)^{\circ}$-conjugate to a $\QQ$-parabolic,
provided $\g_{1}:=\oplus_{p\in\ZZ}\g^{p,1-p}$ bracket-generates $\tilde{W}_{-1}$.
(This does \emph{not} follow from our bracket-generating assumption
on the horizontal distribution.) This is because $\lq=Lie(Q)=\tilde{W}_{0}$
defined over $\QQ$ $\implies$ $\tilde{W}_{-1}=Lie(U(Q))$ and $\tilde{W}_{1}=\tilde{W}_{-1}^{\perp}$
are defined over $\QQ$, and bracket-generation then implies $\tilde{W}_{-2}=[\tilde{W}_{-1},\tilde{W}_{-1}]$
and so forth, so that all filtrands are defined over $\QQ$.\end{rem}
\begin{defn}
\label{defn polarizable} (a) $\mathcal{O}$ is \emph{polarizable}
if and only if there exists $\hat{N}\in\g_{\RR}^{-1,-1}$ such that:\begin{flalign*}\;\;\; & \text{(i)} \;\; \hat{N}^j  \; \text{gives isomorphisms}\; \g^{p,j-p}\overset{\cong}{\to}\g^{p-j,-p}\;\text{for each}\; p,j;\;\text{and} &\\
\;\;\;\ & \text{(ii)}\;\; i^{-j} (-1)^{p+1} B(v,\hat{N}^j \bar{v})>0\; \text{for each}\; p,\; j,\; \text{and} &\\
\;\;\; & \;\;\;\;\;\;\;  v\in\hat{P}^{p,j-p}:=\g^{p,j-p}\cap \ker(\hat{N}^{j+1}). & 
\end{flalign*}(b) $\mathcal{O}$ belongs to the \emph{nilpotent closure} $ncl(D)$
$\iff$ there exists a nilpotent $\hat{N}\in F^{-1}\cap\g_{\RR}$
such that $e^{iy\hat{N}}\fb\in D$ for $y>0$. (Clearly $ncl(D)\subseteq cl(D)$.)
\end{defn}
The criteria (a) and (b) are useful in different situations, and will
turn out to be equivalent (cf. Theorem \ref{thm equiv} below). Evidently
(b) is independent of the choice of $F^{\b}\in\mathcal{O}$ and $H$,
and hence well-defined. (That the same is true for (a) follows from
the proof of Theorem \ref{thm equiv}.)
\begin{rem}
Unlike the $\{\g^{p,q}\}$, the $\hat{P}^{p,q}$ need not be sums
of root spaces, precisely because $\hat{N}$ need not be a multiple
of a root vector, cf. $\S6.2.1$.
\end{rem}
An additional criterion, which will make an appearance in the examples
in $\S6$, is given by
\begin{defn}
A boundary stratum $\mathcal{O}$ {[}resp. boundary component $B(N)${]}
is \emph{cuspidal} $\iff$ $\lq$ {[}resp. $W(N)_{0}\g${]} is a cuspidal
parabolic subalgebra.
\end{defn}
Since the anti-diagonal flip sends $W(N)_{0}\g$ exactly to $\lq$,
the naive limit map sends cuspidal boundary components to cuspidal
strata.
\begin{prop}
Codimension-one boundary strata are cuspidal.\end{prop}
\begin{proof}
This is an immediate consequence of Corollary \ref{cor p. 27}(i),
as the Cartan $\Psi_{\mathbf{c}_{\alpha}}(H_{0})$ will have real
rank $1$, with ``noncompact part'' $\mathcal{A}=e^{\mathbb{R}\langle Y\rangle}$
centralizing the Levi.
\end{proof}
To put definition \ref{defn polarizable}(a) in context, recall the
notion of a polarized $\RR$-MHS (on $(\g,-B)$), which for us shall
mean a triple $\left(W_{\b},\fb,N\right)$ such that:
\begin{lyxlist}{00.00.0000}
\item [{(I)}] $W_{\b}$ is an increasing filtration of $\g_{\RR}$, and
$(\fb,W_{\b})$ is an $\RR$-MHS (not necessarily split), with associated
Deligne bigrading $\{\g^{p,q}\}$ of $\g_{\CC}$;
\item [{(II)}] $N$ is a nilpotent element of $F^{-1}\cap\g_{\RR}$, with
$W(N)_{\b}=W_{\b}$ (which implies $N\in\g^{-1,-1}$); and
\item [{(III)}] the Hodge structure induced by $\fb$ on 
\[
\ker\left\{ N^{j+1}:Gr_{j}^{W}\to Gr_{-j-2}^{W}\right\} =:P_{j}
\]
 is polarized by $-B(\cdot,N^{j}(\cdot))$, for each $j\geq0$.
\end{lyxlist}
Conditions \ref{defn polarizable}(a)(i,ii) are nothing but a translation
of (II,III) for the specific (split) setting considered there. We
shall require a couple of lemmas, from the work of Cattani, Kaplan
and Schmid (cf. \cite[Thm. 6.16]{Sc}, \cite[Cor. 3.13]{CKS}, \cite[(2.18)]{CK}):
\begin{lem}
\label{lem 1 p. 39}If $e^{zN}\fb$ is an $\RR$-nilpotent orbit,
then $(W(N)_{\b},\fb,N)$ is a polarized $\RR$-MHS.
\end{lem}
${}$
\begin{lem}
\label{lem 2 p. 39}If $(W(N)_{\b},\fb,N)$ is a polarized $\RR$-MHS,
then $e^{zN}\fb$ is an $\RR$-nilpotent orbit; if it is $\RR$-split,
then $e^{zN}\fb\in D$ for $\Im(z)>0$.
\end{lem}
We are now ready to prove the first main theorem of this section:
\begin{thm}
\label{thm equiv}For $\mathcal{O}\subset cl(D)$, the following are
equivalent:

(A) $\mathcal{O}\subset ncl(D)$;

(B) $\mathcal{O}$ is of the form $\hat{\mathcal{B}}(\mathcal{N})$;

(C) $\mathcal{O}$ is polarizable.\end{thm}
\begin{proof}
$\underline{(C)\implies(A)}:$ Let $\fb\in\mathcal{O}$, $\{\g^{p,q}\}$,
$\hat{N}\in\g_{\RR}^{-1.-1}$ be as in Definition \ref{defn polarizable}(a)
and $W_{\bullet}=\oplus_{p+q\leq\b}\g^{p,q}$. (Note that $\hat{N}$
must be nilpotent.) By \ref{defn polarizable}(a)(i), we have $W_{\b}=W(\hat{N})_{\b}$,
whereupon \ref{defn polarizable}(a)(i-ii) and \eqref{eqn !! p. 37}
$\implies$ $(W_{\b},\fb,\hat{N})$ is a polarized split $\RR$-MHS.
By Lemma \ref{lem 2 p. 39}, $e^{iy\hat{N}}\fb\in D$ for $y>0$.\vspace{2mm}\\
$\underline{(B)\implies(C)}:$ $(B)$ says that there exists a polarized
$\RR$-mixed Hodge structure $(W(N)_{\b},\tfb,N)$, without loss of
generality $\RR$-split, such that $\lim_{y\to\infty}e^{iyN}\tfb=:\fb\in\mathcal{O}$.
Let $\{\g^{p,q}\}$ be the associated bigrading, $N_{+}\in\tilde{\g}^{-1,-1}$
be as in the discussion preceding \eqref{eqn ***}, and $\hat{N}:=-N_{+}$.
Then $\g^{p,q}:=\tilde{\g}^{-q,-p}$ is the bigrading associated to
$(\fb,W(\hat{N})_{\b})$ and it is an easy exercise to check that
$(W(\hat{N})_{\b},\fb,\hat{N})$ is a polarized $\RR$-MHS. It is
obviously split, and $(C)$ follows at once.\vspace{2mm}\\
$\underline{(A)\implies(B)}:$ Let $\fb\in\mathcal{O}$, $\{\g^{p,q}\}$
be as in the discussion preceding Definitions \ref{defn rational}
and \ref{defn polarizable}, and put $W_{\b}:=\oplus_{p+q\leq\b}\g^{p,q}$,
$\g_{j}:=\oplus_{p+q=j}\g^{p,q}$. By assumption, $\hat{N}\in F^{-1}\cap\g_{\RR}$
is nilpotent with $e^{iy\hat{N}}\fb\in D$ for $y>0$. The projection
of $\hat{N}$ to $\g_{\RR}^{-1,-1}$ still satisfies this hypothesis,
since $i\hat{N}\in F^{-1}\cap\overline{F^{-1}}$ while $F^{-1}\cap\overline{F^{-1}}\cap(F^{0}+\overline{F^{0}})$
belongs to $\g_{\RR}+F^{0}$. Hence we may assume $\hat{N}\in\g_{\RR}^{-1,-1}$.

By Lemma \ref{lem 1 p. 39}, $(W(\hat{N})_{\b},\fb,\hat{N})$ is a
polarized $\RR$-MHS. To deduce that it is split, we will show that
$W(\hat{N})_{\b}=W_{\b}$. If this is not the case, then for some
$j\geq0$ the map
\[
\nu_{j}:\g_{j}\to\g_{-j}
\]
induced by $\hat{N}^{j}$ is not an isomorphism, and there exists
\[
\alpha\in F^{p}\cap\overline{F^{j-p}}\cap\ker\hat{N}^{j}.
\]
But then $\alpha\in\overline{F^{j-p}\cap\ker\hat{N}^{j}}$ $\implies$
\[
\begin{array}{ccccc}
e^{iy\hat{N}}\alpha & \in & e^{iy\hat{N}}(\overline{F^{j-p}\cap\ker\hat{N}^{j}}) & = & \overline{e^{-iy\hat{N}}}(\overline{F^{j-p}\cap\ker\hat{N}^{j}})\;\;\;\;\;\;\;\;\;\;\;\;\;\;\\
 &  &  & = & \overline{e^{iy\hat{N}}\left\{ e^{-2iy\hat{N}}(F^{j-p}\cap\ker\hat{N}^{j})\right\} }\\
 &  &  & \subseteq & \overline{e^{iy\hat{N}}F^{-p+1}},\;\;\;\;\;\;\;\;\;\;\;\;\;\;\;\;\;\;\;\;\;\;\;\;\;\;\;\;\;\;
\end{array}
\]
where the last inclusion is argued as follows: given $\beta\in F^{j-p}\cap\ker\hat{N}^{j}$,
\[
e^{-2iy\hat{N}}\beta=\beta-2iy\hat{N}\beta-\frac{4y^{2}}{2}\hat{N}^{2}\beta+\cdots+\frac{(-2i)^{j-1}}{(j-1)!}y^{j-1}\hat{N}^{j-1}\beta+0
\]
\[
\in F^{(j-p)-(j-1)}=F^{-p+1}.
\]
So 
\[
(0\neq)\, e^{iy\hat{N}}\alpha\in e^{iy\hat{N}}F^{p}\cap\overline{e^{iy\hat{N}}F^{-p+1}}=:F_{y}^{p}\cap\overline{F_{y}^{-p+1}},
\]
where $F_{y}^{\b}\in D$ is a Hodge flag (for $y>0$). This violates
the $p$-opposed condition on a Hodge flag.

We conclude that $(W(\hat{N})_{\b}=W_{\b},\fb,\hat{N})$ is a split
polarized $\RR$-MHS. Let $\hat{Y}\in\g_{\RR}^{0,0}$ be the element
inducing the grading $\{\g_{j}\}$, and $\hat{N}_{+}\in\g_{\RR}^{1,1}$
complete $\hat{N}$, $\hat{Y}$ to an $\mathfrak{sl}_{2}$-triple.
Then setting $\tilde{F}^{-a}:=\oplus_{p\in\ZZ;\, q\leq a}\g^{p,q}$,
$\tilde{W}_{-b}:=\oplus_{p+q\geq b}\g^{p,q}$, $(\tilde{N},\tilde{Y},\tilde{N}_{+}):=(-\hat{N}_{+},-\hat{Y},-\hat{N})$,
we have $\tilde{W}_{\b}=W(\tilde{N})_{\b}$, and $(W(\tilde{N})_{\b},\tfb,\tilde{N})$
is a split polarized $\RR$-MHS. At this point the Deligne-Schmid
formula applies to give $\lim_{y\to\infty}e^{iy\tilde{N}}\tilde{F}^{\b}=\lim_{y\to\infty}e^{-\frac{i}{y}\hat{N}}\fb=\fb$,
completing the proof.\end{proof}
\begin{prop}
(i) Any codimension-one boundary stratum is polarizable.

(ii) Suppose $D$ is \emph{strongly classical}, in the sense that
$\pi_{\chi_{0}}$ takes values in $\{-1,0,1\}$. (This implies in
particular that $D$ is Hermitian symmetric.) Then all boundary strata
are polarizable.

(iii) If $\g_{\mathcal{O}}^{-1,-1}=\{0\}$, then $\mathcal{O}$ is
not polarizable.\end{prop}
\begin{proof}
\emph{(i)} By Corollary \ref{cor codim 1 p.19b}, $\g^{-1,-1}$ (for
some $\fb\in\mathcal{O}$) is spanned by a single real root vector
$X$. Clearly $iX$ spans the normal tangent space $(\mathcal{N}_{\mathcal{O}/\check{D}})_{\fb}$,
and so $X$ or $-X$ satisfies Definition \ref{defn polarizable}(b).
Now use the implication $(A)\implies(C)$.

\emph{(ii)} The normal space $(\mathcal{N}_{\mathcal{O}/\check{D}})_{\fb}$
identifies naturally with $i\g_{\RR}^{-1,-1}$, yielding a diffeomorphism
between a ball $B\ni\fb$ and a ball in $i\g_{\RR}^{-1,-1}$. We must
have $B\cap D\neq\emptyset$, and so there exists $N\in\g_{\RR}^{-1,-1}\backslash\{0\}$
such that $e^{iN}\fb\in D$. The same argument as in the proof of
$(A)\implies(B)$ (with $y\hat{N}$ replaced by $N$) in Theorem \ref{thm equiv}
shows that $W_{\b}=W(N)_{\b}$, whereupon direct calculation establishes
\emph{(ii)} in Definition \emph{\ref{defn polarizable}(a)}. (For
example, given $v\in\ker(N^{2})\subset\g^{1,0}$, we have $e^{iN}v\in F_{0}^{1}=\g_{0}^{1,-1}$
where $F_{0}^{\b}\in D$, so that $0<B(e^{iN}v,\overline{e^{iN}v})=-2iB(v,N\bar{v})$.)

\emph{(iii)} is obvious.
\end{proof}
Our second result completely characterizes when $\mathcal{O}$ has
a flag occurring as the naive limit of a $\QQ$-VHS into a discrete
quotient $\Gamma\backslash D$.
\begin{thm}
Let $\mathcal{O}\subset cl(D)$ be a boundary stratum. Then $\mathcal{O}$
contains a $\hat{B}(N)$ ($N\in\g_{\QQ}$) $\iff$ $\mathcal{O}$
is rational and polarizable.\end{thm}
\begin{proof}
Only ``$\Longleftarrow$'' requires proof. By polarizablility, $\mathcal{O}$
contains the naive limit of an element $(\tfb,\tilde{W}_{\b}=W(N)_{\b})\in\tilde{B}(N)$,
$N\in\g_{\RR}$. Bearing in mind the antidiagonal flip, $\tilde{W}_{\b}$
is the $\tilde{W}_{\b}$ in Definition \ref{defn rational}, and by
rationality we may assume it is defined over $\QQ$. The issue is
whether we can orbit by $Q(\RR)$ to get $N$ into $\g_{\QQ}$. But
$W_{-2}\g_{\RR}$ is exactly $N(W_{0}\g_{\RR})$, and for $\gamma\in W_{0}\g_{\RR}$,
\[
\frac{d}{dt}\left(\Ad(e^{t\gamma})N\right)|_{t=0}=\ad(\gamma)N=-N(\gamma).
\]
Hence $T_{N}\left(\Ad Q(\RR).N\right)=W_{-2}\g_{\RR}$, and $\Ad Q(\RR).N$
contains an open subset of $W_{-2}\g_{\RR}$ centered about $N$.
Since $\QQ$-points are dense in $\tilde{W}_{-2}\g_{\RR}$, we are
done.\end{proof}
\begin{rem}
It suffices to assume $\mathfrak{q}$ is $G(\RR)^{\circ}$-conjugate
to a $\QQ$-parabolic if $\g_{1}$ bracket-generates $\tilde{W}_{-1}$
(cf. Remark \ref{rem * p. 37}) or if $\tilde{W}_{-1}=\tilde{W}_{-2}$.
\end{rem}

\subsection{Dimension formulas}

In the remainder of this section, we turn our attention to the naive
limit map
\[
\mathscr{F}_{lim}^{N}:\,\tilde{B}(N)\twoheadrightarrow\hat{B}(N)
\]
and its image, where $N\in\g_{\QQ}\backslash\{0\}$ is nilpotent and
$\tilde{B}(N)\neq\emptyset$.

Let $\tfb\in\tilde{B}_{\RR}(N)$ resp. $\fb:=\mathscr{F}_{lim}^{N}(\tfb)\in\hat{B}(N)$
be given, with associated MHS $(\tfb,\tilde{W}_{\b}:=W(N)_{\b})$
resp. $(\fb,W_{\b}:=W(N_{+})_{\b})$%
\footnote{As above, $N_{+}$ is determined by $(N,Y)$ where $Y$ arises from
the Deligne bigrading associated to $(\tilde{F}^{\b},\tilde{W}_{\b})$.%
} and bigradings $\tilde{\g}^{p,q}$ resp. $\g^{p,q}=\tilde{\g}^{-q,-p}$
with dimension $h^{p,q}:=\dim_{\CC}\tilde{\g}^{p,q}=\dim_{\CC}\g^{p,q}$.
In $\tilde{\g}^{\b,\b}$ indexing we have the pictures \[\includegraphics[scale=0.6]{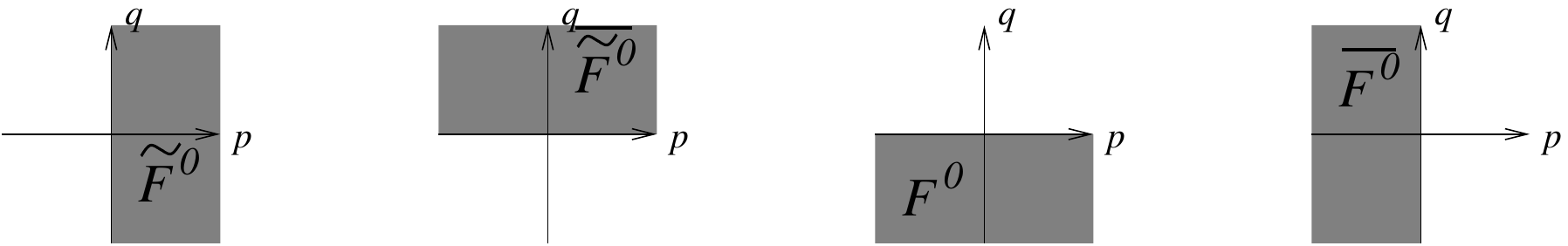}\]
For $p+q\geq0$, the primitive subspaces are $\tilde{P}^{p,q}:=\{\ker(N^{p+q+1})\subseteq\tilde{\g}^{p,q}$\}
and 
\[
P^{p,q}:=N^{p+q}(\tilde{P}^{p,q})=\{\ker(N)\subseteq\g^{p,q}\},
\]
of dimension $h_{prim}^{p,q}$; for $p+q\leq0$ we set $z^{p,q}:=h_{prim}^{-q,-p}$.
Write $\tilde{\g}_{j}:=\oplus_{p+q=j}\tilde{\g}^{p,q}=\g_{-j}$, and
$\lq=\tilde{W}_{0}\g=Lie(Q)$. Of course, $c_{\fb}>0$ and $\hat{B}(N)\subset\mathcal{O}_{\fb}\subset\partial D$.

We recall the following basic material on boundary components from
\cite[sec. 7]{KP}.%
\footnote{In this paragraph, all groups are tacitly identity components of the
group written.%
} Let $Z(N)$ denote the centralizer of $N$ in $G$, with unipotent
radical $U_{N}$ and Levi subgroup $G_{N}\cong Z(N)/U_{N}$; write
\[
\mathscr{G}(N):=U_{N}(\CC)\rtimes G_{N}(\RR)\geq U_{N}(\RR)\rtimes G_{N}(\RR)=Z(N)(\RR).
\]
On the Lie algebra level, we have $\mathfrak{z}(N):=Lie(Z(N)(\CC))=\ker(\ad N)\subset W_{0}\g_{\CC},$
$\mathfrak{u}_{N}:=Lie(U_{N}(\CC))=\mathfrak{z}(N)\cap W_{-1}\g_{\CC},$
and $\g_{N}:=Lie(G_{N}(\CC))=\g_{0}\cap\mathfrak{z}(N).$ Letting
$Z(N)(\RR)$ resp. $\mathscr{G}(N)$ act on $(\tfb,\tilde{W}_{\b})$
gives isomorphisms
\[
\tilde{B}_{\RR}(N)\cong Z(N)(\RR)/\{Z(N)(\RR)\cap Q_{\tfb}(\cap Q_{\overline{\tfb}})\},
\]
\[
\tilde{B}(N)\cong\mathscr{G}(N)/\{\mathscr{G}(N)\cap Q_{\tfb}\}.
\]
Passing to the associated graded (or quotienting by $U_{N}$) gives
projections from both to the Mumford-Tate domain
\[
D(N)\cong G_{N}(\RR)/\mathcal{H}_{N}
\]
(where $Lie(\mathcal{H}_{N})=\g_{\RR}^{0,0}\cap\mathfrak{z}(N)$),
which is Hermitian symmetric if $\g_{N}\subset\g^{-1,1}\oplus\g^{0,0}\oplus\g^{1,-1}.$
The boundary components $B_{\RR}(N)$ and $B(N)$ are obtained by
quotienting out by $e^{\RR\langle N\rangle}$ resp. $e^{\CC\langle N\rangle}$
on the left. Finally, from the equivariant nature of $\mathscr{F}_{lim}^{N}$,
it is evident that 
\[
\hat{B}(N)=Z(N)(\RR).\fb\cong Z(N)(\RR)/\{Z(N)(\RR)\cap Q_{\fb}(\cap Q_{\overline{\fb}})\}.
\]
Like $Q_{\tfb}\cap Z(N)(\RR)$, $Q_{\fb}\cap Z(N)(\RR)$ projects
to $\mathcal{H}_{N}$ in $G_{N}(\RR)$; and so $\hat{B}(N)$, too,
maps to $D(N)$. In a diagram, we have \begin{equation}\label{eqn !* p. 46}\xymatrix{
{} & B(N) \ar@{->>} [rd] \ar@{->>} [ddl]
\\ 
{} & B_{\RR}(N) \ar@{^(->} [u]_{(\beta)} \ar@{->>} [dl] \ar@{->>} [r]^{(\gamma)} & \hat{B}(N) \ar@{->>} [dll]^{(\alpha)} \ar@{^(->} [r]^{(\delta)} & \mathcal{O}_{\fb} \subset \partial D,
\\ 
D(N)
}\end{equation} and one might wonder (for example) when $(\alpha)$-$(\delta)$
are isomorphisms. To state the next result, consider the regions
\begin{lyxlist}{00.00.0000}
\item [{$\mathbf{I}$:}] $p<0$, $q\geq0$, and $p+q\leq0$
\item [{$\mathbf{I'}$:}] $q>0$ and $p+q\leq0$
\item [{$\mathbf{I''}$:}] $q>0$ and $p+q<0$
\item [{$\mathbf{II}$:}] $p\leq0$ and $q\leq0$, but $(p,q)\neq(0,0)$
\item [{$\mathbf{II'}$:}] $p<0$ and $q<0$
\end{lyxlist}
in $\ZZ^{2}$; for example, \textbf{$\mathbf{I'}$} is \[\includegraphics[scale=0.6]{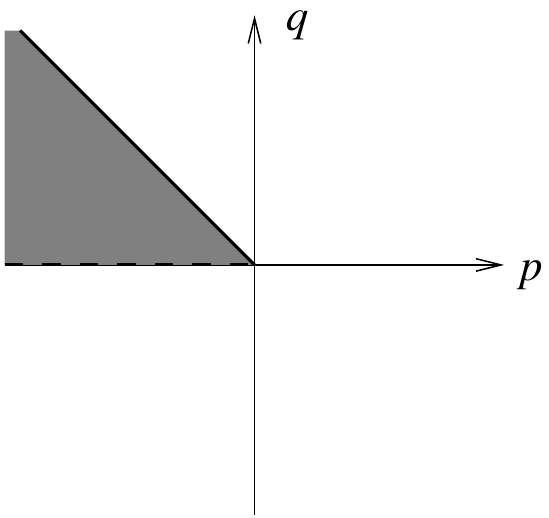}\]
and $\mathbf{I}\amalg\mathbf{II}\amalg\overline{\mathbf{I}}=\mathbf{I'}\amalg\mathbf{II}\amalg\overline{\mathbf{I'}}$
consist of pairs $(p,q)\neq(0,0)$ with $p+q\leq0$. Write $d:=\dim_{\CC}\check{D}$.
\begin{prop}
(Dimension formulas)\\
(i) $\dim_{\RR}B(N)=2\sum_{(p,q)\in\mathbf{I}}z^{p,q}+2\left(\sum_{(p,q)\in\mathbf{II'}}z^{p,q}\;-\;1\right)$.\\
(ii) $\dim_{\RR}B_{\RR}(N)=2\sum_{(p,q)\in\mathbf{I}}z^{p,q}+\left(\sum_{(p,q)\in\mathbf{II'}}z^{p,q}\;-\;1\right)$.\\
(iii) $\dim_{\RR}\hat{B}_{\RR}(N)=2\sum_{(p,q)\in\mathbf{I'}}z^{p,q}$.\\
(iv) $\dim_{\RR}D(N)=2\sum_{p<0}z^{p,-p}$.\\
(v) $\dim_{\RR}\mathcal{O}_{\fb}=\sum_{_{p<0\text{ or q>0}}}h^{p,q}=2d-c_{\fb},$ 

where we recall $c_{\fb}=\sum_{(p,q)\in\mathbf{II'}}h^{p,q}.$\end{prop}
\begin{proof}
Only \emph{(i)-(iii)} require justification. These results follow
from computing tangent spaces:
\[
T_{\tfb}\tilde{B}_{\RR}(N)\cong\mathfrak{z}(N)_{\RR}/\{\mathfrak{z}(N)_{\RR}\cap\tilde{F}^{0}(\cap\overline{\tilde{F}^{0}})\}
\]
\[
=\mathfrak{z}(N)_{\RR}/\mathfrak{z}(N)_{\RR}^{0,0}\;\cong\;\left(\bigoplus_{(p,q)\neq(0,0)}P^{p,q}\right)_{\RR};
\]
similarly,
\[
T_{\tfb}\tilde{B}(N)\cong\mathfrak{z}(N)/\{\mathfrak{z}(N)\cap\tilde{F}^{0}\}=\bigoplus_{p\in\ZZ;\, q>0}P^{p,q}.
\]
To get tangent spaces to $B(N)$ and $B_{\RR}(N)$, quotient out the
span of $N$ in $P_{\RR}^{1,1}$ resp. $P^{1,1}$. Noting that $\tilde{W}_{0}\cap F^{0}\supset\tilde{W}_{0}\cap\tilde{F}^{0}$
and $\mathfrak{z}(N)\subset\tilde{W}_{0}$ $\implies$ $\mathfrak{z}(N)\cap F^{0}\supset\mathfrak{z}(N)\cap\tilde{F}^{0},$
\[
T_{\fb}\hat{B}(N)\cong\mathfrak{z}(N)_{\RR}/\{\mathfrak{z}(N)\cap F^{0}(\cap\overline{F^{0}})\}\mspace{100mu}
\]
\[
\mspace{100mu}\cong\mathfrak{z}(N)\cap\left(\bigoplus_{p<0\text{ or }q>0}\tilde{\g}^{p,q}\right)=\left(\bigoplus_{p<0\text{ or }q>0}P^{p,q}\right)_{\RR}.
\]

\end{proof}
One immediate consequence is that if $\fb$ is Hodge-Tate, then $\hat{B}(N)$
is a point (namely, $\fb$). We also have:
\begin{cor}
The maps in \eqref{eqn !* p. 46} are isomorphisms under the following
conditions:\\
$(\alpha):$ $z^{p,q}=0$ for $(p,q)\in\mathbf{I''}$;\\
$(\beta):$ $c_{\fb}=1$;\\
$(\gamma):$ $c_{\fb}=1$, and $z^{p,0}=0$ for $p<0$; or equivalently, 

$z^{p,q}=0$ for $(p,q)\in\mathbf{II}\backslash\{(1,1)\}$ and $z^{-1,-1}=1$;\\
$(\delta):$ never.
\end{cor}
On an infinitesimal level, the $P^{p,q}$ with $p+q=0$ and $p\neq0$
parametrize the Hodge structure given by $\tfb$ on the associated
graded $\oplus_{i}Gr_{i}^{\tilde{W}}$; the $P^{p,q}$ with $p+q>0$
parametrize extension classes. The information lost by the naive limit
map $(\gamma)$ is precisely that which is parametrized by the $P^{p,q}$
with $p\geq0$ and $q\geq0$ (except for $\langle N\rangle\subset P^{1,1}$).
Carayol's nonclassical $SU(2,1)$ example (\cite{Car,KP}), with mixed-Hodge
diagram \[\includegraphics[scale=0.6]{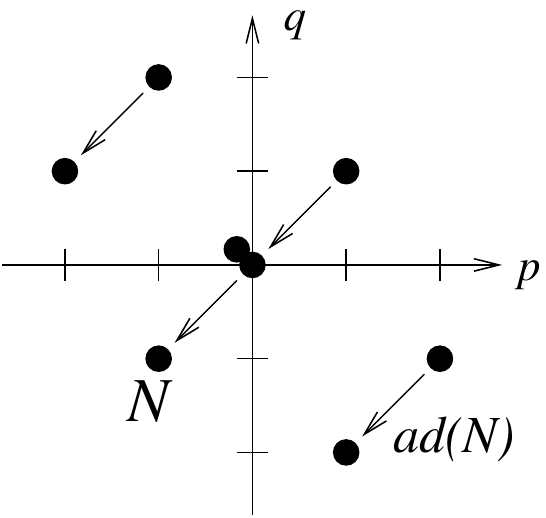}\] is one instance where
$(\gamma)$ is an isomorphism.

At another extreme is the \emph{strongly classical} case, where $D$
is Hermitian symmetric with $F^{-1}\g_{\CC}=\g_{\CC}$, so that (for
$\Gamma\leq G(\QQ)$ neat arithmetic) $X:=\Gamma\backslash D$ is
a connected Shimura variety. From \cite{KU}, it is known that smooth
toroidal compactifications (as in \cite{AMRT}) are obtained by adding
in quotients of our $B(N)$'s (parametrizing nilpotent orbits). The
Baily-Borel compactification, on the other hand, is obtained by using
$\Gamma$-invariant sections of $K_{D}^{\otimes M}$ (for some $M\gg0$)
to embed $X$ in a projective space, and then taking (Zariski or analytic)
closure. Heuristically, since $K_{D}\cong\bigwedge^{d}(F^{-1}/F^{0})$
measures exactly the changes in the Hodge flag (in any direction),
the limits of sections of any $K_{D}^{\otimes M}$ keep track of limits
of flags. Writing $\Gamma_{N}:=Stab(\langle N\rangle)\leq\Gamma$,
this gives part \emph{(a)} of:
\begin{cor}
In the strongly classical case:

(a) the map \[\xymatrix{\Gamma_N \backslash B(N) \ar@{->>} [r]_{\mathscr{F}^{N,\Gamma}_{lim}} & \Gamma_N \backslash \hat{B}(N) }\] recovers
the map from the AMRT boundary components to the Baily-Borel boundary
components; and

(b) $\mathscr{F}_{lim}^{N,\Gamma}$ sends LMHS (up to $\Gamma$ and
$e^{\CC\langle N\rangle}$) to their associated graded (up to $\Gamma_{N}$,
in $\Gamma_{N}\backslash D(N)$).\end{cor}
\begin{proof}
We must show $\mathscr{F}_{lim}$ loses all extension information,
or equivalently that $(\alpha)$ is an isomorphism. But this is clear
since the nonzero $\tilde{\g}^{p,q}$ are in the square $[-1,1]^{\times2}.$
\end{proof}

\subsection{Parabolic induction and parabolic orbits}

For applications of this material to representation theory, which
we plan to pursue in subsequent work, the following Hodge-theoretic
approach to an Iwasawa decomposition will be of use. Writing $\fb=\mathscr{F}(H,\chi)^{\bullet}$,
let $\Theta$ be as in Proposition \ref{prop '} so that (by Corollary
\ref{cor double prime}) $\Theta(\g^{p,q})=\g^{-q,-p}$. For each
$(p,q)$ with $p+q\neq0$, the $(+1)$- and $(-1)$-eigenspaces of
$\Theta$ on $\left(\g^{p,q}\oplus\g^{q,p}\oplus\g^{-q,-p}\oplus\g^{-p,-q}\right)_{\RR}$
are clearly both of (real) dimension $2h^{p,q}$. On the anti-diagonal
line, we can use the fact that $(\tilde{W}_{\b},\tfb,N)$ is a polarized
MHS on $(\g,-B)$ to deduce that the sign of $\Theta$ on $N^{j}\tilde{P}^{p+j,j-p}\subset\tilde{g}^{p,-p}$
is $(-1)^{j+p}$, determining its eigenspaces in $\g_{\RR}^{0,0}$
and $(\g^{p,-p}\oplus\g^{-p,p})_{\RR}$. The sum $\mathfrak{k}$ of
all the $(+1)$-eigenspaces is the Lie algebra of a maximal compact
subgroup $K\leq G(\RR)$.

Now assume $\mathcal{Q}:=Q(\RR)$ is cuspidal, with Langlands decomposition
$\mathcal{Q}=\mathcal{MAN}$ as in Proposition \ref{prop HT cuspidality};
in particular, $\mathcal{MA}=G_{0}(\RR)$ and $\mathcal{N}=U(\mathcal{Q})$.
(Note that $Lie(\mathcal{Q})=\lq_{\RR}=\tilde{W}_{0}\g_{\RR},$ $Lie(\mathcal{N})=\tilde{W}_{-1}\g_{\RR}$,
and $\mathcal{M}$ is reductive with compact Cartan, but possibly
not connected.) Then $\mathfrak{k}\oplus\lq_{\RR}$ evidently gives
all of $\g_{\RR}$, and so
\[
G(\RR)^{\circ}=K\mathcal{Q}=K\mathcal{MAN}.
\]
Since $\mathcal{Q}$ contains the stabilizer $Q_{\fb}(\cap Q_{\fb})\cap G(\RR)^{\circ}$
of $\fb$ in $G(\RR)^{\circ}$, we have a natural fibration
\[
\mathcal{O}_{\fb}\;\overset{\pi}{\twoheadrightarrow}\; G(\RR)^{\circ}/\mathcal{Q}\;\cong\; K/K\cap\mathcal{M}\;=:\;\mathcal{K}_{\fb},
\]
with compact base and fibers of real dimension $\sum_{(p,q)\in\mathbf{I'}}h^{p,q}\,(\geq\dim\hat{B}(N))$,
one of which contains $\hat{B}(N)$.

Associated to any complex representation $\rho:\mathcal{Q}\to Aut(V)$
is a vector bundle 
\[
\mathcal{V}_{\rho}:=\frac{G(\RR)^{\circ}\times V}{\mathcal{Q}}\twoheadrightarrow\mathcal{K}_{\fb},
\]
and we may define a representation $Ind_{\mathcal{Q}}^{G(\RR)^{\circ}}(\rho)$
of $G(\RR)^{\circ}$ by letting the latter act by left translation
on the space of $C^{\infty}$ $\CC$-valued sections of $\mathcal{V}_{\rho}$.
In order for boundary components to provide a useful framework for
studying these representations, we should have at least $\mathcal{M}/Z(\mathcal{M})\subseteq G_{N}(\RR)$,
or equivalently 
\[
\left\{ \begin{array}{c}
z^{p,-p}=h^{p,-p}\;\text{for}\; p\neq0\;\;\;\;\;\;\;\;\;\;\;\\
\text{and}\;\;\;\;\;\;\;\;\;\;\;\;\;\;\;\;\;\;\;\;\;\;\;\;\;\;\;\;\;\;\;\\
\g^{0,0}/\mathfrak{z}(\g_{0})\cong\{\ker(N)\subset\g^{0,0}\}.
\end{array}\right.
\]
In this situation, one can begin with a representation $\mu$ of $G_{N}(\RR)$,%
\footnote{optionally twisted by a character of the component group of $\mathcal{M}$%
} for instance on a coherent cohomology group $H^{i}(D(N),\mathcal{O}(\mathcal{V}'))$
(with $\mathcal{V}'$ a holomorphic homogeneous vector bundle over
$D(N)$), together with a character $\sigma$ of $\mathcal{A}$, and
pull $\sigma\mu$ back to $\mathcal{Q}$ (via the projection $\mathcal{Q}\twoheadrightarrow\mathcal{Q}/Z(\mathcal{M})\mathcal{N}$).

Two special cases of interest are:
\begin{lyxlist}{00.00.0000}
\item [{(A)}] when $\mathcal{O}_{\fb}$ has $c_{\fb}=1$, $\dim(Gr_{\pm2}^{W}\g)=1$,
and $\dim(Gr_{\pm k}^{W}\g)=0$ for $k>2$; and
\item [{(B)}] when $\mathcal{O}_{\fb}$ is the closed orbit, $(\fb,W_{\b})$
Hodge-Tate, and $G_{N}$ trivial.
\end{lyxlist}
In case (B), $\mathcal{M}$ is finite, and for $\sigma=\Delta_{\mathcal{Q}}^{\frac{1}{2}}$
($\Delta_{\mathcal{Q}}:=$ modular character) together with an appropriate
choice of $\mu$, $Ind_{\mathcal{Q}}^{G(\RR)}(\sigma\mu)$ is the
direct sum of the TDLDS (totally degenerate limits of discrete series)
for $G(\RR)^{\circ}$.

In another (but related) direction, we expect in some cases (including
(A) above), the $Q(\CC)$-orbit of $\tfb$ to play a role in generalizing
H. Carayol's results on Fourier coefficients \cite{Car} for nonclassical
automorphic cohomology classes (in some $H^{i}(\Gamma\backslash D,\mathcal{O}(\mathcal{V}))$).
More precisely, $(Q(\CC).\tfb)\cap D$ will project to a sort of punctured
tubular neighborhood of $\Gamma_{N}\backslash B(N)$ in $\Gamma\backslash D$
, whenever
\[
\dim_{\CC}(Q(\CC).\tfb)=\sum_{(p,q)\in\mathbf{I}\amalg\mathbf{II'}}h^{p,q}
\]
equals $\dim_{\CC}(B(N))+1$, which is to say when $z^{p,q}=h^{p,q}$
for $(p,q)\neq(0,0).$ (Basically, this gives a homogeneous structure
to a union of nilpotent orbits.) The pullback of a cohomology class
to this neighborhood then is expected to have a Laurent expansion
``about'' $\Gamma_{N}\backslash B(N)$, with coefficients lying
in groups of the form $\left\{ H^{i}(\Gamma_{N}\backslash B(N),\mathcal{O}(\mathcal{L}^{\otimes k}\otimes W))\right\} _{k\in\ZZ}.$
This will be taken up in a future work.

\section{Examples}

The simplest nontrivial case is, of course, the upper half-plane,
Let $G=PGL_{2}$, $D=\mathfrak{H}$, $\check{D}=\PP^{1}$, where we
think of $\mathfrak{H}$ as parametrizing polarized Hodge structures
on $\g=\mathfrak{sl}_{2}$ with $h^{-1,1}=h^{0,0}=h^{1,-1}=1$. The
group $G(\RR)$ has two components, with $G(\RR)^{\circ}\cong SL_{2}(\RR)/\{\pm\text{id}\}$
and \Tiny $\left[\left(\begin{array}{cc}
0 & 1\\
1 & 0
\end{array}\right)\right]$\normalsize  in the non-identity component, so that $W_{\RR}^{\circ}$
is trivial and $W_{\RR}=W_{\CC}\cong\ZZ/2\ZZ$. The two nontrivial
$G(\RR)^{\circ}$-conjugacy classes of Cartan subgroups in $G(\RR)^{\circ}$
are depicted in a Hasse diagram \[\includegraphics[scale=0.6]{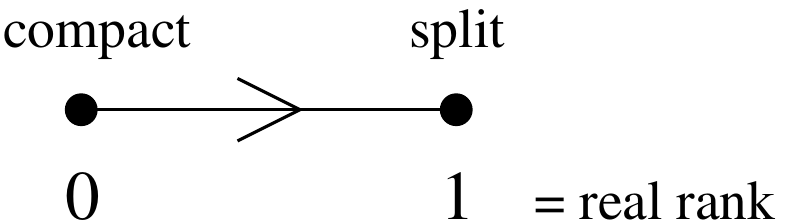}\]where
the arrow denotes a Cayley transform. The $G(\RR)^{\circ}$-orbits
in $\check{D}$ are of course \begin{equation}\label{eqn EHD sl2}\includegraphics[scale=0.6]{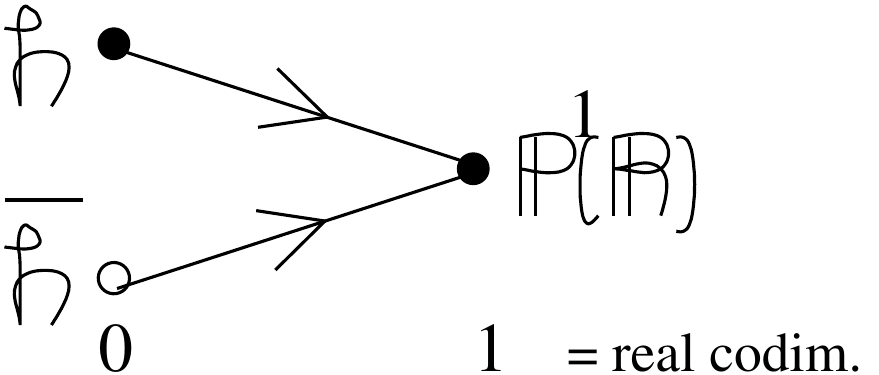}\end{equation}where
the segments denote incidence: that is, the orbit on the right endpoint
is contained in the closure of the left-endpoint orbit. 

We call \eqref{eqn EHD sl2} an \emph{enhanced Hasse diagram}, and
produce them for a number of other examples in $\S6.1$. The ``enhancements''
are as follows:
\begin{itemize}
\item a solid vertex corresponds to an orbit in $ncl(D)$ (i.e., polarizable);
\item a vertex with an ``$\times$'' denotes an orbit in $cl(D)\backslash ncl(D)$
(i.e. non-polarizable);
\item an open vertex signifies an orbit not in $cl(D)$;
\item a solid edge with {[}resp. without{]} an arrow is an incidence obtained
via a Cayley transform {[}resp. cross-action{]}; and
\item a dotted edge is an incidence deduced from the subexpression property.
\end{itemize}
Orbits will be labeled as in $\S4.2$ (viz., $\mathbf{o}_{j}^{\{w\}}$),
and we shall indicate as well the $\{\dim(\g^{p,q})\}$ attached to
each orbit: for $PGL_{2}$ this is simply \[\includegraphics[scale=0.6]{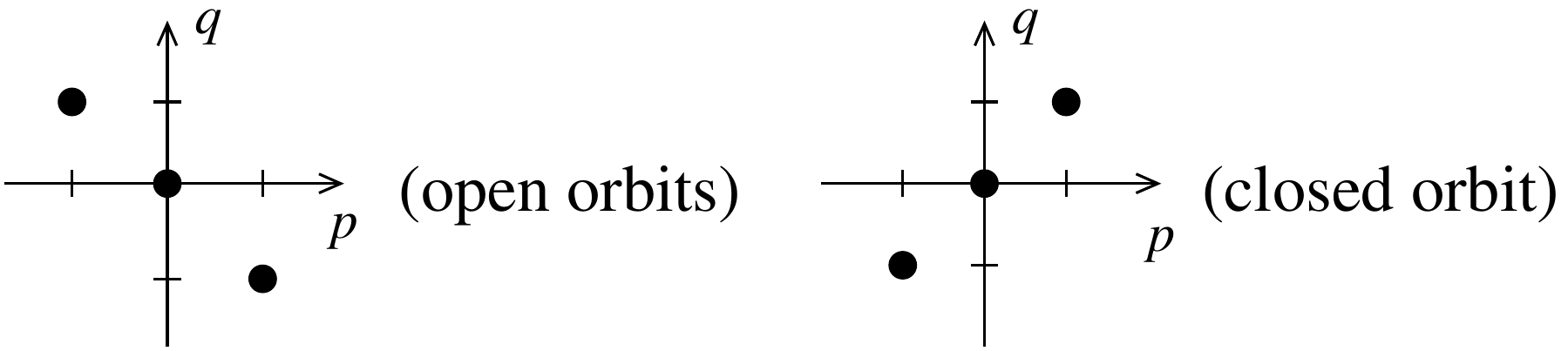}\]where
a dot stands for a single complex dimension. We call these \emph{mixed
Hodge diagrams}.

In $\S6.2$ we will discuss a few simple ``negative examples'' which
motivated the definitions in $\S5.2$.

\subsection{Enhanced Hasse diagrams}

This subsection treats the cases where $G$ is (a $\QQ$-form of)
$SU(2,1)^{ad}$, $PSp_{4}$, and ($\RR$-split) $G_{2}$, briefly
illustrating the method for $SU(2,1)^{ad}$ and merely describing
results for the other two groups.

\subsubsection{$G=SU(2,1)^{ad}$}

$G(\RR)\,(=G(\RR)^{\circ})$ has two conjugacy classes of Cartan subgroups
\[\includegraphics[scale=0.6]{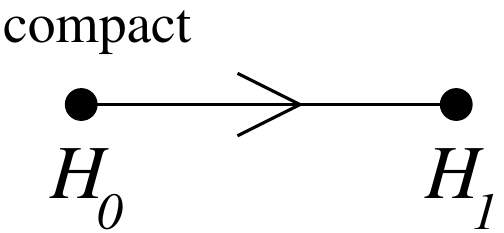}\]and no Cartan of real rank
$2$. In the root diagram associated to a choice of (real) Cartan,
we denote by \[\includegraphics[scale=0.6]{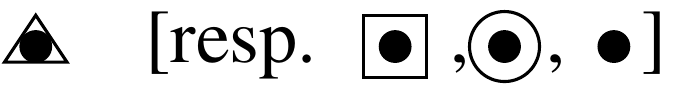}\]a noncompact
imaginary {[}resp. compact imaginary, real, complex{]} root.%
\footnote{We represent the Cartan subalgebra in our root diagram by two bullets
at the origin.%
} A character $\chi\in\mathrm{X}_{*}(H(\CC))$ is depicted by shading
half of the root diagram, which is meant to heuristically indicate
$\chi^{-1}(\RR_{\geq0})\subset\Lambda\otimes\RR$. We begin with $(H_{0},\chi_{0})$,
apply a Cayley transform to get $(H_{1},\chi_{1})$, then apply $W_{\CC}$
to both, and finally, group the results in $W_{\RR}^{\circ}(H_{0})$-
resp. $W_{\RR}^{\circ}(H_{1})$-orbits. These latter are labeled by
their images under the orbit map ($\S$4.2), with elements of $W_{\CC}$
written in terms of reflections in simple roots.

Now, the results willl depend upon $D$, for which there are essentially
two choices compatible with the assumptions of $\S2$. For both, we
let $V$ be a $6$-dimensional $\QQ$-vector space, with a symplectic
form $Q$ and a decomposition $V_{\QQ(i)}=V_{+}\oplus\overline{V_{+}}$
with $Q(V_{+},V_{+})=0$. We assume that the Hermitian form $H(v,w):=-2iQ(v,\bar{w})$
on $V_{+}$ has signature $(2,1)$, so that the projectivization of
those $v\in V_{+}$ with $H(v,v)<0$ yields a (Picard) $2$-ball $\mathbb{B}\subset\PP(V_{+})$.
This parametrizes Hodge structures on $V$ with Hodge numbers $(3,3)$,
and a nontrivial involution (with eigenspaces $V_{+},\overline{V_{+}}$).
Alternatively, one can consider Carayol's nonclassical domain $\mathbb{D}$
parametrizing point-line pairs $(p,L)$ in $\PP(V_{+})$ with $L\cap\mathbb{B}\neq\emptyset$
and $p\in L\backslash(L\cap cl(\mathbb{B}))$, or equivalently HS-with-involution
on $V$ with numbers $(1,2,2,1)$. The corresponding Hodge structures
on $\g$ have Hodge numbers $(2,4,2)$ for $\mathbb{B}$ and $(1,2,2,2,1)$
for $\mathbb{D}$.

W carry out the procedure for $\mathbb{D}$ first: since it is a complete
flag domain, Theorem \ref{thm 4.9} applies. From $H_{0}$, we obtain
the (three) open orbits: \[\includegraphics[scale=0.5]{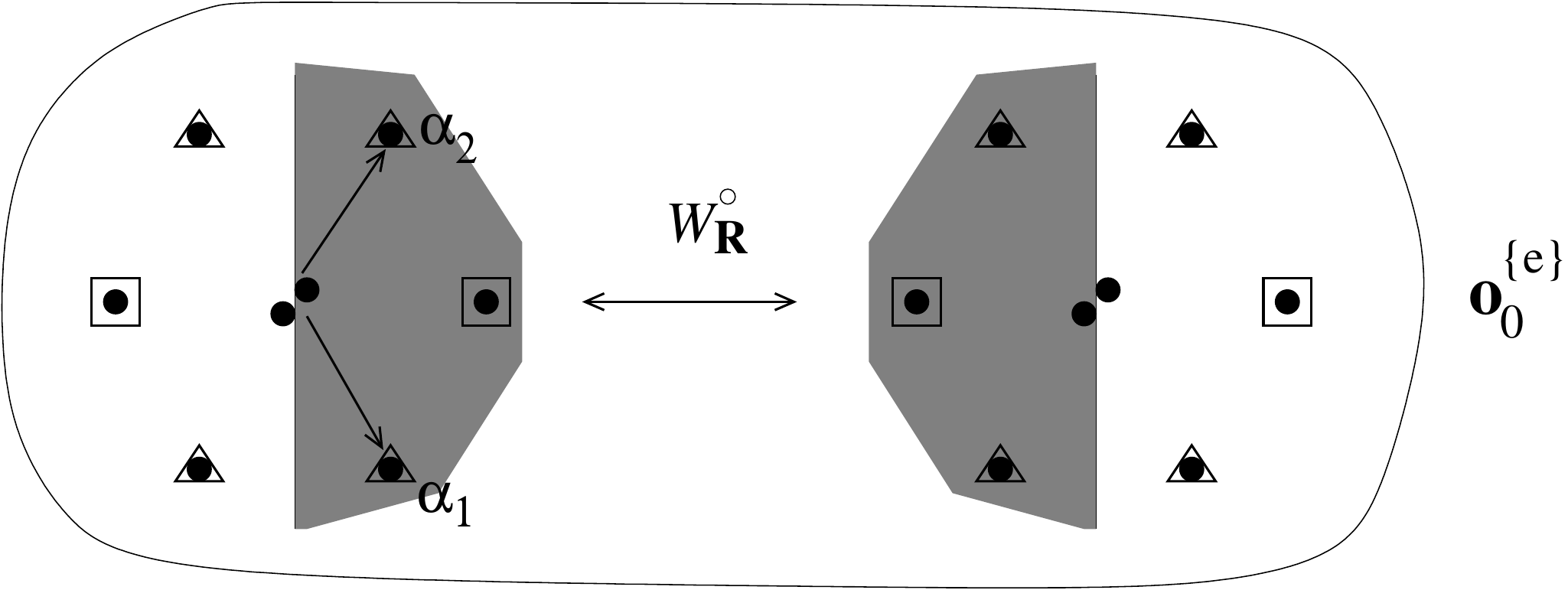}\]\[\includegraphics[scale=0.5]{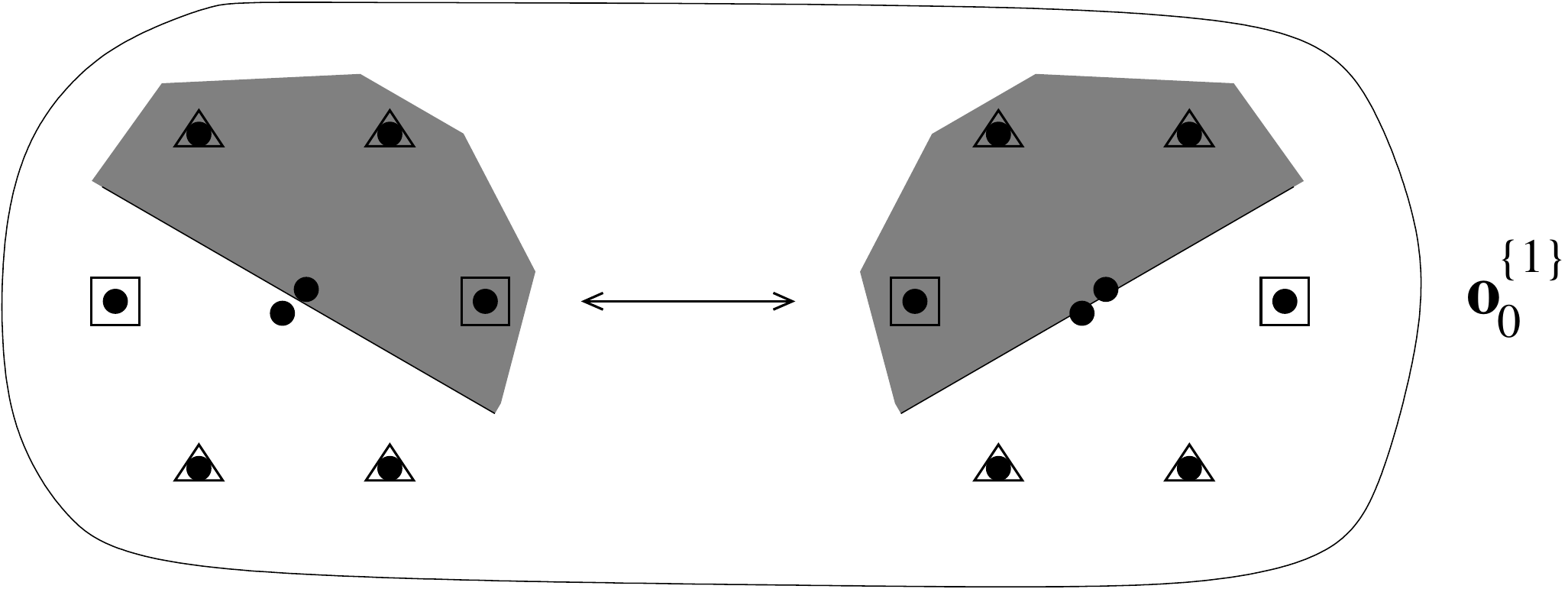}\]\[\includegraphics[scale=0.5]{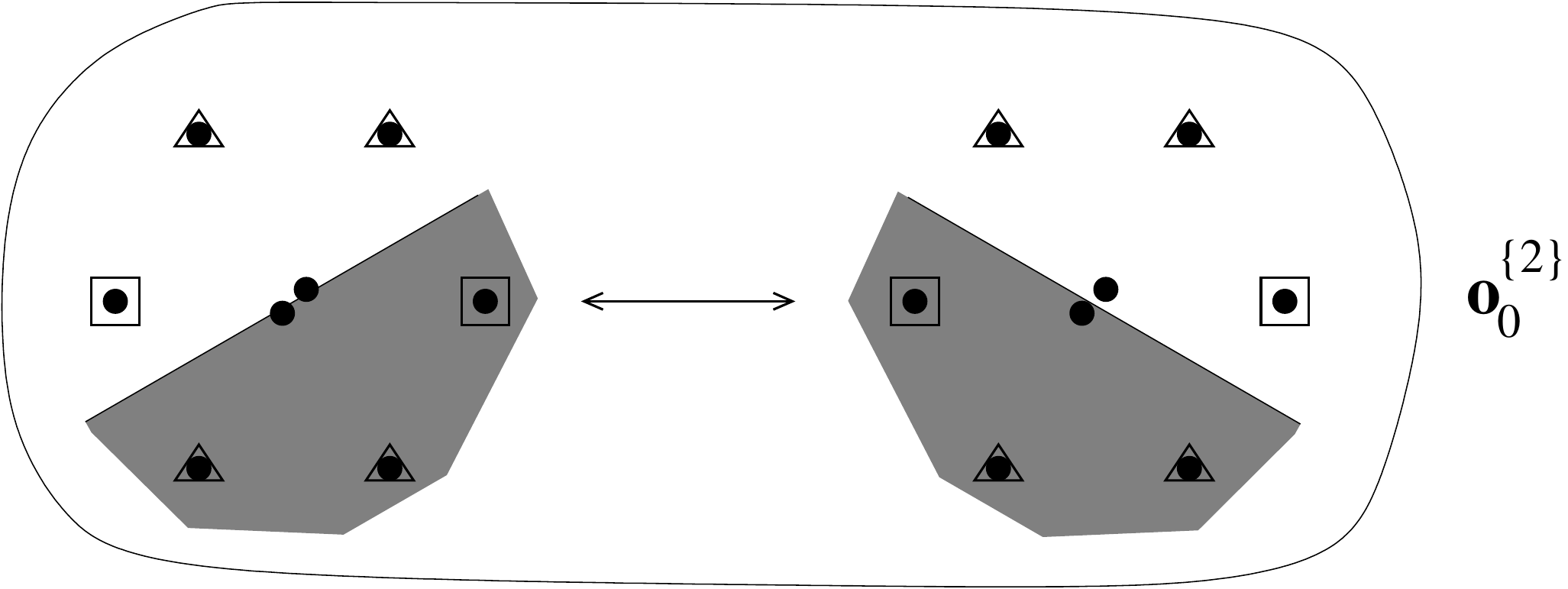}\](Note
that $W_{\RR}(H_{0})\cong\ZZ/2\ZZ$ is generated by reflection in
$\boxed{\bullet}$.) From $H_{1}$, we get three more orbits whose
codimensions can be read off from the accompanying mixed Hodge diagrams:
\[\includegraphics[scale=0.45]{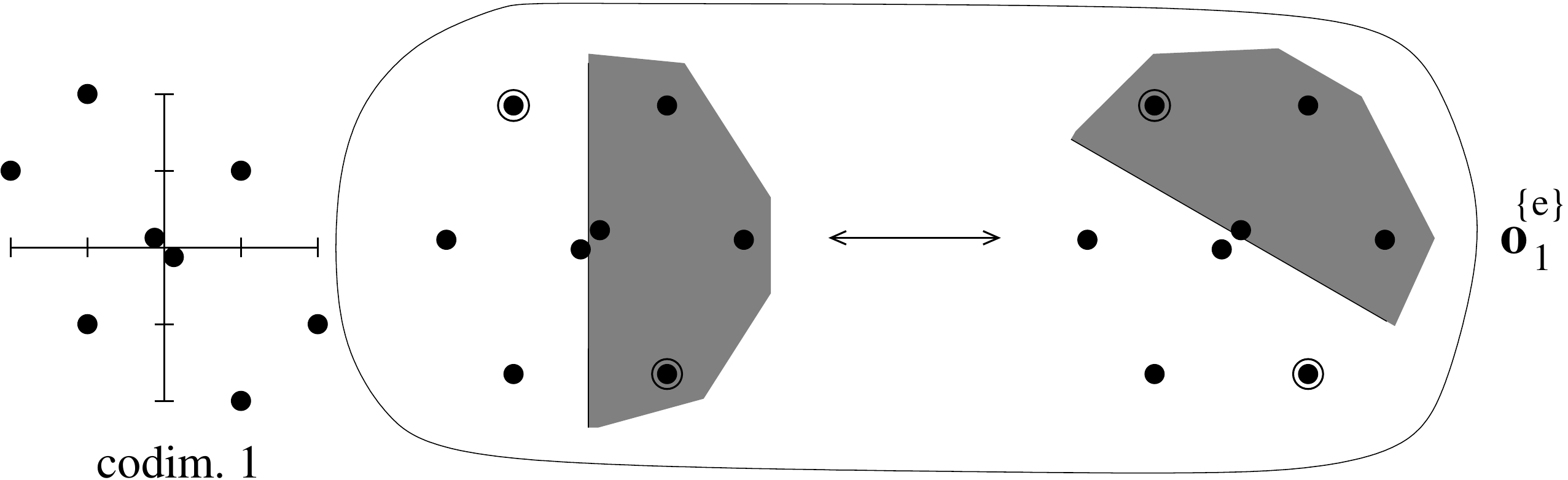}\]\[\includegraphics[scale=0.45]{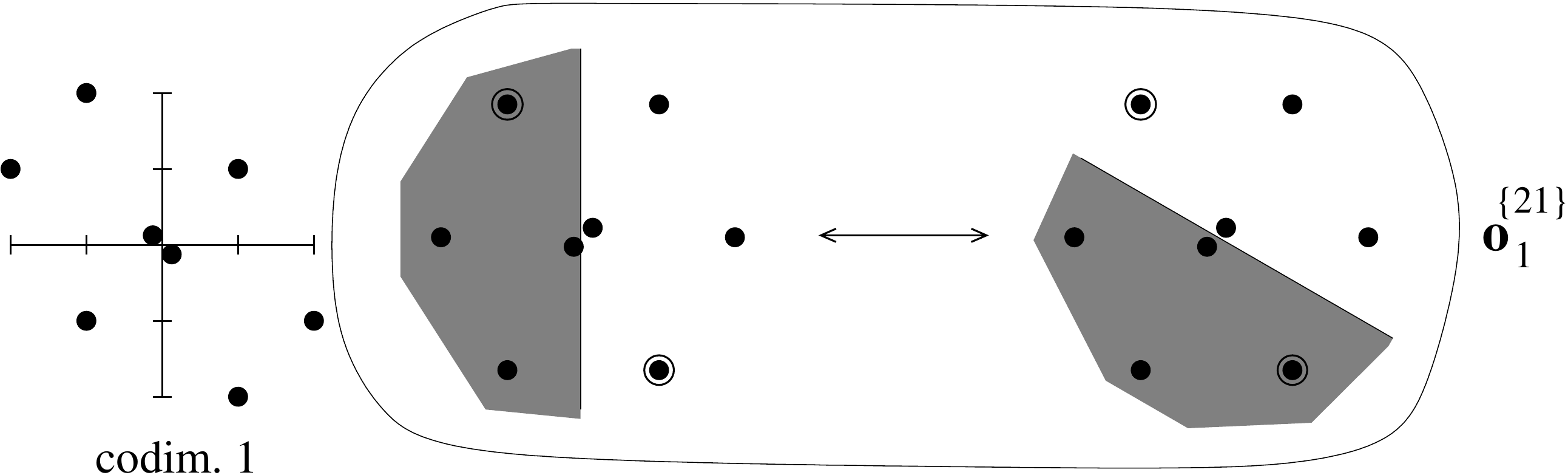}\]\[\includegraphics[scale=0.45]{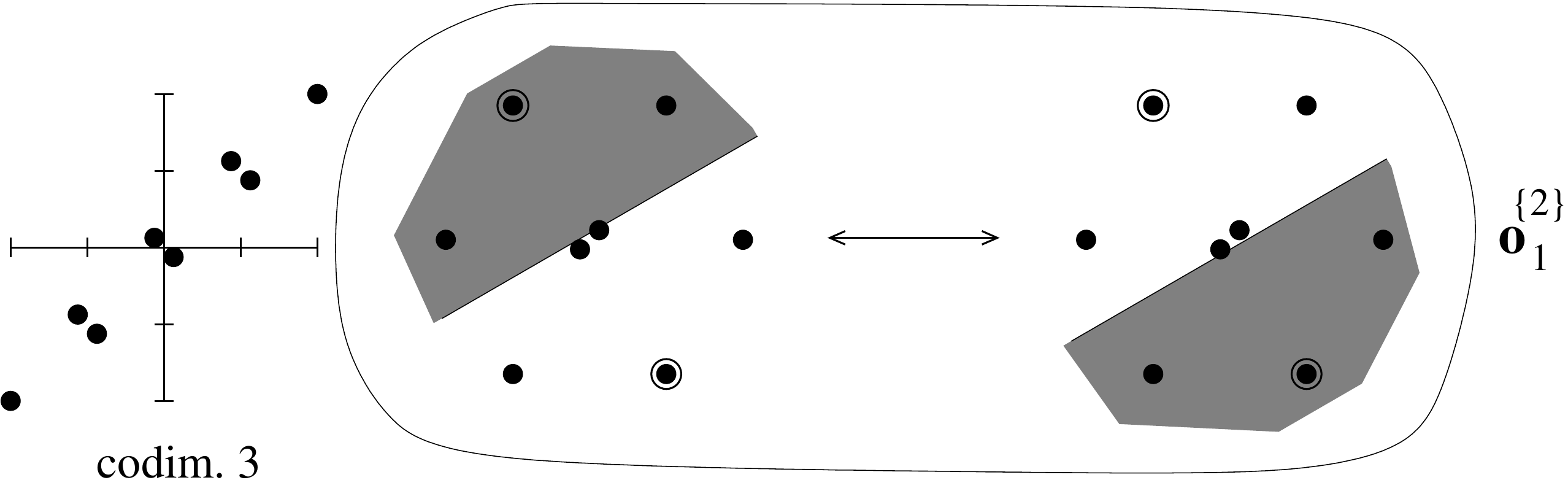}\](Here
$W_{\RR}(H_{1})\cong\ZZ/2\ZZ$ is generated by reflection in a real
root.) 

For the enhanced Hasse diagram, we read off inclusions
\[
cl(\mathbf{o}_{0}^{\{1\}})\supset\mathbf{o}_{1}^{\{e\}}\subset cl(\mathbb{D})\supset\mathbf{o}_{1}^{\{21\}}\subset cl(\mathbf{o}_{0}^{\{2\}})
\]
from visually ``obvious'' Cayley transforms,%
\footnote{What seems clear from the pictures may be justified more carefully
using $\S4.3$; this is left to the reader.%
} obtaining the left half of \[\includegraphics[scale=0.6]{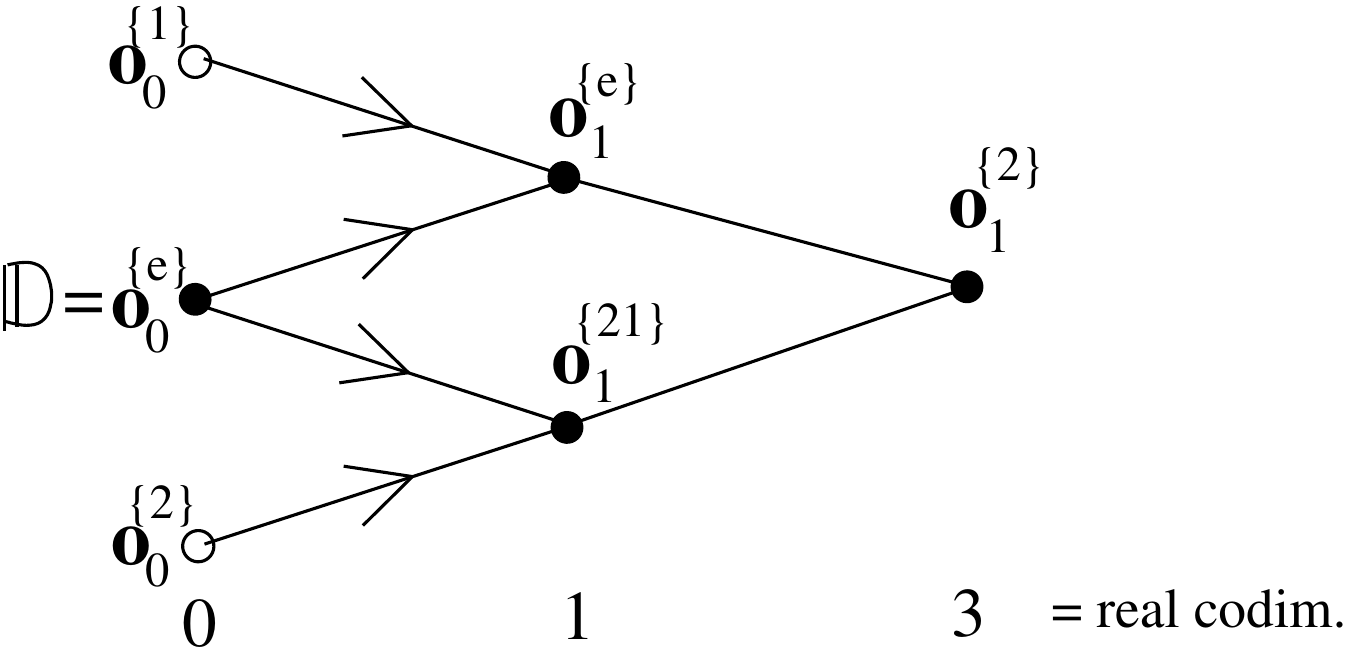}\]The
right-hand inclusions are, as the diagram indicates, given by cross-actions.

Turning to the classical domain $\mathbb{B}$, we have (from $H_{0}$)
only 2 open orbits:%
\footnote{The ``photographic negatives'' of these pictures do not appear because
they cannot be obtained from the first picture via $W_{\CC}$.%
} \[\includegraphics[scale=0.55]{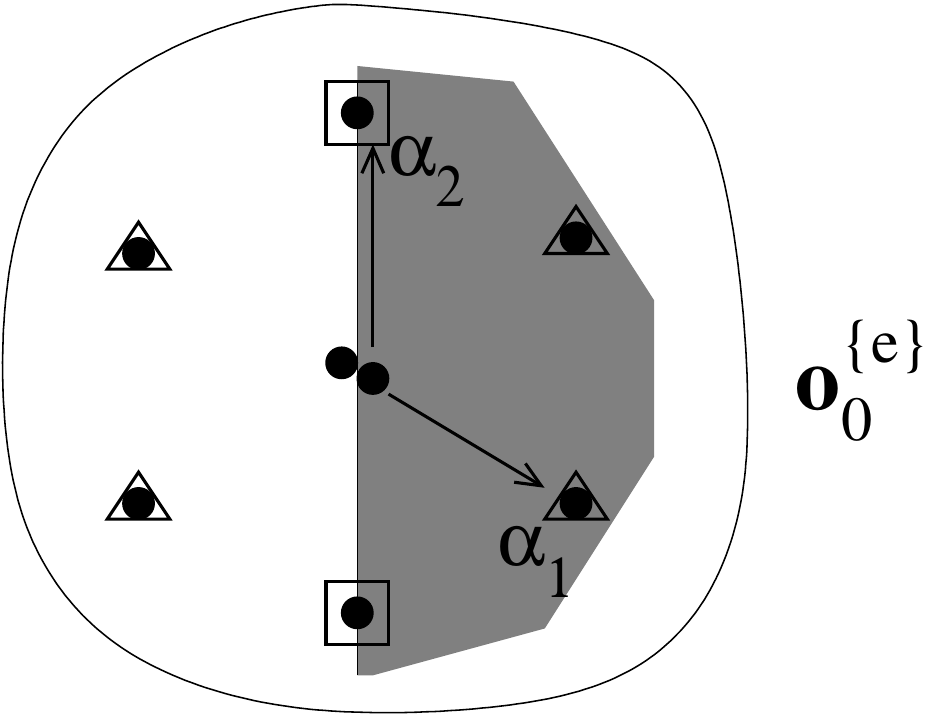}\]\[\includegraphics[scale=0.55]{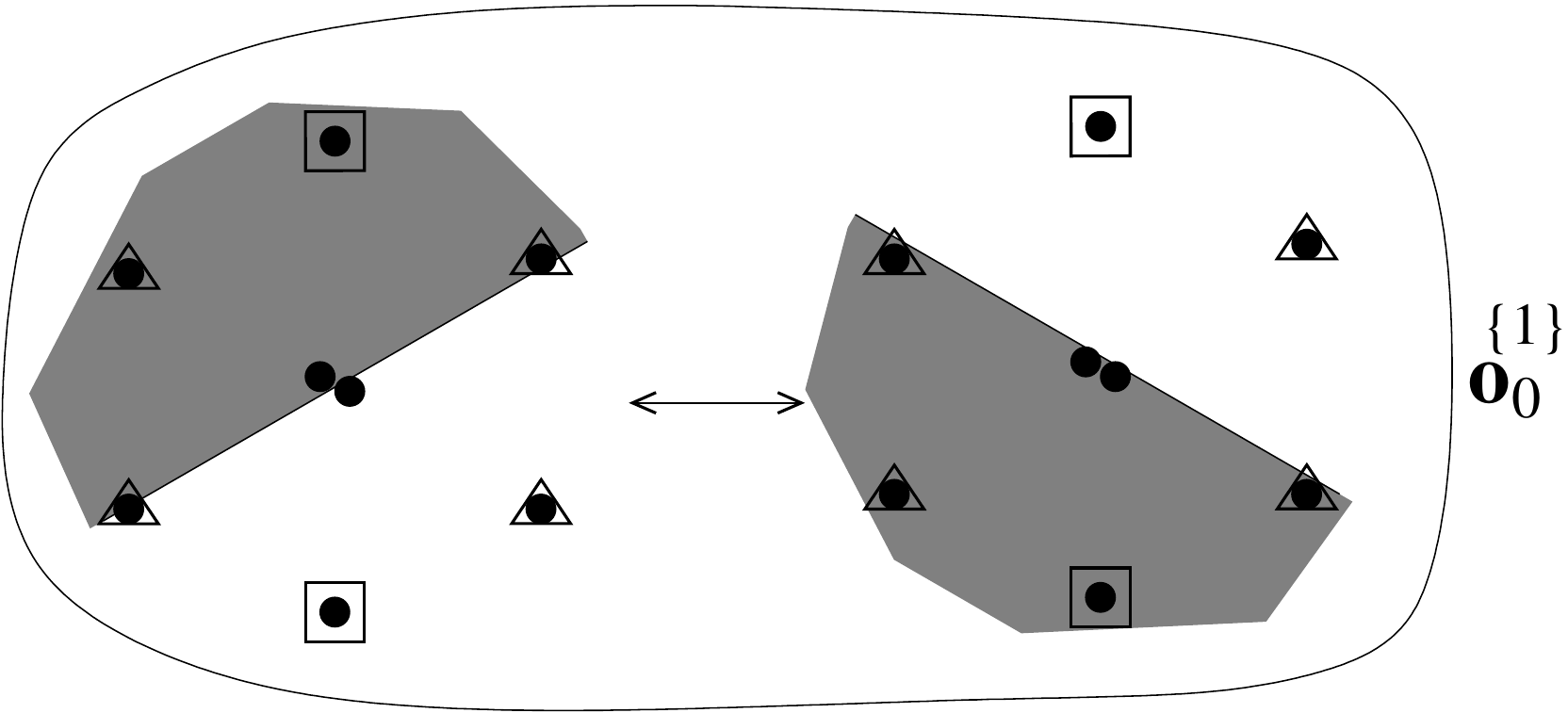}\]From
$H_{1}$, we get \[\includegraphics[scale=0.5]{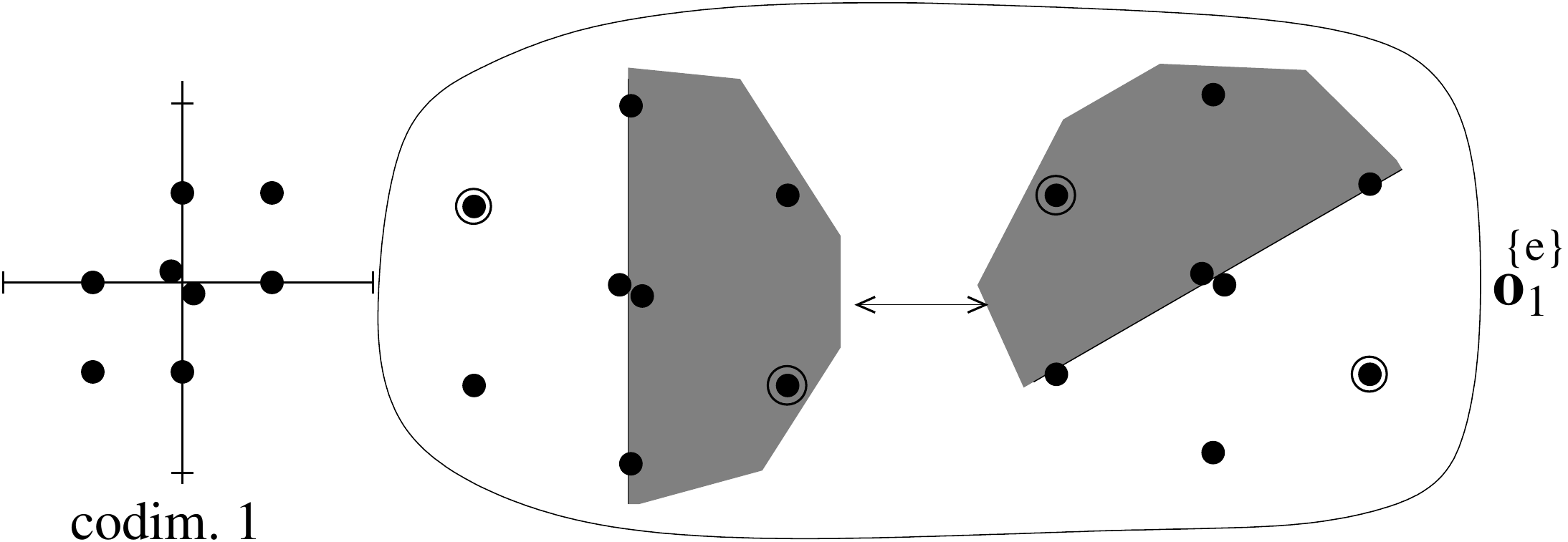}\]\[\includegraphics[scale=0.55]{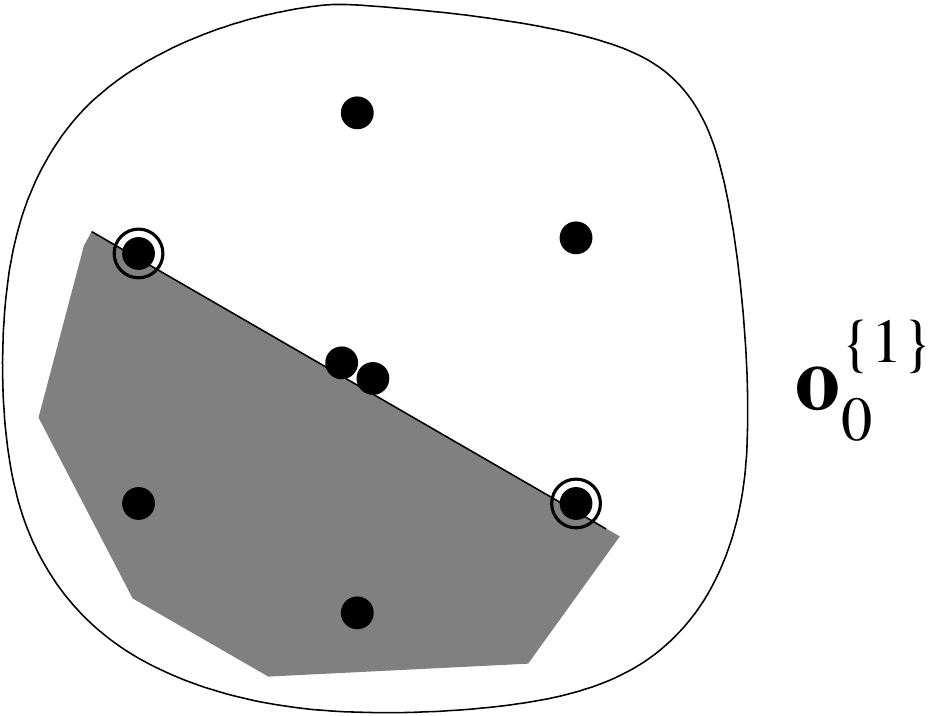}\]where
the last $W_{\RR}$-orbit demonstrates that the orbit map $\mathbf{o}$
need not be one-to-one in the non-complete-flag case. The incidence
diagram is the same as for $PGL_{2}$, since $\mathbf{o}_{0}^{\{e\}}=\mathbb{B}$,
$\mathbf{o}_{0}^{\{1\}}=\PP(V_{+})\backslash cl(\mathbb{B})$, and
$\mathbf{o}_{1}^{\{e\}}=\partial\mathbb{B}\cong S^{3}$.

\subsubsection{$G=PSp_{4}$}

The Hasse diagram for the Cartan subgroups of $G(\RR)^{\circ}$ is
\[\includegraphics[scale=0.6]{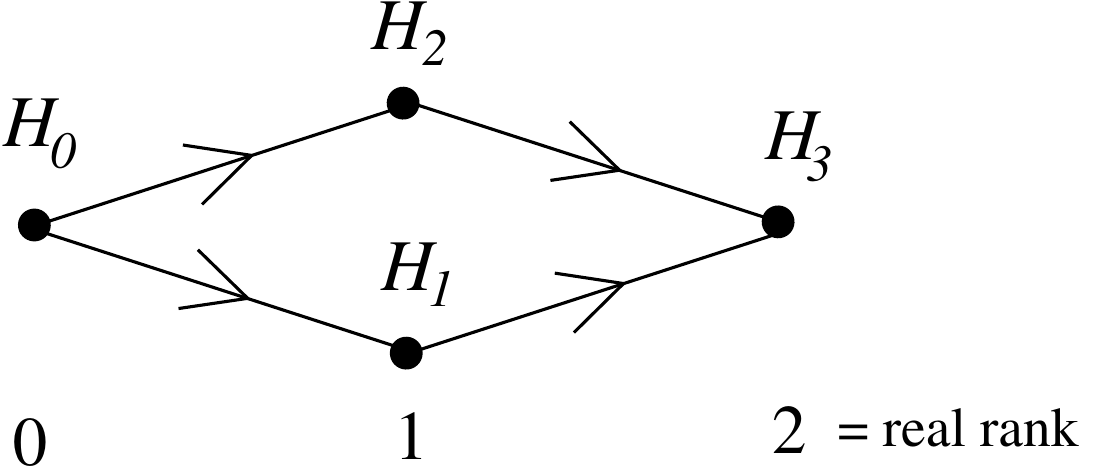}\]where $H_{1}$ {[}resp. $H_{2}${]}
is obtained from $H_{0}$ by the Cayley transform ub a long {[}resp.
short{]} root. We shall only consider the two cases%
\footnote{There is also a weight $2$ case with numbers $(1,2,1)$, left to
the reader.%
} where $D$ is the period domain for rank $4$ Hodge structures of
weight $1$ resp. 3 with Hodge numbers $(2,2)$ resp. $(1,1,1,1)$
(inducing weight zero HS of type $(3,4,3)$ resp. $(1,1,2,2,2,1,1)$
on $\g$). The pictures corresponding to $(H_{0},\chi_{0})$ are%
\footnote{with apologies to the reader for taking $\alpha_{1}$ to be the short
root (meaning that the Cayley transform in $\alpha_{1}$ gives $H_{2}$
and vice versa).%
} \[\includegraphics[scale=0.55]{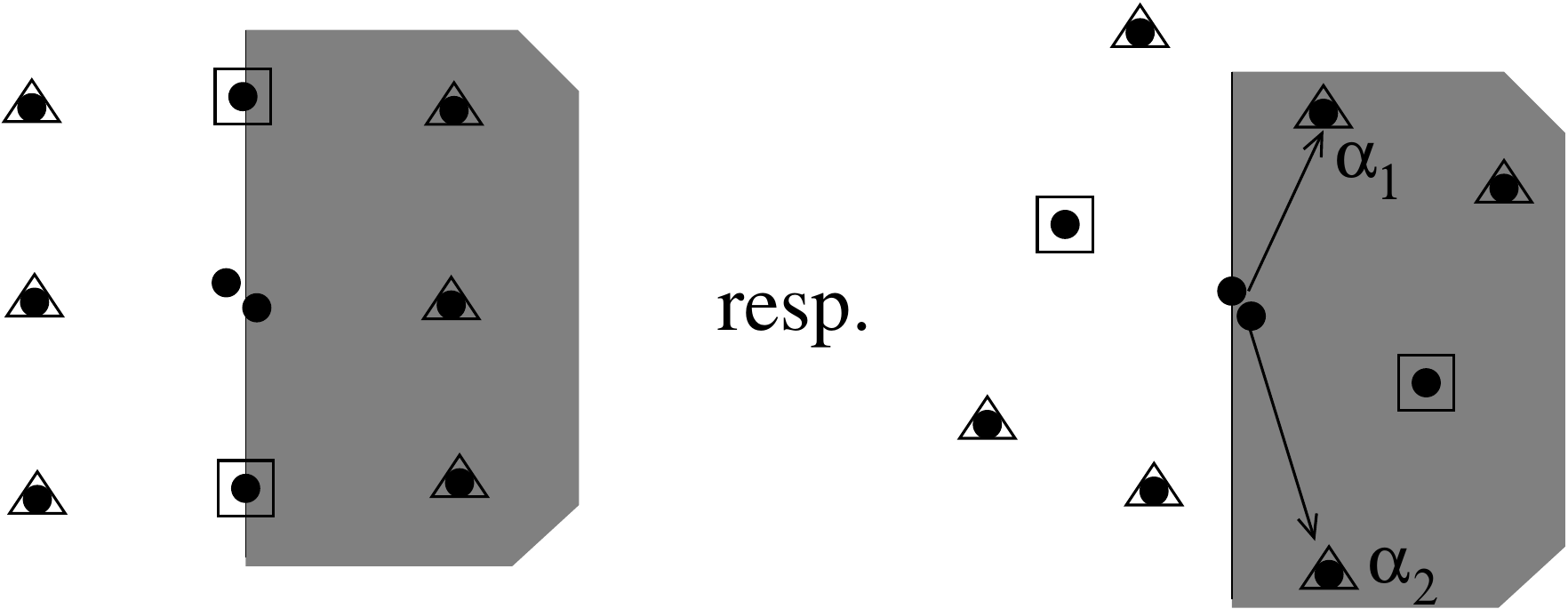}\]and the enhanced Hasse
diagram in the first (Siegel $\mathfrak{H}_{2}$) case is, up to labeling,
the same as for the Carayol domain. The second (complete flag) case
$D=D_{(1,1,1,1)}$ has enhanced Hasse diagram \[\includegraphics[scale=0.6]{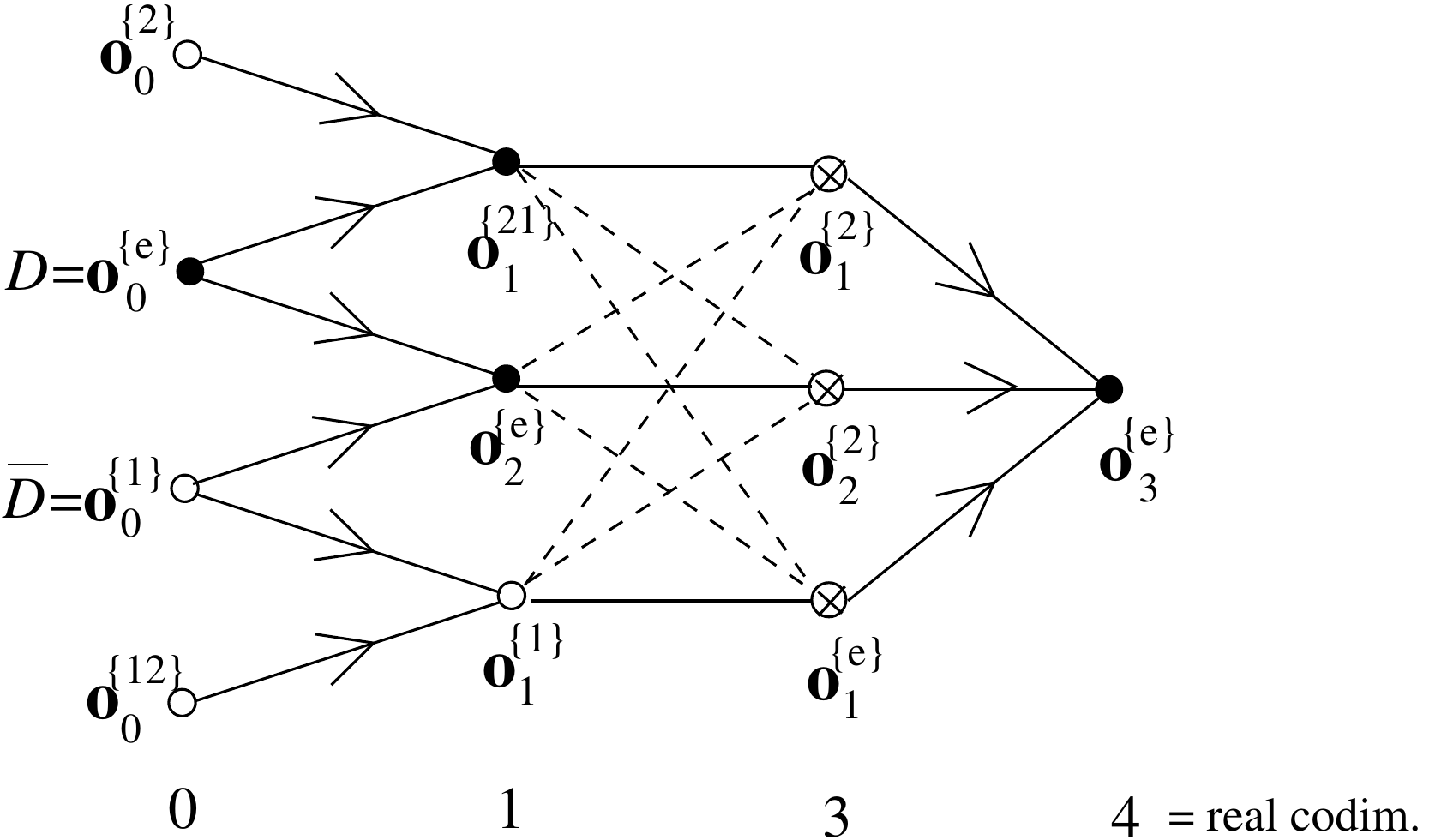}\]in
which we obtain the first examples of non-polarizable boundary strata.
Of the mixed Hodge diagrams \[\includegraphics[scale=0.55]{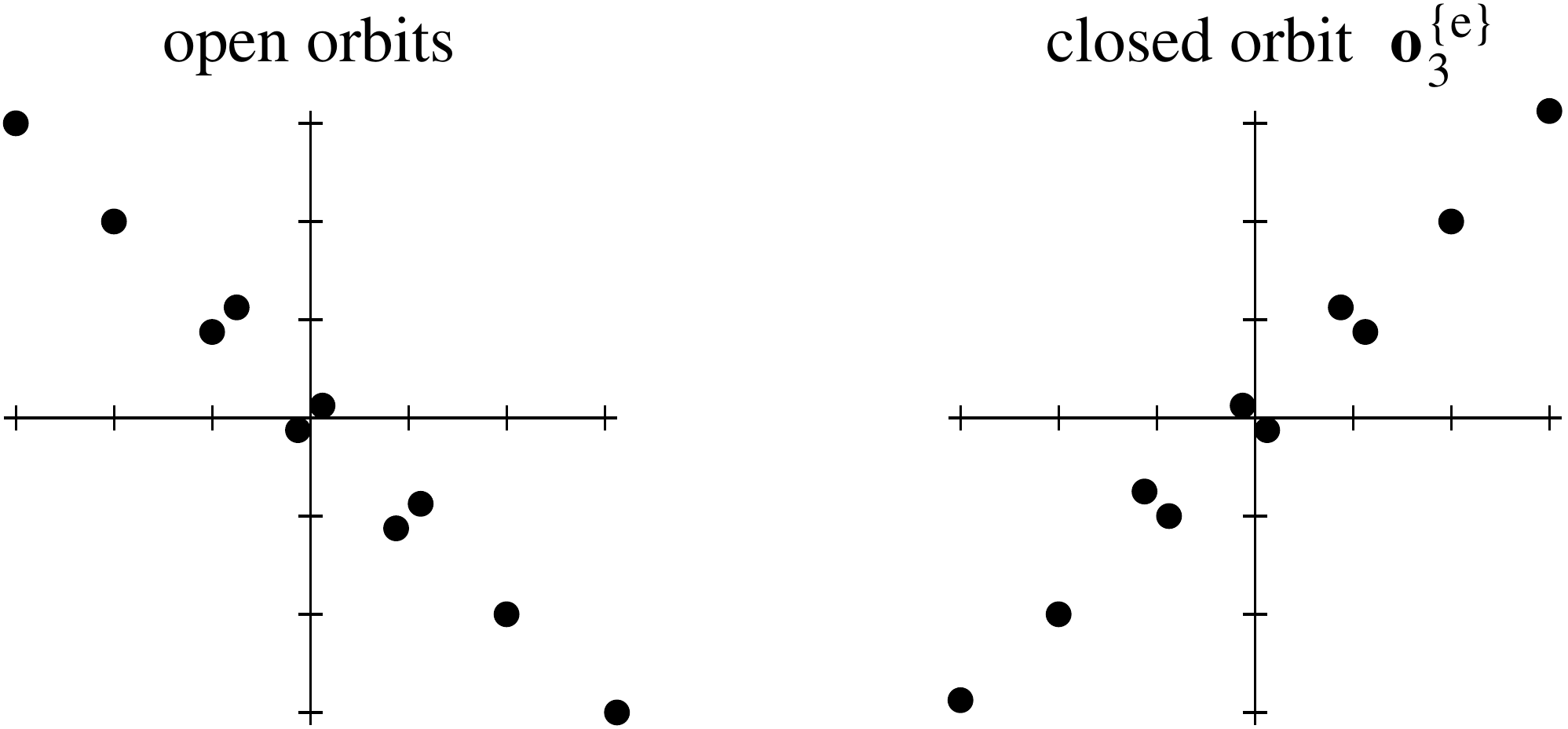}\]\[\includegraphics[scale=0.55]{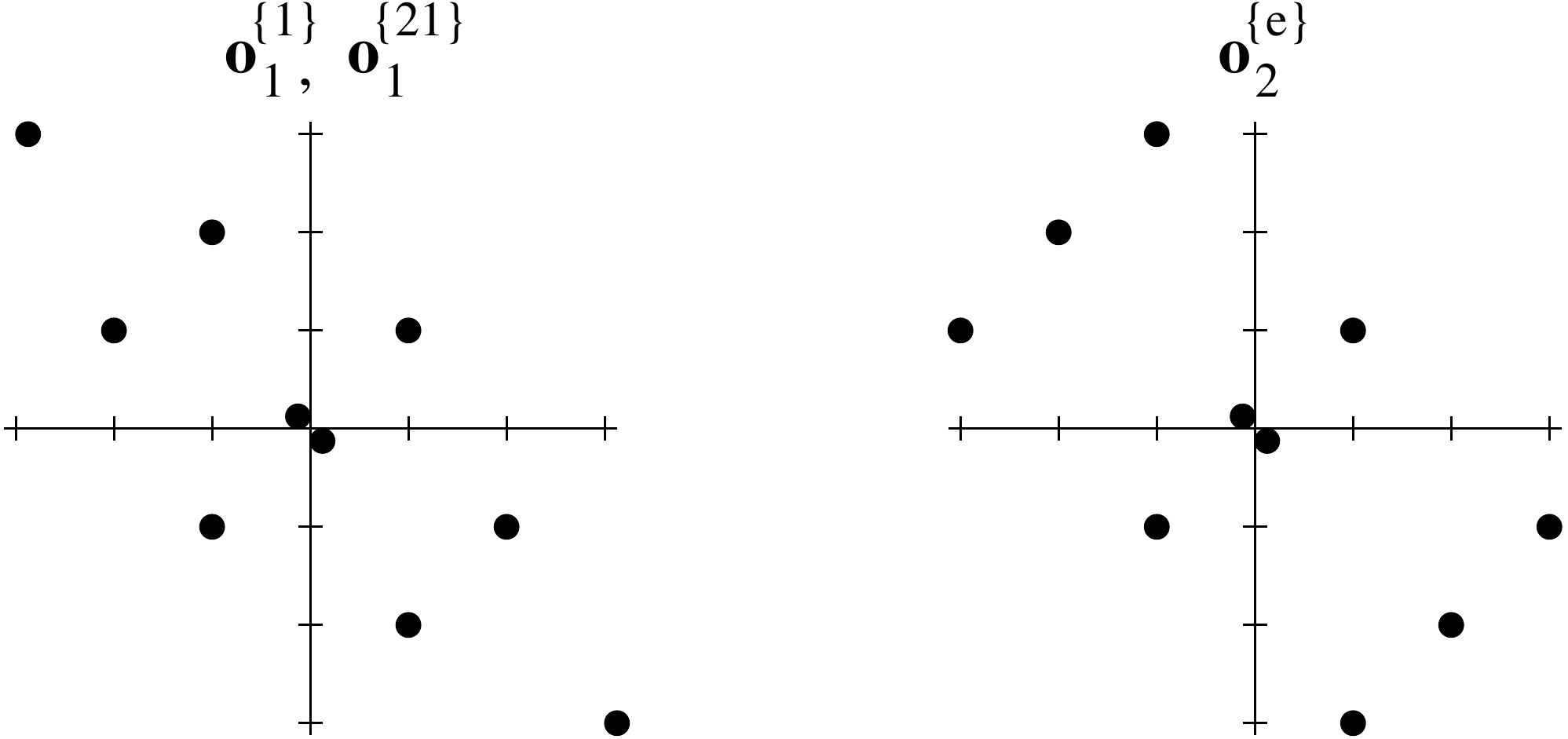}\]\[\includegraphics[scale=0.55]{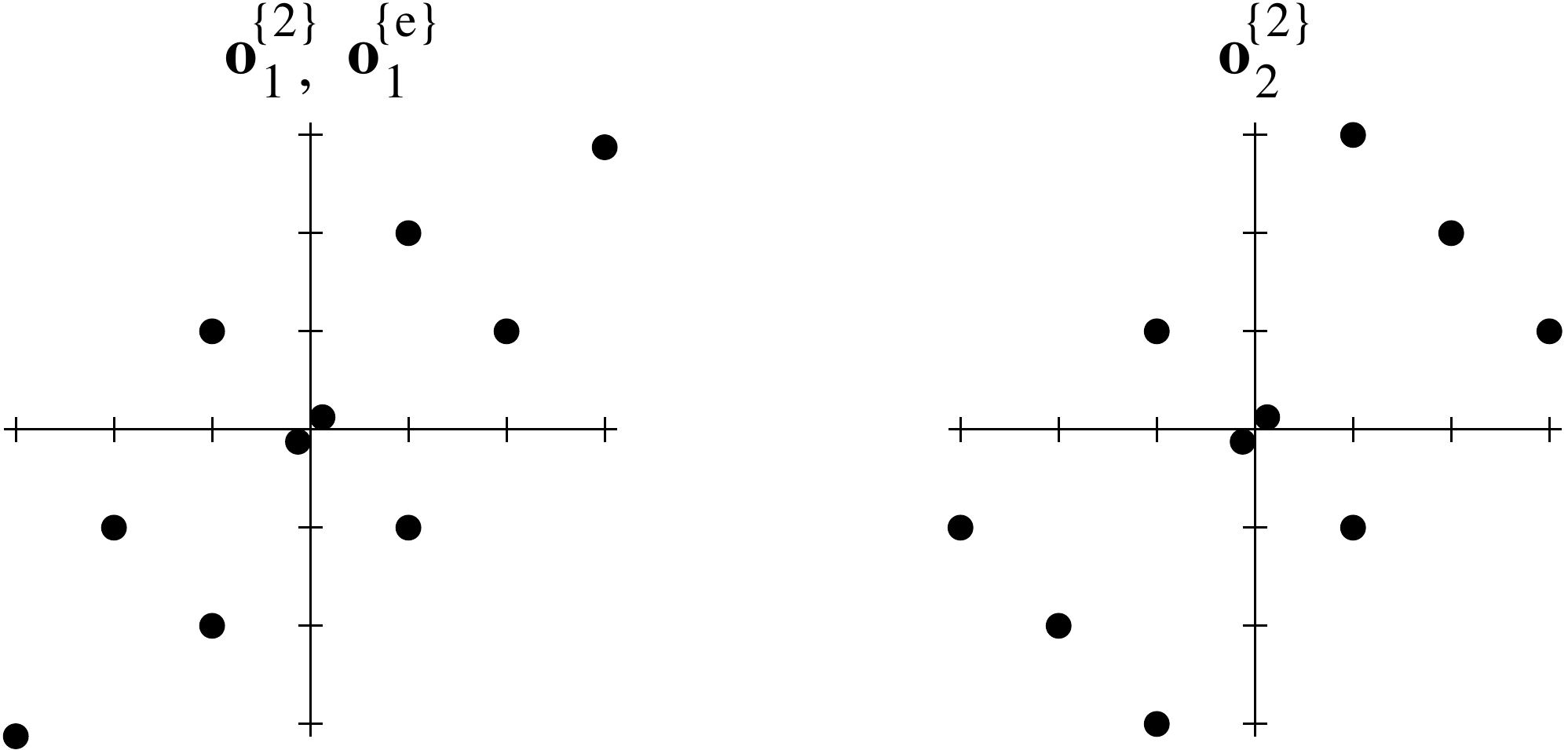}\]the
first four correspond to boundary components in \cite{KP}, whereas
the bottom two have $\dim(\g^{-1,-1})=0$ making the codimension-three
substrata obviously non-polarizable.

\subsubsection{$G=$ split $G_{2}$}

Note that $G(\RR)=G(\RR)^{\circ}$. There are three cases, corresponding
to polarized HS's with Hodge numbers (A) $(2,3,2)$, (B) $(1,2,1,2,1)$,
resp. (C) $(1,1,1,1,1,1,1)$ on the standard (7-dimensional) representation
and (A) $(1,4,4,4,1)$, (B) $(2,1,2,4,2,1,2)$, resp. (C) $(1,1,1,1,2,2,2,1,1,1,1)$
on $\g$ (see \cite[Chap. 4]{GGK}). The Cartan diagram is as for
$PSp_{4}$,%
\footnote{same meaning for $H_{1}$ vs. $H_{2}$, and the same absurd convention
on $\alpha_{1}$ vs. $\alpha_{2}$%
} and the $(H_{0},\chi_{0})$ pictures are \[\includegraphics[scale=0.45]{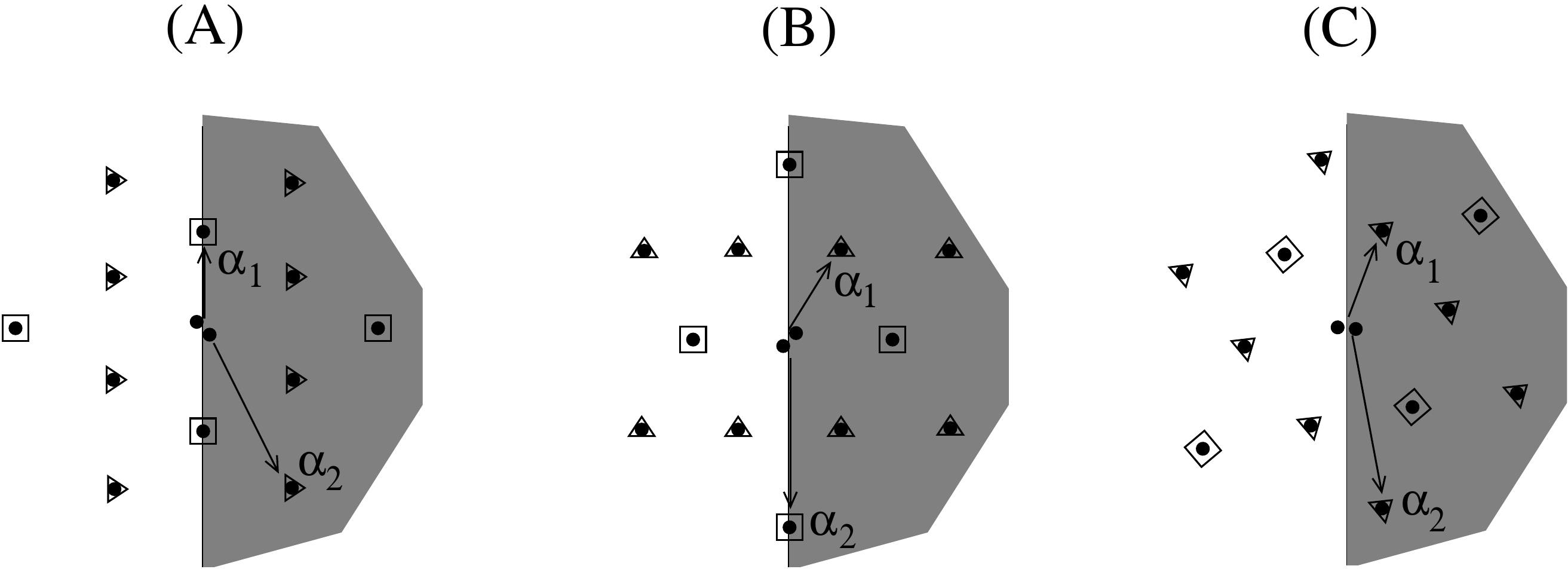}\](Note
that $W_{\RR}^{\circ}(H_{j})$ is generated by reflections in $\Delta_{c}$
resp. $\Delta_{\RR}$ for $j=0$ resp. 3, and by the reflection in
$\Delta_{\RR}$ and $(-\text{id})$ for $j=1$ and $2$.) The complete
flag case is (C), with enhanced Hasse diagram \[\includegraphics[scale=0.6]{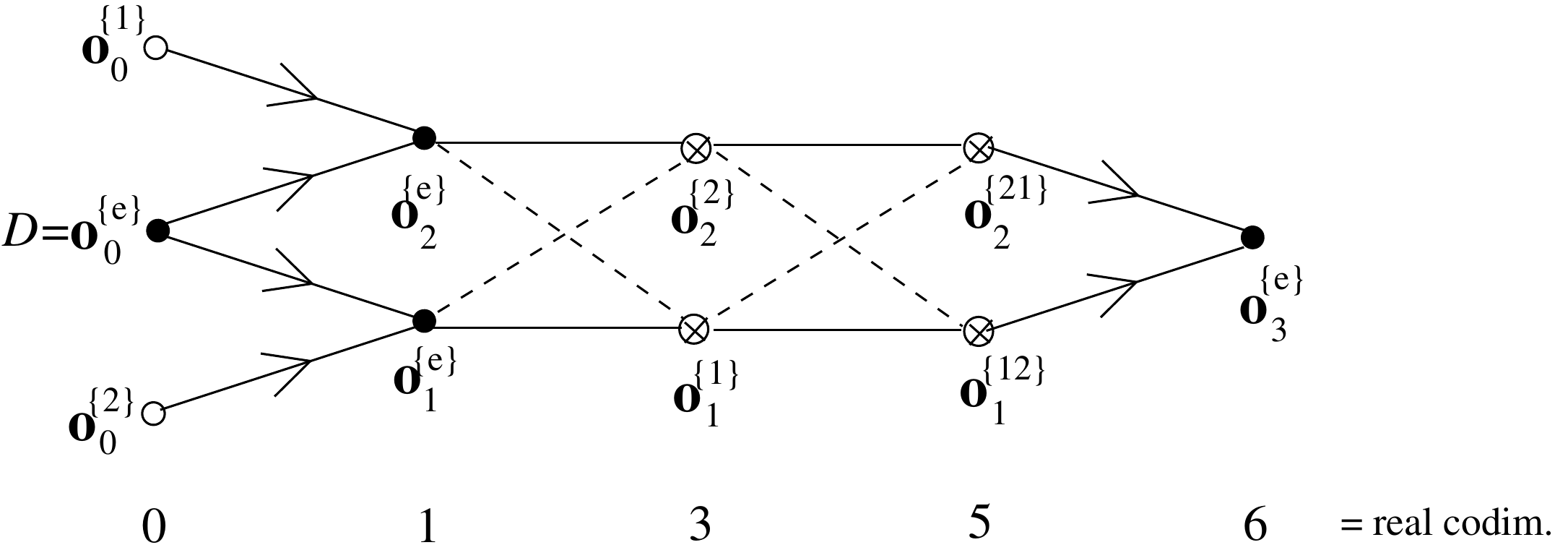}\]We
omit the mixed Hodge diagrams, which are unwieldy, but include them
for (B) \[\includegraphics[scale=0.75]{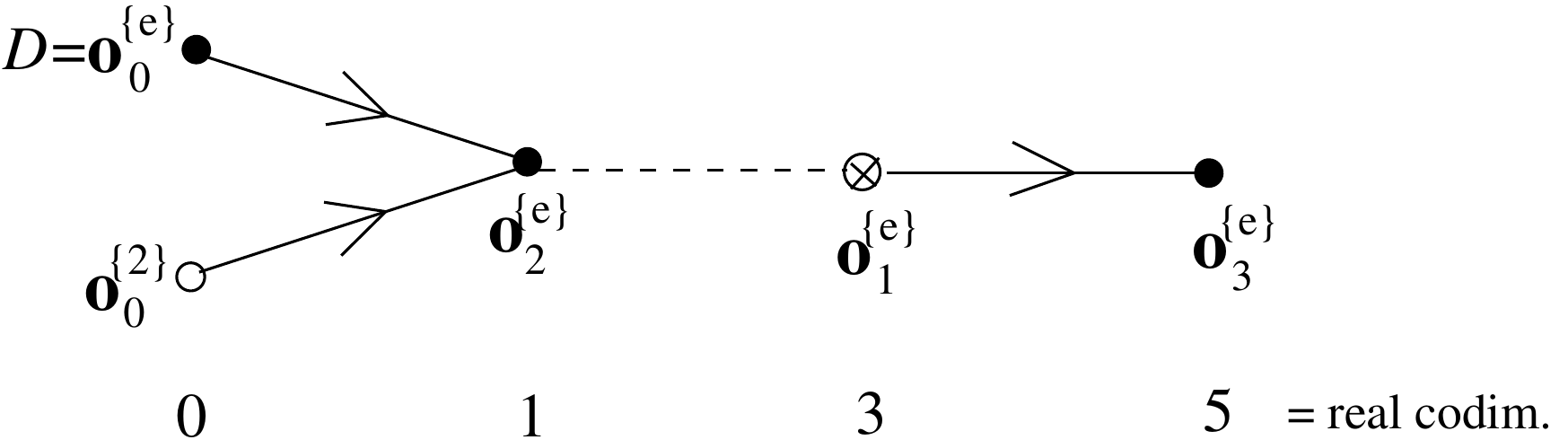}\]\[\includegraphics[scale=0.4]{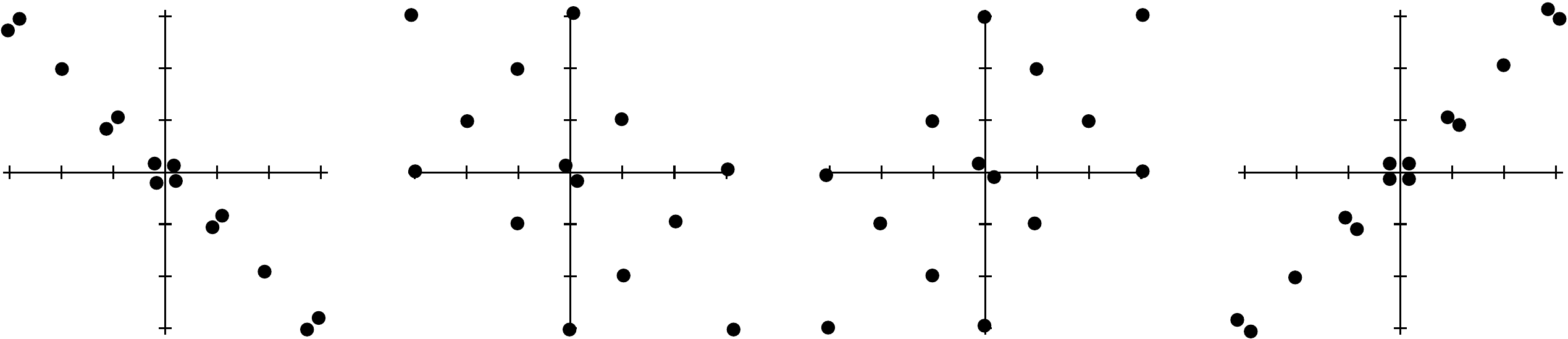}\]and
for (A) \[\includegraphics[scale=0.75]{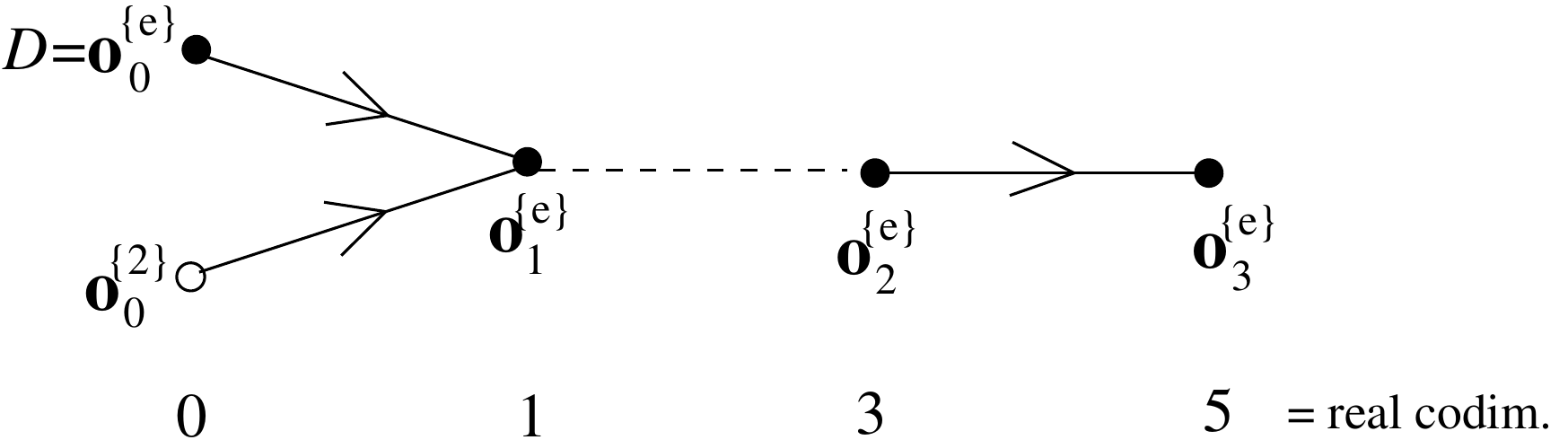}\]\[\includegraphics[scale=0.5]{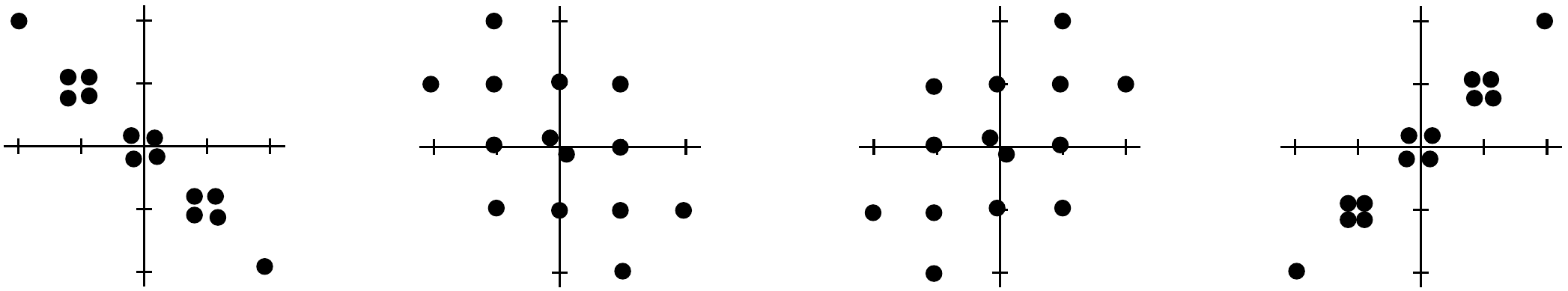}\]Except
for (B) $\mathbf{o}_{1}^{\{e\}}$, one can check that these orbits
are all in the image of a $B(N)$ (see \cite[sec. 8]{KP} for (A)).

\subsection{Counterexamples}

The 3 vignettes with which we conclude this paper illustrate, for
polarizable boundary strata, the potential failure of cuspidality,
of rationality, and of a stronger notion of polarizability.

\subsubsection{$\mathfrak{g}^{-1,-1}$ need not contain a real root}

To see this, we have to consider strata of codimension strictly larger
than $1$. Let $D$ be Carayol's nonclassical domain (cf. $\S6.1$),
and $\mathcal{O}_{\mathrm{c}}\subset\check{D}$ the (polarizable)
closed orbit. The Cartan $H$ determined by $F^{\b}\in\mathcal{O}_{\mathrm{c}}$
has real rank $1$, which gives a root diagram\[\includegraphics[scale=0.6]{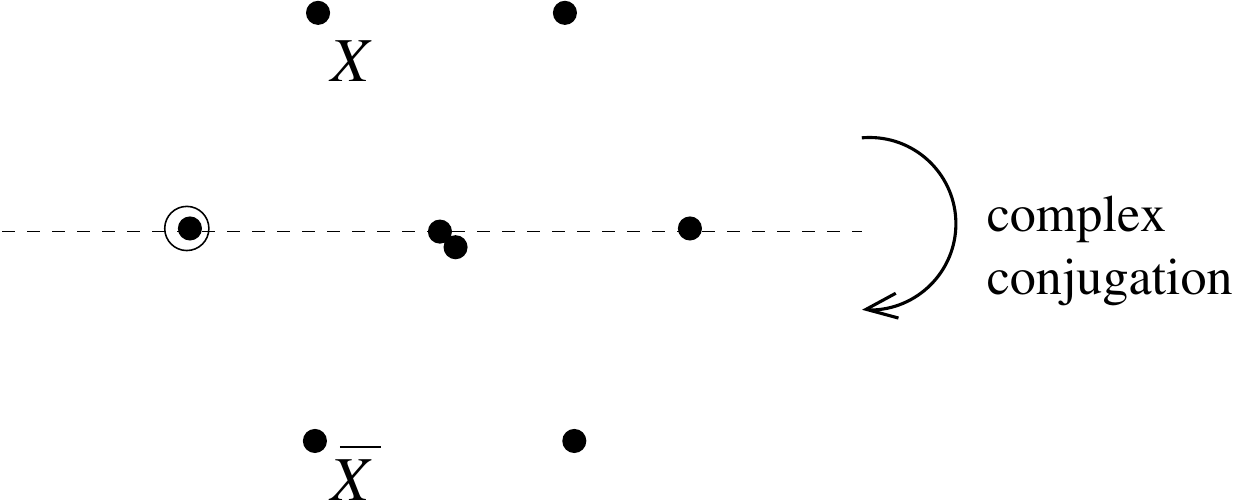}\]The
associated bigrading is\[\includegraphics[scale=0.6]{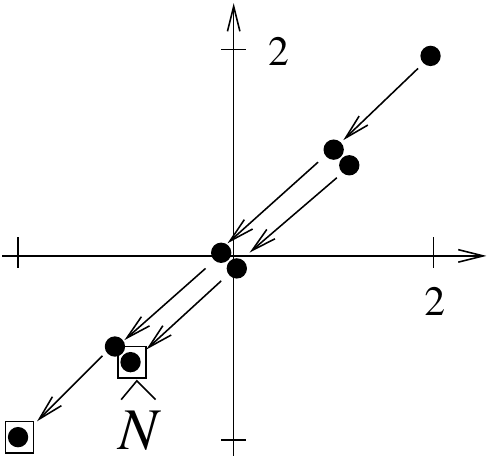}\]where
$\hat{N}\in\mathfrak{g}_{\RR}^{-1.-1}$ is as in Definition \ref{defn polarizable}
and arrows denote the action of $\text{ad}N$. Clearly $\mathfrak{g}^{-1,-1}=\CC\langle X,\bar{X}\rangle$
and $\hat{N}$ is a multiple of $X+\bar{X}$.

\subsubsection{A noncuspidal boundary component}

This time, in addition to codimension $>1$, we have to start with
a Mumford-Tate group of rank at least $3$. Taking $G=Sp_{6}$, we
consider the Siegel domain $D=\mathfrak{H}_{3}$ parametrizing Hodge
structures $\vf$ of type $(3,3)$. The associated projection $\pi_{\chi}$
on roots takes the form\[\includegraphics[scale=0.55]{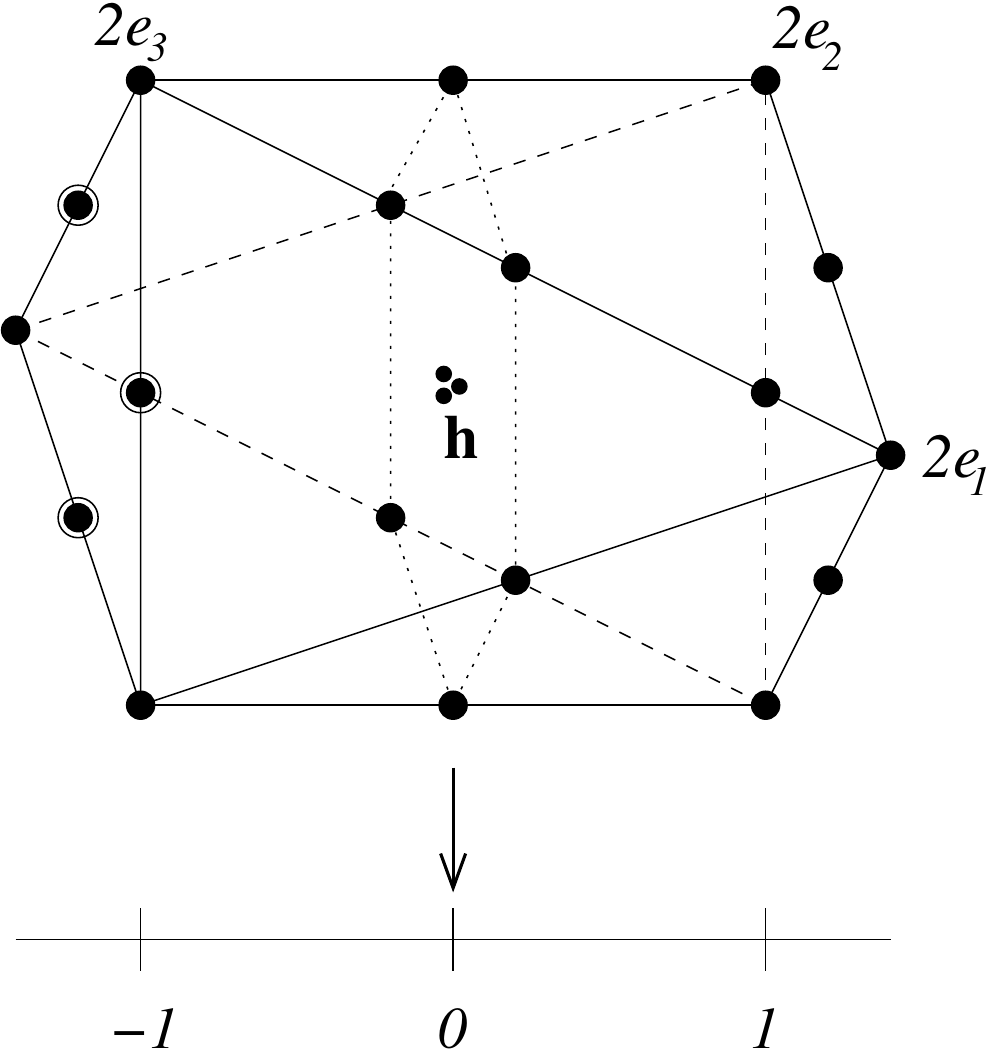}\]sending
$ae_{1}+be_{2}+ce_{3}\mapsto\frac{1}{2}(a+b-c)$. 

On the other hand, with $H$ $\mathbb{R}$-split, the same picture
describes $\pi_{\chi_{\tilde{Y}}}$ for the (Hodge-Tate) limiting
mixed Hodge structures parametrized by $B_{\RR}(N)$, where $N$ is
the sum of root vectors for the $3$ circled roots ($e_{3}-e_{1}$,
$e_{3}-e_{2}$, $-e_{1}-e_{2}$). The subalgebra $\ker(\text{ad}\tilde{Y})\cong Gr_{0}^{\tilde{W}}\mathfrak{g}$,
whose root system maps to $0$ in the picture, is $\mathfrak{sl}_{3}$.
Since $SL_{3}$ has no compact Cartan, $\mathfrak{q}=\tilde{W}_{0}\mathfrak{g}$
is not a cuspidal parabolic subgroup. The boundary component $B(N)$
maps to the closed orbit $\mathcal{O}_{\mathrm{c}}\subset\check{D}$,
and so $\mathcal{O}_{\mathrm{c}}$ is not cuspidal.

\subsubsection{A non-rational $G(\RR)^{\circ}$-orbit in $\partial D$}

Take $D$ once more to be Carayol's $(1,2,2,1)$-domain, but this
time with $G$ a $\QQ$-anisotropic form of $SU(2,1).$ More precisely,
let $V_{+}$ be a 3-dimensional vector space over $\QQ(i)$, $H$
a $\QQ(i)$-Hermitian form of signature $(2,1)$ on $V_{+}$ hat does
not represent zero, and $Q$ the alternating form on $V:=Res_{\QQ(i)/\QQ}(V_{+})$
given by minus the imaginary part of $H$. So
\[
G=Res_{\QQ(i)/\QQ}\left(Aut(V_{+},H)\right)=Res_{\QQ(i)/\QQ}\left(GL(V_{+})\right)\cap Sp(V,Q),
\]
and extending $Q$ to $V_{\QQ(i)}=V_{+}\oplus\overline{V_{+}}$, we
have $H(v,v)=-2iQ(v,\bar{v})$.

Now, we know that the resulting domain $D$ has polarizable boundary
strata, which then come from $\RR$-limit mixed Hodge structures.
But if the stratum is rational in our sense, then $\tilde{W}_{\b}=W(N)_{\b}$
can be defined over $\QQ$ (with weights on $V$ centered about $3$,
not $0$). Taking any nonzero rational vector $w\in W_{2}V$, we have
$w=v+\bar{v}$ for $(0\neq)\, v\in W_{2}V_{+}$, and $0=Q\left(W_{2}W_{\QQ(i)},W_{2}V_{\QQ(i)}\right)$
forces
\[
0=-2iQ(v,\bar{v})=H(v,v),
\]
in contradiction to anisotropy.

\curraddr{\noun{${}$}\\
\noun{Department of Mathematics, Campus Box 1146}\\
\noun{Washington University in St. Louis}\\
\noun{St. Louis, MO} \noun{63130, USA}}

\email{\emph{${}$}\\
\emph{e-mail}: matkerr@math.wustl.edu}

\curraddr{${}$\\
\noun{Mathematics Department, Mail stop 3368}\\
\noun{Texas A\&M University}\\
\noun{College Station, TX 77843, USA}}

\email{${}$\\
\emph{e-mail}: gpearl@math.msu.edu}

\begin{thebibliography}{AMRT}
\bibitem[Ad]{Ad}J. Adams, \emph{Guide to the Atlas Software: Computational
Representation Theory of Real Reductive Groups}, in ``Representation
Theory of Real Reductive Groups (Snowbird, July 2006)'', Contemp.
Math., AMS, Providence, 2008.

\bibitem[AdC]{AdC}J. Adams and F. de Cloux, \emph{Algorithms for
representation theory of real reductive groups}, J. Inst. Math. Jussieu
8 (2009), no. 2, 209\textendash{}259.

\bibitem[AMRT]{AMRT}A. Ash, D. Mumford, M. Rapoport, and Y. Tai,
``Smooth compactification of locally symmetric varieties'', Math.
Sci. Press, Brookline, Mass., 1975.

\bibitem[Bo]{Bo}A. Borel, ``Linear Algebraic Groups (2nd Ed.)'',
Graduate Texts in Math. 126, Springer-Verlag, New York, 1991.

\bibitem[BT]{BT}A. Borel and J. Tits, \emph{Groupes reductifs}, Publ.
Math. IHES 27 (1965), 55-150.

\bibitem[Car]{Car}H. Carayol, \emph{Cohomologie automorphe et compactifications
partielles de certaines vari\'et\'es de Griffiths-Schmid}, Compos.
Math. 141 (2005), 1081-1102.

\bibitem[Cat]{Ca}E. Cattani, \emph{Mixed Hodge structures, compactifications
and monodromy weight filtration}, in Annals of Math Study 106, Princeton
University Press, Princeton, NJ, 1984, pp. 75-100.

\bibitem[CK]{CK}E. Cattani and A. Kaplan, \emph{Polarized mixed Hodge
structures and the local monodromy of a VHS}, Invent. Math. 67 (1982),
101-115.

\bibitem[CKS]{CKS}E. Cattani, A. Kaplan, and W. Schmid, \emph{Degeneration
of Hodge structure}, Ann. Math. 123 (1986), no. 3, 457-535.

\bibitem[FHW]{FHW} G. Fels, A. Huckleberry, and J. Wolf, ``Cycle
spaces of flag domains: a complex geometric viewpoint'', Progress
in Math. v. 245, Birkhauser, Boston, 2006.

\bibitem[GGK1]{GGK}M. Green, P. Griffiths and M. Kerr, ``Mumford-Tate
groups and domains: their geometry and arithmetic'', Ann. Math. Stud.

\bibitem[GGK2]{GGK2}---------, ``Hodge Theory, Complex Geometry
and Representation Theory'', CBMS Regional Conference Series in Mathematics,
to appear.

\bibitem[KU]{KU}K. Kato and S. Usui, ``Classifying spaces of degenerating
polarized Hodge structure'', Ann. Math. Stud. 169, Princeton Univ.
Press, Princeton, NJ, 2009.

\bibitem[Ke]{Ke}M. Kerr, \emph{Shimura varieties: a Hodge-theoretic
perspective}, preprint, 2010, available at http://www.math.wustl.edu/\textasciitilde{}matkerr/SV.pdf

\bibitem[KP]{KP}M. Kerr and G. Pearlstein,\emph{ Boundary components
of Mumford-Tate domains}, preprint, arXiv:1210.5301.

\bibitem[Kn]{Kn}A. Knapp, ``Lie groups: beyond an introduction (2nd
Ed.)'', Progress in Math. 140, Birkhauser, Boston, MA, 2002.

\bibitem[Ma1]{Ma1}T. Matsuki, \emph{Closure relations for orbits
on affine symmetric spaces under the action of minimal parabolic subgroups},
In: Representation of Lie Groups, Kyoto, Hiroshima, 1986, (eds. K.
Okamoto and T. Oshima), Adv. Stud. Pure Math., 14, Kinokuniya Company
LTD., Tokyo, 1988, pp. 541\textendash{}559.

\bibitem[Ma2]{Ma2}---------, \emph{Closure relations for orbits on
affine symmetric spaces under the action of parabolic subgroups. Intersections
of associated orbits}, Hiroshima Math. J. 18 (1988), 59-67.

\bibitem[RS]{RS}R. Richardson and T. Springer, \emph{The Bruhat order
on symmetric varieties}, Geometriae Dedicata 35 (1990), 389-436.

\bibitem[Ro]{Ro}C. Robles, \emph{Schubert varieties as variations
of Hodge structure}, arXiv:1301.5606, to appear in Selecta Math.

\bibitem[Sc]{Sc}W. Schmid, \emph{Variation of Hodge structure: the
singularities of the period mapping}, Invent. Math. 22 (1973), 211-319.

\bibitem[Tr]{Tr}P. Trapa, \emph{Computing real Weyl groups}, notes
from a 2006 lecture by D. Vogan, available at http://www.math.utah.edu/\textasciitilde{}ptrapa/preprints.html

\bibitem[Wo]{Wo}J. Wolf, \emph{The action of a real semisimple Lie
group on a complex manifold I: Orbit structure and holomorphic arc
components}, Bull. Amer. Math. Soc. 75 (1969), 1121-1237.

\bibitem[Ye]{Ye}W.-L. Yee, \emph{Simplifying and unifying Bruhat
order for $B\backslash G/B$, $P\backslash G/B$, $K\backslash G/B$,
and $K\backslash G/P$}, preprint, arXiv:1107.0518v3

\end{thebibliography}
\end{document}